\documentclass{amsart}
\usepackage{amsfonts, amsmath,amscd, amssymb, latexsym, mathrsfs, stmaryrd, verbatim, wasysym, amsrefs, tabu}
\usepackage{enumerate}
\usepackage{graphicx}
\usepackage[all]{xy}
\usepackage{float}
\usepackage{hyperref}
\usepackage[capitalize,nameinlink]{cleveref}
\usepackage{tikz-cd}

\newtheorem{theorem}{Theorem}

\newtheorem{definition}[theorem]{Definition}

\newtheorem{corollary}[theorem]{Corollary}

\newtheorem{lemma}[theorem]{Lemma}
\newtheorem{proposition}[theorem]{Proposition}
\newtheorem{remark}[theorem]{Remark}

\numberwithin{equation}{section}
\numberwithin{theorem}{section}

 \usepackage{color}
\definecolor{darkgreen}{cmyk}{1,0,1,.2}
\definecolor{m}{rgb}{1,0.1,1}

\DeclareMathOperator{\supp}{supp}   
\DeclareMathOperator{\vol}{vol}    

 \DeclareMathOperator{\tr}{tr}

\DeclareMathOperator{\ind}{ind}
\DeclareMathOperator{\Ind}{Ind}
\DeclareMathOperator{\sgn}{sgn}

\newcommand{\forget}[1]{}

\def  \nuint {\raise10pt\hbox{$\nu$}\kern-6pt\int}

\newcommand\Tr{\operatorname{Tr}}

\def \Sp {{\cal S}}

\newcommand\Di{D\kern-6pt/}
\newcommand\cDi{{\mathcal D}\kern-6pt/}
\newcommand\spi{S\kern-6pt/}
\newcommand \cspi{\Sp\kern-6pt/}

\newcommand\CC{\mathbb C}

\def \cal {\mathcal}

\newcommand\RR{\mathbb R}

\newcommand\pa{\partial}
\newcommand\Ker{\operatorname{Ker}}

{\catcode`@=11\global\let\c@equation=\c@theorem}

\allowdisplaybreaks[2]
\date{}
 \usepackage{color}
\definecolor{darkgreen}{cmyk}{1,0,1,.2}
\definecolor{m}{rgb}{1,0.1,1}

\title{Higher Lefschetz formulas on $\Gamma$-proper manifolds}
\author{Paolo Piazza}
\address{Dipartimento di Matematica, Sapienza Universit\`a di Roma}
\email{piazza@mat.uniroma1.it}
\author{Hessel B. Posthuma}
\address{University of Amsterdam }
\email{H.B.Posthuma@uva.nl}
\author{Yanli Song}
\address{Department of Mathematics and Statistics, Washington University, St. Louis, MO, 63130, U.S.A.} 
\email{yanlisong@wustl.edu}
\author{Xiang Tang}
\address{Department of Mathematics and Statistics, Washington University, St. Louis, MO, 63130, U.S.A.}
\email{xtang@math.wustl.edu}



\begin{document}

\maketitle

\begin{abstract}
Let $\Gamma$ be a finitely generated discrete group acting properly and cocompactly on a smooth manifold $M$. By employing heat-kernel techniques we 
prove a geometric formula for the  pairing of the index class associated to a $\Gamma$-equivariant Dirac operator $D$ with a delocalized cyclic cocycles $\tau$ in 
$HP^* (\mathbb{C}\Gamma,\langle \gamma \rangle)$. Our formula takes place on the fixed point manifold $M^\gamma$ and should be  regarded as a higher Lefschetz formula
for $D$. The formula  involves the Atiyah-Segal-Singer form and an explicit $Z_\gamma$-invariant form on $M^\gamma$ that is naturally associated to $\tau\in HP^* (\mathbb{C}\Gamma,\langle \gamma \rangle)$.
		\end{abstract}

\tableofcontents
\section{Introduction}
This article is devoted to the formulation and proof of new results in equivariant index theory 
in a non-compact setting. In order to place our results in the right context we begin by 
 reviewing the classical theory, where we have an even dimensional  compact manifold $M$, a compact
Lie group $G$ acting  by diffeomorphisms on $M$ and a $\mathbb{Z}_2$-graded Dirac operator $D$ on $M$ 
acting on the sections of a $G$-bundle of Clifford modules $E$ on $M$ and commuting with the action of $G$.

We have, first of all, a numeric index $\ind (D)\in\mathbb{Z}$ and the Atiyah-Singer index formula,
expressing this integer in terms of characteristic classes attached to the manifold $M$ and the bundle of Clifford modules $E$. We also have a $G$-index, obtained by considering the  $G$-representations
 $\mathrm{Ker}(D^\pm)$ and 
\begin{equation}\label{G-index}
\ind_G (D)= [\mathrm{Ker}(D^+)]-[\mathrm{Ker}(D^-)]\in R(G).
\end{equation}
Given $g\in G$ we can thus consider the associated character, that is
\begin{equation}\label{g-index}\ind_G (D)(g)=\Tr(g|_{\mathrm{Ker}(D^+)})- \Tr(g|_{\mathrm{Ker}(D^-)})\,.
\end{equation}
Notice that $\ind (D)=\ind_G (D) (e)\,.$
The equivariant index theorem of Atiyah-Segal-Singer gives a formula for 
$\ind_G (D)(g)$ and this formula depends on characteristic classes associated to the the connected components of the fixed point set of $g$, $M^g$, and the normal bundles to them. See \cite{ASII, ASIII} where \eqref{g-index} is denoted $L(g, D^+)$ and referred to as 
the {\bf  Lefschetz number} associated to $g\in G$ and $D^+$. When the Dirac operator is defined by an elliptic complex and  the fixed point set is made of a  finite number of points,  the Atiyah-Segal-Singer formula is nothing but the Atiyah-Bott-Lefschetz formula for such a complex. See \cite{AB1, AB2}. 
The original proof of the Atiyah-Segal-Singer formula  was a consequence of two fundamental results:\\
(i) the Atiyah-Singer $G$-index theorem, giving the equality of the topological and the analytic $G$-indices, as homomorphisms from $K_G (TX)$ to $R(G)$;\\
(ii) the computation of the topological $G$-index in terms of  data associated to the fixed point set, a result obtained by applying the localization theorem in K-theory, due to Segal. See \cite{segal}.

As for the Atiyah-Singer index formula, the heat equation was subsequently used  in order to give a different proof of  the Atiyah-Segal-Singer theorem for Dirac operators. This is due originally to Gilkey \cite{gilkey-equivariant} and 
Donnelly-Patodi \cite{donnelly-p}, using invariance theory. Later, more analytic approaches were proposed:

\begin{itemize}
\item 
by Berline-Vergne in \cite{BV}, see also \cite[Chapter 6]{BGV} and \cite{duistermaat}; this treatment, inspired by work of Bismut \cite{bismut-lefschetz}, employs crucially  a connection between the heat kernel of the Laplacian on the frame bundle of $M$ and the heat kernel of $D^2$ on  $M$; 
\item by Lafferty-Yu-Zhang in \cite{LYZ}; this work employs Yu's version \cite{yu} of  Getzler rescaling;
\item by  Liu and Ma in \cite{liu-ma}, where an extension to  the family context is also proved; this work builds on further work of Bismut \cite{bismut-two, bismut-jdg};
\item by Ponge and  Wang in   \cite{Ponge-Wang-2}, where the Volterra calculus is used.
\end{itemize}
For a treatment of the heat equation approach to Lefschetz formulas 
on (certain) groupoids we refer the reader to  \cite{haj-liu-loizides}.
Notice that in contrast with the original approach by Atiyah-Segal-Singer, the above contributions all produce a {\it local} equivariant index theorem.\\
Let us state the Atiyah-Segal-Singer formula, for simplicity
on a $G$-spin manifold $M$ and for a  spin-Dirac operator $D_V$ twisted by a complex G-equivariant vector bundle $V$ so that $E=\mathcal{S}\otimes V$ with $\mathcal{S}$ the spinnor bundle on $M$. Consider the fixed point set $M^g$ and let $N^g$ be the associated normal bundle. Then 
\begin{equation}\label{intro:ASS} 
\ind_G (D_V)(g)=\int_{M^g} \hat{A}(M^g)\wedge \operatorname{det}^{-\frac{1}{2}}\left(1-\gamma|_{N^g} e^{- \frac{R^{\perp}}{2\pi i}}\right)\wedge {\rm Tr}(g|_{V}e^{-\frac{F^V}{2\pi i}}) , 
\end{equation}
where 
$R^\perp$ is the curvature of the normal bundle $N^g$,
$g|_{N^g}$ is the action of $g$ on $N^g$, $g |_V$ is the
action of $g$ on $V$ 
and $F^V$ is the curvature of $V$.
We follow the convention that 
\[
\hat{A}(M^g)=\operatorname{det}^{\frac{1}{2}}\left(\frac{ R^{T} / 4\pi i }{\sinh \left( R^{T} / 4\pi i\right)}\right) ,
\] 
with $R^T$ the curvature of $TM^g$.\\
We shall sometime write the integrand in formula \eqref{intro:ASS} as $AS_g (D_V)$
and refer to it as the Atiyah-Segal-Singer integrand:
\begin{equation}\label{A-S-S-form} AS_g (D_V):= \hat{A}(M^g)\wedge \operatorname{det}^{-\frac{1}{2}}\left(1-\gamma|_{N^g} e^{- \frac{R^{\perp}}{2\pi i}}\right)\wedge {\rm Tr}(g|_{V}e^{-\frac{F^V}{2\pi i}}).
\end{equation}

\medskip
Consider now a complete Riemannian manifold $M$ of even dimension and the  action of a non-compact 
groups $G$ on $M$. We shall be interested in $G$ being an infinite discrete group and acting properly and cocompactly on $M$. (The case of $G$ being a finitely connected Lie group can also be considered but
leads to results that in the context of Lefschetz formulas
are less precise than in the discrete case \cite{Hochs-Song-Tang, Song-Tang}).

Natural examples arise from symmetric spaces and locally symmetric spaces. Take a connected semisimple Lie group $G$ with finite center and a maximal compact subgroup $K$. Then $X = G/K$ is the associated Riemannian symmetric space. Let $\Gamma \subset G$ be a discrete subgroup (for instance, a lattice), acting on $X$ by isometries via
$\gamma \cdot (gK) = (\gamma g)K$.
For any such discrete subgroup $\Gamma \subset G$ the action on $X$ is properly discontinuous, hence proper. If $\Gamma$ is not torsion-free, that is, if it contains elliptic elements, then the $\Gamma$-action on $X$ is proper but not free, and the corresponding locally symmetric quotient $\Gamma \backslash X$ is an orbifold. In this setting, the case of the subgroup of $G=PSL(2,\mathbb{R})$ generated by reflections across the sides of a hyperbolic triangle in $\mathbb{H}^2\cong PSL(2,\mathbb{R})/SO(2)$ gives a concrete example with compact quotients since the triangle itself gives a fundamental domain. More examples shall be given in the next Section.


Thus, let  us denote by $\Gamma$ our infinite discrete group and assume momentarily that
the action is also {\bf free}. Let $D$ be a Dirac operator on $M$, acting on the 
sections of a $\mathbb{Z}_2$-graded $\Gamma$-vector bundle $E$ over $M$ and  commuting with the action of $\Gamma$. Then there is, first of all, Atiyah's $\Gamma$-index 
${\rm ind}_{\Gamma}  (D^+),$
\cite{atiyah-gamma}.
In order to define ${\rm ind}_{\Gamma}  (D^+)$ we consider the
von Neumann algebra of all bounded $\Gamma$-equivariant operators on $L^2 (M,E)$; 
this is a von Neumann algebra endowed with a faithful positive normal trace $\Tr_\Gamma$.
The orthogonal projections onto $\Ker D^\pm$, denoted $\Pi_\pm$, are elements in this von Neumann
algebra and they 
are trace class with respect to $\Tr_\Gamma$. It thus  makes sense to define
$
{\rm ind}_{\Gamma}  (D^+):= \Tr_\Gamma  \Pi_+ - \Tr \Pi_-
$.
Atiyah's $\Gamma$-index theorem asserts that this number is equal to the index of the
operator induced on the compact quotient.

While initially defined using von Neumann algebras, Atiyah's $\Gamma$-index is in fact the
pairing of a K-theoretic index class associated to $D$ with a cyclic cocycle of degree 0.  
Let us see the details. One can define the algebra of $\Gamma$-equivariant pseudodifferential operators
of $\Gamma$-compact support (i.e., with compact support in $M\times M/\Gamma$ with respect to the diagonal action), denoted
$\Psi^{*}_{\Gamma,c}(M,E)$. By standard elliptic theory the Dirac operator $D^+$ admits a
parametrix $Q\in \Psi^{-1}_{\Gamma,c}(M,E^-, E^+)$ with remainders  $S_\pm\in\Psi^{-\infty}_{\Gamma,c}(M,E^\pm)$. In this article we shall  denote the algebra of smoothing operators 
$\Psi^{-\infty}_{\Gamma,c}(M,E)$ by $\mathcal{A}_\Gamma^c (M,E)$ and we adopt this notation from now on:
\begin{equation}\label{def-a-c}
\mathcal{A}_\Gamma^c (M,E):= \Psi^{-\infty}_{\Gamma,c}(M,E).
\end{equation}
We shall often expunge the vector bundle $E$ from the notation.\\
Consider the $2\times 2$ matrix
\begin{equation}\label{matrix-compact-intro}
P:= \left(\begin{array}{cc} S_{+}^2 & S_{+}  (I+S_{+}) Q\\ S_{-} D^+ &
I-S_{-}^2 \end{array} \right).
\end{equation}
This is an idempotent with entries in the unitalization of $\mathcal{A}_\Gamma^c (M,E)$;
also $e_1:=\left( \begin{array}{cc} 0 & 0 \\ 0&1_{E^-}
\end{array} \right)$
is an idempotent matrix in such unitalization and we define the {\em compactly supported  index class} associated to $D$ as
\begin{equation}\label{ind-cs}\operatorname{Ind}_c (D):= [P] - [e_1]\in K_0
  (\mathcal{A}_\Gamma^c (M,E)).
\end{equation}
The definition of the idempotent $P$  comes from the short exact sequence of algebras
$$0\to \Psi^{-\infty}_{\Gamma,c}(M,E)\to \Psi^0_{\Gamma,c} (M,E)\to \Psi^0_{\Gamma,c} (M,E)/\Psi^{-\infty}_{\Gamma,c}(M,E)\to 0$$
and the associated long exact sequence in K-theory.
We refer to \cite{CM} for the details. The  $\Gamma$-Trace  $\Tr_\Gamma$
defines a cyclic cocycle of degree $0$ on the algebra  $\mathcal{A}_\Gamma^c (M,E)$
and we have in fact an explicit formula for this trace on an element $S\in  \mathcal{A}_\Gamma^c (M,E)$:
 $$\Tr_\Gamma (S)=\int_{M} \chi(p) \tr_p K_{S} (p,p) \,{\rm dp}$$
 with $K_{S}$ the smooth integral kernel of $S$ 
 and with $\chi$ a cut-off function for the action of $\Gamma$ on $M$, that is a smooth positive real function 
 $\chi\in C^\infty_c (M)$ such that $\sum_\gamma \chi (\gamma p)=1$ for all $p\in M$.
 Using the proof of Atiyah's $\Gamma$-index theorem one checks that for the pairing 
 of the index class $\operatorname{Ind}_c (D)\in K_0 (\mathcal{A}_\Gamma^c (M,E))$
  with the cyclic cocycle $[\Tr_\Gamma]\in HC^0 (\mathcal{A}_\Gamma^c (M,E))$ we have:
 $$\ind_\Gamma (D^+)=\langle \Ind_{c} (D), [\Tr_\Gamma] \rangle$$
 as anticipated.
There are other numbers that we can extract from the index class $\operatorname{Ind}_c (D)\in 
K_0 (\mathcal{A}_\Gamma^c (M,E))$. These numbers, known as (compactly supported) higher indices, are defined as follows. Recall that we have a decomposition by \cite{burghelea}
$$HC^\bullet (\mathbb{C}\Gamma)=\prod_{\langle \gamma \rangle} HC^\bullet (\mathbb{C}\Gamma,\langle \gamma\rangle)=
HC^\bullet (\mathbb{C}\Gamma,\langle e\rangle)\times \prod_{\langle \gamma \rangle, \gamma\not= e} HC^\bullet (\mathbb{C}\Gamma,\langle \gamma\rangle)\,,$$
where $HC^\bullet (\mathbb{C}\Gamma,\langle \gamma\rangle)$ is the cohomology of the subcomplex of the cyclic complex $C^{\bullet}_\lambda (\mathbb{C}\Gamma)$ made of cyclic cocycles supported on the conjugacy class $\langle \gamma \rangle$. 
As an example, consider  $\tau_{\langle \gamma \rangle}: \mathbb{C}\Gamma \to \mathbb{C}$,
\begin{equation}\label{delocalized-trace}
\tau_{\langle \gamma \rangle} (\sum_g a_g g):= \sum_{g\in \langle \gamma \rangle} a_g \,.
\end{equation}
This is indeed a  0-degree cyclic cocycle for $\mathbb{C}\Gamma$: $\tau_{\langle \gamma \rangle}\in HC^0 (\mathbb{C}\Gamma, \langle \gamma \rangle)$.\\
We  recall that there are natural isomorphisms  $\alpha: HC^\bullet (\mathbb{C}\Gamma,\langle e\rangle)\rightarrow H^\bullet (\Gamma,\mathbb{C})$
and $\eta: H^\bullet (\Gamma,\mathbb{C}) \to H^\bullet (B\Gamma,\mathbb{C})$. 
Let us denote by $$\beta:  HC^\bullet (\mathbb{C}\Gamma,\langle e\rangle)\rightarrow H^\bullet (B\Gamma,\mathbb{C})$$ the composition of these two isomorphisms. 
There is a natural chain map $\Phi:  C^{\bullet}_\lambda (\mathbb{C}\Gamma)\to  C^{\bullet}_\lambda (\mathcal{A}_\Gamma^c (M,E))$, explained in the next Section, and this induces a map in cyclic cohomology. For example, the canonical trace on $\mathbb{C}\Gamma$, that is the map defined by 
$$\sum_g a_g g \to\,\, a_e$$
gives rise to $\Tr_\Gamma$, whereas 
$\tau_{\langle \gamma \rangle}\in HC^0 (\mathbb{C}\Gamma, \langle \gamma \rangle)$, see \eqref{delocalized-trace}, gives rise to 
$\Phi (\tau_{\langle \gamma \rangle})\in HC^0 (\mathcal{A}_\Gamma^c (M,E))$ given explicitly by
$$ \Phi (\tau_{\langle \gamma \rangle}) (K) = \sum_{g\in \langle \gamma \rangle} \int_M \chi(p) K(p, g p) \,{\rm dp}
$$
with $\chi$ denoting again a cut-off function for the action of $\Gamma$. 
Going back to the general case, 
if $\tau\in HC^\bullet (\mathbb{C}\Gamma)$ then we have $\Phi (\tau)\in 
HC^\bullet (\mathcal{A}_\Gamma^c (M,E))$ and the higher index associated to $\tau$ is, by definition, the number
obtained by pairing  $\operatorname{Ind}_c (D)\in 
K_0 (\mathcal{A}_\Gamma^c (M,E))$ with $\Phi (\tau)\in 
HC^\bullet (\mathcal{A}_\Gamma^c (M,E))$; viz. $\langle \operatorname{Ind}_c (D), \Phi (\tau)\rangle$.
The Connes-Moscovici higher index theorem \cite{CM}, one of the most profound results in this area of Mathematics,
 gives a formula for such a number:
\begin{itemize} 
\item if $\tau\in HC^{2q} (\mathbb{C}\Gamma,\langle e\rangle)$  and $\psi:M/\Gamma \to B\Gamma$
is the classifying map for the principal bundle $\Gamma\to M\to M/\Gamma$, then
\begin{equation}\label{cm-intro}
\langle \operatorname{Ind}_c (D), \Phi (\tau)\rangle=\frac{(-1)^{\dim M}}{(2\pi i)^q}\frac{q!}{(2q) !}\int_{M/\Gamma}  AS (M/\Gamma) \wedge \psi^* (\beta (\tau))
\end{equation}
with $AS (M/\Gamma)$ the Atiyah-Singer cohomology class;
\item if  $\tau\in HC^{2q} (\mathbb{C}\Gamma,\langle \gamma\rangle)$, $\gamma\not= e$ then 
\begin{equation}\label{cm-intro-bis}\langle \operatorname{Ind}_c (D), \Phi (\tau)\rangle=0\,.
\end{equation}
\end{itemize}
Further contributions around the higher index formula were given in \cites{moscovici-wu, lott-gafa, g-dK-n, ppt-jdg}.
We point out that together with the compactly supported index class we also have a $C^*$-index class
$\Ind (D)\in K_\bullet (C^* (M,E)^\Gamma)$, with $C^* (M,E)^\Gamma$ denoting the Roe algebra associated to $M$ and $\Gamma$, that is the closure of 
$\mathcal{A}_\Gamma^c (M,E)$ in the $C^*$-algebra of bounded operators on $L^2 (M,E)$.
In general, it is not known whether $\langle \cdot, \Phi (\tau)\rangle: K_0 (\mathcal{A}_\Gamma^c (M,E))\to \mathbb{C}$ extends to $\langle \cdot, \Phi (\tau)\rangle: K_0 (C^* (M,E)^\Gamma)\to \mathbb{C}$.
Connes and Moscovici prove that this extension exists if $\Gamma$ is Gromov hyperbolic, a result that 
implies the Novikov conjecture on the homotopy invariance of Novikov higher signatures
for manifolds with a Gromov hyperbolic fundamental group. Notice incidentally
that $\Tr_\Gamma$ does extend to  $ K_0 (C^* (M,E)^\Gamma)$.

 \medskip
 Let us now drop the additional assumption that our action $\Gamma\times M\to M$ is free and consider the general case of a {\bf proper cocompact action} of $\Gamma$ on $M$. There are still  index classes 
 $$\operatorname{Ind}_c (D)\in 
K_0 (\mathcal{A}_\Gamma^c (M,E))\quad\text{and}\quad \Ind (D)\in K_0 (C^* (M,E)^\Gamma)$$ and there  is a numeric von Neumann 
index $\ind_\Gamma (D)$, studied thoroughly in \cite{wang-jncg}.
This numeric index, originally defined in a von Neumann context,  is in fact  obtained by  pairing  the index class $\Ind_c (D)$ with the 0-cyclic cocycle 
defined by the analogue of $\Tr_\Gamma$. This pairing actually extends to the $C^*$-index class $\Ind (D)\in K_\bullet (C^* (M,E)^\Gamma)$. 
 Hang Wang proved in  \cite{wang-jncg} a formula for this numeric index, extending 
Atiyah's result to the general proper case. 

Next we consider the pairing of the index class $\operatorname{Ind}_c (D)$ with higher cyclic cocycles 
localized at the identity element: if $\tau\in HC^\bullet (\mathbb{C}\Gamma,\langle e\rangle)$ then 
there is an explicit formula for the pairing 
$\langle \operatorname{Ind}_c (D), \Phi (\tau)\rangle$ and this reads
$$\langle \operatorname{Ind}_c (D), \Phi (\tau)\rangle=\frac{(-1)^{\dim M}}{(2\pi i)^q}\frac{q!}{(2q) !}\int_{M}  \chi AS (M) \wedge \Psi_M (\beta (\tau)),$$
where $\chi$ is a compactly supported cutoff function for the $\Gamma$ action on $M$
and  $\Psi_M: H^\bullet(B\Gamma)\to H^\bullet(M)^\Gamma$ is the van Est map, \cite[Theorem 1.6]{ppt-jdg}. This is due to Pflaum, Posthuma and Tang, see\footnote{We point out that the difference between the normalization factor $\frac{(-1)^{\dim M}}{(2\pi i)^q}\frac{q!}{2q !}$ in the above formula and the one in \cite[Theorem 3.1]{ppt-jdg} comes from the definition of the pairing between cyclic cohomology and $K$-theory.  In \cite{ppt-jdg}, the authors followed \cite[Section 8.3]{Loday} and considered the pairing between the $B$-$b$ bicomplex and  $K$-theory, while here we follow \cite{CM} and consider the pairing between Connes' $\lambda$-complex and $K$-theory. } \cite[Theorem 3.1]{ppt-jdg}.
This result was proved making use of the algebraic index theory of Fedosov, Nest, and Tsygan. Heat kernel proofs 
were subsequently\footnote{These two articles treat the case
of $G$-proper manifolds, with $G$ an almost connected Lie group; however, the proof can be easily adapted to the discrete case.} given in  \cite{PP2} and \cite{wang-wang-jncg}. Notice that in these two latter contributions one must
assume, because of the use of the heat kernel, that the cyclic cocycle $\tau$  has at most exponential growth. 

Now, in addition to the cocycles in  $HC^\bullet (\mathbb{C}\Gamma,\langle e\rangle)$
we know  that there are other cyclic cocycles in $HC^\bullet (\mathbb{C}\Gamma)$
and, consequently, in $HC^\bullet( \mathcal{A}_\Gamma^c (M,E))$. 
These are all the {\it delocalized} cyclic cocycles in $\prod_{\langle \gamma \rangle, \gamma\not= e} HC^\bullet (\mathbb{C}\Gamma,\langle \gamma\rangle)$. 
For example
$\tau_{\langle \gamma \rangle}\in HC^0 (\mathbb{C}\Gamma, \langle \gamma \rangle)$, see \eqref{delocalized-trace}, and hence
$\Phi (\tau_{\langle \gamma \rangle})\in HC^0 (\mathcal{A}_\Gamma^c (M,E))$.
In contrast with the free case, see \eqref{cm-intro-bis}, we have that
for the delocalized trace $\tau_{\langle \gamma \rangle}$ it holds that  $\langle \operatorname{Ind}_c (D),\Phi (\tau_{\langle \gamma \rangle}) \rangle \not= 0$. In fact, there is a formula for this pairing due to Hang Wang and 
Bai-Ling Wang \cite[Theorem 6.1]{wang-wang} and the formula applies in fact to the
$C^*$-index class $\Ind (D)$ but with additional assumptions on the group $\Gamma$, for example
$\Gamma$ Gromov hyperbolic. The Wang-Wang formula reads
\begin{equation}\label{wang-wang-jdg}
\langle \operatorname{Ind} (D),\Phi (\tau_{\langle \gamma \rangle}) \rangle=\int_{M^\gamma} \chi_\gamma AS_\gamma (D)
\end{equation}
with $M^\gamma$ denoting the fixed-point set of $\gamma$,
$AS_\gamma (D) $ the Atiyah-Segal-Singer form and $\chi_\gamma$  a cut-off function for the action of the centralizer $Z_\gamma$ on $M^\gamma$.
 We refer to this formula as a {\bf Lefschetz formula} on the $\Gamma$-proper manifold $M$.

\medskip
\noindent
We finally come to the main question of this article: can one prove a {\bf higher Lefschetz formula}
for the pairing of the index class with the elements in $\prod_{\langle \gamma \rangle, \gamma\not= e} HC^\bullet (\mathbb{C}\Gamma,\langle \gamma\rangle)$, $\bullet>0$ ?

\medskip
\noindent
{\em The main goal of this article is to give an affermative answer to this question and to establish such a formula.}

\medskip
\noindent
In order to state our formula we need to further explore the groups $HC^\bullet (\mathbb{C}\Gamma, \langle \gamma \rangle)$. First we make a few preliminary remarks.
If $\gamma\in \Gamma$ is such that $M^\gamma=\emptyset$, then one can prove that
$\langle \operatorname{Ind}_c (D), \Phi (\tau)\rangle=0$ for any 
$\tau \in HC^\bullet (\mathbb{C}\Gamma, \langle \gamma \rangle)$
and the
argument is exactly the same as the one establishing the same result in the free case\footnote{This argument is in fact implicitly contained in our analysis below.}, see \eqref{cm-intro-bis}.
Thus we can assume that $\gamma\in\Gamma$ is such that $M^\gamma\not= \emptyset$. Next we observe
that because of the assumed properness of the action, an element $\gamma$ such that 
$M^\gamma\not= \emptyset$ must necessarily be of finite order. Finally, notice that since we are interested in the pairing of K-theory with cyclic cohomology, we only care about {\em periodic} 
cyclic cohomology $HP^\bullet$. Now, by Burghelea's theorem \cite{burghelea}
we know that for an element $\gamma$ of finite order
$$HP^\bullet (\mathbb{C}\Gamma,\langle \gamma \rangle)= H^\bullet (N_\gamma,\mathbb{C})$$
where equality is meant up to isomorphism;
here $\bullet$ equals  {\em even} or {\em odd} and  $N_\gamma$ denotes the normalizer of $\gamma$
in $\Gamma$. Since $\gamma$ is finite order we have $H^\bullet (N_\gamma,\mathbb{C})=H^\bullet (Z_\gamma,\mathbb{C}),$ with $Z_\gamma$ denoting the centralizer, see \cite[Corollary 7.9]{PSZ}. Recall now that Lott, in \cite{LottII}, introduces a complex $\mathscr{C}^\bullet_\gamma$
that is made of {\em group cocycles on $\Gamma$} and that
computes $H^\bullet (N_\gamma,\mathbb{C}) (= H^\bullet (Z_\gamma,\mathbb{C}))$.
Summarizing, if $\gamma\in \Gamma$ is an element of finite order and if $H\mathscr{C}_\gamma^\bullet $ denotes the cohomology of Lott's complex, then 
$$H\mathscr{C}_\gamma^\bullet =  H^\bullet (N_\gamma,\mathbb{C}) = H^\bullet (Z_\gamma,\mathbb{C})= HP^\bullet
(\mathbb{C}\Gamma,\langle \gamma \rangle)\,,\quad \bullet= {\rm even}\;\;\text{or} \;\;{\rm odd}\,.$$
The advantage of using Lott's complex  is  that it comes with a natural chain map: 
$\tau:
\mathscr{C}^\bullet_\gamma \to C^\bullet_\lambda (\mathbb{C}\Gamma,\langle \gamma \rangle)$ implementing the isomorphism $H\mathscr{C}_\gamma^\bullet \simeq HP^\bullet
(\mathbb{C}\Gamma,\langle \gamma \rangle)$. Composing 
this chain map with the natural chain map
$\Phi:  C^\bullet_\lambda (\mathbb{C}\Gamma,\langle \gamma \rangle)\to C^\bullet_\lambda (\mathcal{A}_\Gamma^c (M,E))$
that we have already encountered, gives a chain map
$$\mathscr{C}^\bullet_\gamma \,\xrightarrow{\Phi\circ \tau}  \,C^\bullet_\lambda (\mathcal{A}_\Gamma^c (M,E)).$$
Recall $\Ind_c (D)\in K_\bullet ( \mathcal{A}_\Gamma^c (M,E))$. Given $c$ in Lott's complex,  $c\in \mathscr{C}^\bullet_\gamma$, our goal is therefore to give a geometric formula for 
$\langle \operatorname{Ind}_c (D), \Phi (\tau (c))\rangle\,.$

In order to state the formula we observe in Section \ref{sect:characteristic} that 
the chain maps 
$$\tau:
\mathscr{C}^\bullet_\gamma \to C^\bullet_\lambda (\mathbb{C}\Gamma,\langle \gamma \rangle)\quad
\text{and}\quad \Phi:  C^\bullet_\lambda (\mathbb{C}\Gamma,\langle \gamma\rangle)\to C^\bullet_\lambda (\mathcal{A}_\Gamma^c (M,E))$$
fit into a diagram:
\begin{equation}
\label{diagram-nc-intro-bis}
\begin{tikzcd}
\mathscr{C}^\bullet_\gamma(\Gamma) \arrow[rr, "\tau"] \arrow[dd, "\Psi_{\rm inv}"'] \arrow[rd, "\Psi"] &                                                                                            & {C_\lambda^\bullet(\mathbb{C}\Gamma, \langle \gamma\rangle) } \arrow[dd, "\Phi"] \\
                                                                                                       & {\mathsf{E}^\bullet_{\rm AS}(M,\gamma)}\arrow[rd, "\rho"] \arrow[ld, "P"', shift right=2] &                                                     \\
{C^\bullet_{\rm AS, inv}(M,\gamma)} \arrow[rr, "\rho_{\rm inv}"] \arrow[ru, "I"']                      &                                                                                            & C_\lambda^\bullet(\mathcal{A}^c_\Gamma(M))                
\end{tikzcd}
\end{equation}
On the left bottom corner a $\gamma$-localized Alexander-Spanier complex appears and
the left vertical map turns out to be a quasi-isomorphism; we refer to the main text for the definitions 
of $\Psi_{{\rm inv}}$ and $\rho_{{\rm inv}}$. At the center of the diagram we have 
a {\em $\gamma$-extended} Alexander-Spanier complex; the maps 
$\Psi$ and $\rho$ are extended versions  of $\Psi_{{\rm inv}}$ and $\rho_{{\rm inv}}$.
The three triangles commutes while the outer rectangle does not commute but thanks to the maps $P$ and $I$
it commutes up to homotopy. Consequently we obtain 
the crucial equality
\begin{equation}\label{intro-equality-indices-bis}
\langle \operatorname{Ind}_c (D), \Phi (\tau (c))\rangle=\langle \operatorname{Ind}_c (D), \rho(\Psi (c))\rangle\,,\ \forall c\in \mathscr{C}^\bullet_\gamma(\Gamma),
\end{equation}
and when $c$ is a cocycle
\begin{equation}\label{intro-equality-indices-ter}
\langle \operatorname{Ind}_c (D), \Phi (\tau (c))\rangle=\langle \operatorname{Ind}_c (D), \rho(\Psi(c))\rangle=\langle \operatorname{Ind}_c (D), \rho_{\rm inv}(\Psi_{\rm inv}(c))\rangle.
\end{equation}
We shall in fact establish our main result by proving
 a formula for the pairing $\langle \operatorname{Ind}_c (D), \rho(\Psi(c))\rangle$.

In order to state the formula
we begin by showing that the  $\gamma$-localized Alexander-Spanier complex 
comes with a chain map  $\Lambda^\gamma: C^\bullet_{\rm AS, inv}(M, \gamma)\to \Omega(M^\gamma)^{Z_\gamma}$.
Pre-composing $\Lambda^\gamma$ with $\Psi_{\rm inv}$, we obtain the following chain map
\begin{equation}\label{psigamma}
\Psi^\gamma:=\Lambda^\gamma\circ \Psi_{\rm inv}: \mathscr{C}^\bullet_\gamma(\Gamma)\to \Omega^\bullet(M^\gamma)^{Z_\gamma}
\end{equation}
which to each cocycle $c\in \mathscr{C}^\bullet_\gamma(\Gamma)$ 
associated a  $Z_\gamma$-invariant differential form $\Psi^\gamma (c)\in \Omega^\bullet(M^\gamma)^{Z_\gamma}$.

\bigskip
We now finally come to our main result, that is, an explicit formula 
for the pairing $\langle \operatorname{Ind}_c (D), \Phi (\tau (c))\rangle$, or, equivalently, for the 
pairing 
$\langle \operatorname{Ind}_c (D), \rho(\Psi (c))\rangle$, see
\eqref{intro-equality-indices-ter}.
We intend to use heat-kernel techniques. In previous work in higher index theory this is achieved 
through the use of the Connes-Moscovici matrix
\begin{equation}\label{eq:R-heat-intro}
  R(\sqrt tD)=
  \begin{pmatrix}
    e^{-t\Delta_+}
    & \left(\frac{1-e^{-t\Delta_+}}{t\Delta_+}\right)e^{-t\Delta_+/2}\sqrt tD^- \\
    e^{-t\Delta_-/2}\sqrt tD^+
    & -e^{-t\Delta_-}
  \end{pmatrix}.
\end{equation}
with 
$$\Delta_+:= D^- D^+, \quad\quad \Delta_-:= D^+ D^-\,.$$
Needless to say, the operators appearing in this matrix are not of $\Gamma$-compact support;
indeed, they all belong to the algebra $\mathcal{A}^{{\rm exp}}_\Gamma (M,E)$ of $\Gamma$-equivariant smoothing kernels 
of exponential rapid decay; consequently, in previous contributions in higher index theory in the proper case,
growth assumptions had to be placed upon the analogue of the cocycle $c\in \mathscr{C}^\bullet_\gamma(\Gamma)$. See 
 \cite{PP2} \cite{wang-wang-jncg}. {\em Through a delicate use of finite propagation techniques we shall be able to avoid this  extra hypothesis
and prove a  general result.} We shall elaborate on this point further ahead in this Introduction.

\smallskip
 \noindent
We are thus ready to state our main result, a {\bf higher Lefschetz formula}
on $\Gamma$-proper manifolds. We state the result, for simplicity, on a $\Gamma$-spin manifold and for a  
$\Gamma$-invariant spin-Dirac operator
twisted by an auxiliary $\Gamma$-vector bundle $V$: 
\begin{theorem}\label{intro-main-theorem} Let $\gamma\in \Gamma$ be such that the fixed point set $M^\gamma$ is non-empty. Consider a  cocycle $c \in \mathscr{C}^{2q}_\gamma(\Gamma)$. Recall 
the chain map 
$\Psi^\gamma:  \mathscr{C}^\bullet_\gamma(\Gamma)\to \Omega^\bullet(M^\gamma)^{Z_\gamma}$ in \eqref{psigamma}. Then the following formula holds
\[
\langle\Ind_{c} (D_V), \Phi(\tau (c))\rangle
= c(q,n) \int_{M^\gamma} (-i)^{\frac{n-a}{2}}\chi_\gamma\,\Psi^\gamma(c)
  \wedge AS_\gamma (D_V)\,. 
\]
where 
\begin{itemize}
\item
$c(q,n)=2(-1)^q \frac{ q!}{(2 \pi i)^q (2q)!}$, 
\item $a$, the dimension of a connected component of $M^\gamma$,  is a locally constant function on $M^\gamma$,
\item $\chi_\gamma$ is a cutoff function for the  action of $Z_\gamma$ on $M^\gamma$, 
\item and
$AS_\gamma (D_V)=\widehat{A}(M^\gamma) \operatorname{det}^{-\frac{1}{2}}\left(1-\gamma|_{N^\gamma} e^{- \frac{R^{\perp}}{2\pi i}}\right)\wedge \Tr(\gamma e^{-\frac{F^V}{2\pi i}} )$ with $R^\perp$ denoting the curvature of the normal bundle to $M^\gamma$, $N^\gamma$, and $F^V$ denoting the curvature of the auxiliary bundle 
$V$
\end{itemize}
\end{theorem}
%
%

\noindent

 
\medskip

We summarize the above material  in the following two tables, where we  present
results relative to the pairing of the index class with the cyclic cohomology groups appearing in the first line. In Table \ref{table intro1} we look at the free proper case, whereas in Table \ref{table intro2}  we consider the general proper case.

\begin{table}[h] \label{table intro1}
\begin{tabular}{|c|c|c|c|}
\hline 
 \begin{tabular}{l} $ HC^0\left(\mathbb{C}\Gamma, \langle e \rangle \right)$ \end{tabular} & \begin{tabular}{l} $  HC^0\left(\mathbb{C}\Gamma, \langle \gamma \rangle \right)  $\end{tabular} & \begin{tabular}{l}$ HC^\bullet\left(\mathbb{C}\Gamma, \langle e \rangle \right)$, $\bullet>0$\end{tabular} & \begin{tabular}{l}$ HC^\bullet\left(\mathbb{C}\Gamma, \langle \gamma \rangle \right)$, $\bullet>0$ \end{tabular}\\
\hline Atiyah \cite{atiyah-gamma}& vanish & Connes-Moscovici \cite{CM}& vanish\\
\hline
\end{tabular}
\caption{\small Pairing of the index class with cyclic cohomology in the free proper case}
\end{table}
\begin{table}[h] \label{table intro2}
\begin{tabular}{|c|c|c|c|}
\hline 
 \begin{tabular}{l} $ HC^0\left(\mathbb{C}\Gamma, \langle e \rangle \right)$ \end{tabular} & \begin{tabular}{l} $  HC^0\left(\mathbb{C}\Gamma, \langle \gamma \rangle \right)  $\end{tabular} & \begin{tabular}{l}$ HC^\bullet\left(\mathbb{C}\Gamma,  \langle e \rangle \right)$, $\bullet>0$\end{tabular} & \begin{tabular}{l}$ HC^\bullet\left(\mathbb{C}\Gamma, \langle \gamma
 \rangle \right)$, $\bullet>0$ \end{tabular}\\
\hline Wang \cite{wang-jncg}& Wang-Wang \cite{wang-wang}& Pflaum-Posthuma-Tang \cite{ppt-jdg}& This article\\
\hline
\end{tabular}
\caption{\small Pairing of the index class with cyclic cohomology in the general proper case}
\end{table}

 \noindent
 We comment briefly on the proof of our main result. We build heavily on the
 treatment of the Atiyah-Segal-Singer equivariant index formula 
 given by Ponge and Wang in \cite{Ponge-Wang-2}; this treatment employs crucially  the Volterra calculus, as developed\footnote{Incidentally, it would be interesting to 
 give a treatment of the heat equation proof of the Atiyah-Segal-Singer formula, and in fact also
 of the higher formula we prove in this article, using 
 the heat calculus of Melrose, see \cite{tapsit}.} in \cites{Piriou, greiner, BGS, Ponge, Ponge-2}. 
 We also employ ideas appearing in the proof given
 by Connes and Moscovici of  the higher index formula for  cyclic cocycles localized at the identity element.
 However, 
 we have to face additional challenges with respect to these two references.
 These challenges are implicitly described in the following schematic description of
 our proof.
 
 \bigskip
 \noindent
 
{\bf (1).} First of all, the computation of $\langle \operatorname{Ind}_c (D), \rho(\Psi (c))\rangle$
involves a parametrix of $\Gamma$-compact support and, thus, remainders that are also of $\Gamma$-compact support. See the expression of the index idempotent in \eqref{matrix-compact-intro}. It is well known that the corresponding index class, and thus the index pairing, is independent of the choice of  
such a parametrix. However, the moment we want to involve the heat-kernel, through the use 
of the Connes-Moscovici matrix  $R(\sqrt tD)$, we leave the world of operators
of $\Gamma$-compact support and thus face additional problems; for example, the well-definedness of the pairing. In a first step we introduce a finite propagation idempotent 
$P_{\operatorname{fp}}(t)$, with the property that 
$R_{\operatorname{fp}}(t)=P_{\operatorname{fp}}(t)-\begin{pmatrix}0&0\\0&1\end{pmatrix}$
is a suitable approximation of the Connes-Moscovici matrix. This we do through a delicate use 
of Fourier transform and finite propagation techniques.

\medskip
\noindent
{\bf (2).}  The index pairing can now be computed through the use of our finite propagation idempotent.
This pairing is of course independent of the time variable $t$ and can be computed, at least in principle,  by computing the limit
\[
\lim_{t\to 0} \rho\circ \Psi(c) (R_{\operatorname{fp}}(t), \cdots, R_{\operatorname{fp}}(t)),\qquad R_{\operatorname{fp}}(t)=P_{\operatorname{fp}}(t)-\begin{pmatrix}0&0\\0&1\end{pmatrix}.
\]
Now, the propagation of $R_{\operatorname{fp}}(t)$ is uniformly finite, independent of $t$; this property allows us to {\em truncate} the sum expressing $\rho\circ \Psi(c)$, a priori an infinite sum, into a {\em finite sum}.

\medskip
\noindent
{\bf (3).} 
Let $R(\sqrt{t}D)$ be the Connes-Moscovici matrix. We prove that the difference of the smoothing kernels $R_{\operatorname{fp}}(t))-R(\sqrt{t}D)$ decays faster than any $t^k$, $\forall k\in \mathbb{N}$ as $t\to 0$.
Thus, as far as the computation of $\lim_{t\to 0} $ is concerned,  we can  substitute in the above {\rm finite} sum the matrix $R_{\operatorname{fp}}(t)$
with the matrix $R(\sqrt{t}D)$; it is this crucial step that will allow us to use Getzler rescaling techniques.

\medskip
\noindent
{\bf (4).} 
We use the Volterra pseudodifferential calculus and  Getzler-Volterra rescaling techniques 
as in  \cite{Ponge-Wang-2} in order to compute the short time limit of the finite sum involving $R(\sqrt{t}D)$. Here there are challenges that come from the very use of the Volterra calculus. Indeed, in order 
to make use of the results established by Ponge and Wang we need to be able to commute, within a trace functional,
certain compactly supported differential operators with the operators appearing
in the Connes-Moscovici matrix. Connes and Moscovici tackled the analogous  step, for cyclic cocycles localized at the identity,
using the full Getzler calculus, as developed in \cite{getzler-CMP}.
In the approach through the Volterra calculus we are able to prove an analogous result but the proof turns out to be rather technical. We give all the details in subsection \ref{subsect:simplification}.

 \bigskip
 \noindent
 {\bf Organization of the paper.}
 
 \medskip
\noindent
 In {\bf Section \ref{sect:characteristic}} we introduce the constituents 
 of the diagram appearing in \eqref{diagram-nc-intro-bis} and discuss the commutativity properties of the diagram. We also prove that the chain map $ \Psi_{{\rm inv}}:  \mathscr{C}^\bullet_\gamma
\to C^\bullet_{\rm AS, inv}(M,\gamma)$ is a quasi-isomorphism. Finally, we introduce 
the chain map 
$\Lambda^\gamma: C^\bullet_{\rm AS, inv}(M, \gamma)\to \Omega(M^\gamma)^{Z_\gamma}$
and, consequently, the chain map 
$\Psi^\gamma:=\Lambda^\gamma\circ \Psi_{\rm inv}: \mathscr{C}^\bullet_\gamma(\Gamma)\to \Omega^\bullet(M^\gamma)^{Z_\gamma}$; the latter chain map is the one  that figures in our higher Lefschetz formula.

\medskip
\noindent
In {\bf Section \ref{section:finite propagation}} we introduce the finite propagation matrix 
$R_{\operatorname{fp}}(t)$ and  prove that the difference of the smoothing kernels $R_{\operatorname{fp}}(t))
-R(\sqrt{t}D)$ decays faster than any $t^k$, $\forall k\in \mathbb{N}$ as $t\to 0$. As already explained,
this will allow us to use Getzler-Volterra rescaling techniques in order to compute the short-time limit
$$\lim_{t\to 0} \rho\circ \Psi(c) (R_{\operatorname{fp}}(t), \cdots, R_{\operatorname{fp}}(t))\,.$$

\medskip
\noindent
In {\bf Section  \ref{section:volterra-main} } we give new results on the Volterra calculus. More precisely, 
Subsection \ref{subsect:simplification} deals
 with the problem of commuting terms in the index pairing; this is a long and combinatorially complex 
 treatment. In Subsection \ref{subsection:useful}, on the other hand, we consider Volterra operators 
 related to the entries in the Connes-Moscovici projector and  compute their Getzler order.

 \medskip
\noindent
In the last Section, {\bf Section  \ref{section:higher lefschetz}}, we finally give the 
proof of 
Theorem  \ref{intro-main-theorem}. Although we do follow the strategy in \cite{CM}
(and also \cite{moscovici-wu}), we must here cope with several complications 
arising from the use of the Volterra calculus and the proper
equivariant situation we are considering.

\medskip
\noindent
Basic results about the Volterra calculus are given in the {\bf Appendix}.
We give a quick introduction to the basic definitions and results in Subsection 
\ref{appendix-subsection:basic-volterra}; next we
extend the calculus to our $\Gamma$-equivariant situation  in Subsection 
 \ref{subsection:volterra-equivariant}, where we also investigate the structure of 
 the operator $(\partial_t + D^2)^{-1}$ in this more general context. Subsection \ref{subsection:getzler-volterra}
 is a summary of results around the localization of trace integrals around fixed point sets, based on the the work of Ponge and Wang \cite{Ponge-Wang-2}.

\medskip
\noindent
We leave the study of the pairing of the $C^*$-index class $\Ind (D)\in K_\bullet (C^* (M,E)^\Gamma)$ with the cyclic cocycles defined by $HP^\bullet(\mathbb{C}\Gamma,\langle \gamma \rangle)$ to a continuation of this work. There we shall also investigate the
connection between our formula 
 and the formulas established in \cite{CWW} and \cite{perrot}.

\medskip
\noindent
We wish to end this Introduction by recalling  that  higher index formulas have been 
pionereed (in the foliated context) by our dear colleague Moulay-Tahar Benameur; see
\cite{benameur} and also \cites{bh-1, bh-2}.

 \bigskip
 \noindent
{\bf Acknowledgements.} 
We wish to thank Thomas Krainer, Qiaochu Ma, Rapha\"el Ponge, Hang Wang and Guoliang Yu
for interesting discussions and email correspondence. We  acknowledge help from  
AI (Chatgpt, Version 5.5 Pro) in order to come-up with the right finite propagation approximation of the Connes-Moscovici  
matrix in Section \ref{section:finite propagation}. 
%
Song's research was partially supported by the NSF grant  DMS-1952557 and Simons Foundation grant MPS-TSM-00014295. Tang's research was partially supported the NSF grants DMS-1952551, DMS-2350181, and Simons Foundation grant MPS-TSM-00007714. Visits of Xiang Tang and Hessel Posthuma to Sapienza Università di Roma were partially sponsored by Sapienza through
{\em Ricerca di Ateneo 2024} and the program {\em Professori visitatori} respectively.
Visits of Paolo Piazza to University of Amsterdam 
and Washington University in St Louis were partially funded by NWO grant 613.001.751 and DMS-1952551 respectively.\\

\section{The characteristic map for proper actions of discrete groups}\label{sect:characteristic}

\subsection{Cyclic cohomology of the group algebra $\CC\Gamma$}
For any unital algebra $A$, the cyclic cohomology $HC^\bullet(A)$ is defined as the cohomology of the subcomplex of Hochschild cochains that are cyclic:
\[
C^k_\lambda:=\{\phi:A^{\otimes{k+1}}\to\mathbb{C}, \phi(a_0,\ldots,a_k)=(-1)^k\phi(a_k,a_0,\ldots,a_{k-1})\}.
\]
The Hochschild differential $b:C^k_\lambda(A)\to C^{k+1}_\lambda(A)$ is defined by the standard formula
\[
b\phi(a_0,\ldots,a_{k+1})=\sum_{i=0}^k(-1)^i\phi(a_0,\ldots,a_ia_{i+1},\ldots,a_{k+1})+(-1)^{k+1}\phi(a_{k+1}a_0,\ldots,a_k).
\]
Let $\Gamma$ be a discrete group, and denote by $\CC\Gamma$ its group algebra. By the well-known calculation of Burghelea \cite{burghelea}, the cyclic cohomology of $\mathbb{C}\Gamma$ splits up 
as a direct product over conjugacy classes
\begin{equation}
\label{eq-dec-b}
HC^\bullet(\mathbb{C}\Gamma)=\prod_{\langle\gamma\rangle\in{\rm Conj}(\Gamma)\atop {\rm finite~order}} H^\bullet(N_\gamma,\mathbb{C})\otimes HC^\bullet(\mathbb{C})\times \prod_{\langle\gamma\rangle\in{\rm Conj}(\Gamma)\atop {\rm infinite~order}} H^\bullet(N_\gamma,\mathbb{C}),
\end{equation}
where $N_\gamma:=Z_\gamma\slash \gamma^\mathbb{Z}$, with $Z_\gamma$ the centralizer of $\gamma\in\Gamma$, and $H^\bullet(N_\gamma,\mathbb{C})$ denotes its group cohomology. The part of the cyclic cohomology `localized at the unit' corresponds to the trivial conjugacy class $\langle e\rangle=\{e\}$ for which $N_e=\Gamma$ and we find a copy of the group cohomology $H^\bullet(\Gamma,\mathbb{C})$ of $\Gamma$ itself. The part of $HC^\bullet(\mathbb{C}\Gamma)$ that corresponds to the nontrivial conjugacy classes is referred to as `delocalized'. In this paper we shall only concern ourselves with the conjugacy classes $\langle\gamma\rangle$ 
of finite order, i.e., where $\gamma$ is a torsion element, c.f.\ Remark \ref{rk-te}. 

Let us fix $\gamma\in\Gamma$ with finite order. In \cite[\S 4.1]{LottII}, Lott describes a cochain complex $\mathscr{C}_\gamma(\Gamma)$ computing the group 
cohomology of $N_\gamma$ as follows: elements $c\in \mathscr{C}^k_\gamma(\Gamma)$ are given by 
antisymmetric maps $c:\Gamma^{k+1}\to\CC$ satisfying:
\begin{subequations}
\label{clc}
\begin{align}
\label{eq:c-z-invariance}c(z\gamma_0,z\gamma_1,\ldots,z\gamma_k)&=c(\gamma_0,\gamma_1,\ldots,\gamma_k),\qquad \forall z\in Z_\gamma\\
c(\gamma\gamma_0,\gamma_1,\ldots,\gamma_k)&=c(\gamma_0,\gamma_1,\ldots,\gamma_k).
\end{align}
\end{subequations}
The differential on this cochain complex is given by the usual formula
\[
(\delta c)(\gamma_0,\ldots,\gamma_{k+1})=\sum_{i=0}^{k+1}(-1)^ic(\gamma_0,\ldots,\hat{\gamma}_i,\ldots,\gamma_{k+1}).
\]
The proof that this complex indeed computes the group cohomology of $N_\gamma$ can be found in \cite{john-sheagan}. In \cite{LottII}, Lott also gives the following map to the cyclic complex $C^\bullet_\lambda(\CC\Gamma)$, which associates to $c\in \mathscr{C}^k_\gamma(\Gamma)$ the cochain
\[
\tau_c(\delta_{\gamma_0},\ldots,\delta_{\gamma_k}):=\begin{cases} 0&\mbox{if}~\gamma_0\cdots\gamma_k\not\in (\gamma)\\ c(\eta,\eta\gamma_0,\ldots,\eta\gamma_0\cdots\gamma_{k-1})&\mbox{if}~\gamma_0\cdots\gamma_k=\eta^{-1}\gamma\eta
\end{cases}
\]
(Remark that using \eqref{eq:c-z-invariance} one can prove easily that the right hand side does not depend on $\eta$; indeed if $\delta$ is another such element, then $\delta \eta^{-1}$ is an element in $Z_\gamma$.)
We denote this map by $\tau: \mathscr{C}^k_\gamma(\Gamma)\to C^\bullet_\lambda(\mathbb{C}\Gamma)$. 
A direct computation shows that
\[
\tau_{\delta c}=b \tau_c.
\]
This construction gives us cyclic cocycles on $\mathbb{C}\Gamma$, localized at the conjugacy class of $\gamma$. 

\subsection{Proper actions}

Let $X$ be a topological space and let $G$ be a locally compact topological group
acting by homeomorphisms on $X$. The action is defined to be proper if the map
$G\times X\to X\times X$, $(g,x)\to (x, gx)$ is a proper map. In this article we shall
assume that $X$ is a smooth manifold, denoted $M$, $G$ is discrete, usually denoted by $\Gamma$, and that the action of $\Gamma$ on $M$ is by  diffeomorphisms.\\ 
Let us give some examples. Let  $\Gamma\leq \RR^n\rtimes O(n)$ be a crystallographic group; thus there exists a finite subgroup $F\leq O(n)$ such that $\Gamma$ is isomorphic to $\mathbb{Z}^n \rtimes F$. Then $\Gamma$ acts properly and cocompactly, but not freely, on $\RR^n$. More generally if $(M,g)$ is a Riemannian manifold and $\Lambda\leq {\rm Isom}(M)$ is a discrete subgroup acting freely, properly and cocompactly, then for each finite group of isometries $F$ normalizing $\Lambda$ we have that $\Gamma:= \Lambda\rtimes F$ acts properly and cocompactly, but not freely, on $M$.
More examples of $\Gamma$-proper manifolds are obtained as follows: consider for  a moment $G$ equal to an almost connected Lie group
and let 
$M$ be a proper $G$-manifold \footnote{As far as higher index theory is concerned this situation has been
studied thoroughly in  \cites{Hochs-Wang, Hochs-Song-Tang, PP1, ppt-jdg, Song-Tang}.}. By Abels’ slice theorem, there exists a $K$-invariant submanifold $S \subset M$ such that
$M \;\simeq\; G \times_K S$, with $G$ acting by left translation on the first factor. This result
can also be 
used in order to {\em construct} $G$-proper manifolds starting from a compact $K$-manifold $S$.
Now let $\Gamma \subset G$ be a discrete subgroup which is not torsion free. Then $\Gamma$ acts on $M$ by restriction of the $G$-action. Since the $G$–action on $M$ is proper and $\Gamma$ is a closed subgroup of $G$, the restricted $\Gamma$–action on $M$ is also proper. In particular, $M$ is a proper $\Gamma$–manifold. Special cases of this construction associated with locally symmetric spaces have been given in the Introduction. In all these examples, if $\gamma \in \Gamma$ is a torsion element, then the fixed-point set $M^\gamma$ is (in general) nonempty, so the $\Gamma$-action on $M$ need not be free and the quotient $\Gamma \backslash M$ is typically an orbifold rather than a manifold.

\subsection{The fundamental diagram} 
Let us go back to the general case and  assume  therefore that $\Gamma$ is a discrete 
group acting properly and cocompactly on an oriented manifold $M$. We fix a $\Gamma$-invariant metric $g$.
We introduce the algebra $\mathcal{A}^c_\Gamma(M)$ of smoothing kernels $A\in C^\infty(M\times M)$ such that 
\begin{itemize}
\item[$i)$] $A$ is $\Gamma$-invariant:
\[
A(\gamma\cdot x,\gamma\cdot y)=A(x,y),
\]
\item[$ii)$]  $A$ has compact $\Gamma$-support: ${\rm sup}(A)\slash \Gamma$ is a compact subset of $(M\times M)\slash \Gamma$.
\end{itemize}
The algebra structure on $\mathcal{A}^c_\Gamma(M)$ is defined as
\[
(A_1 * A_2)(x,y):=\int_MA_1(x,z)A_2(z,y)dz,
\]
where $dz=d{\rm vol}_g(z)$ is the volume form associated to the metric $g$.
To show that this product is well-defined, we first remark that for a proper, cocompact action of $\Gamma$ on $M$ there exists a `cut-off function' $\chi\in C^\infty_c(M)$ satisfying 
\begin{equation}
\label{cut-off}
\sum_{\gamma\in\Gamma}\chi(\gamma^{-1}\cdot x)=1,\qquad\mbox{for all}~x\in M.
\end{equation}
We then insert this identity, change integration variable and use invariance of the kernels to rewrite 
\begin{align*}
\int_MA_1(x,z)A_2(z,y)dz&=\sum_\gamma\int_MA_1(x,z)\chi(\gamma^{-1}z)A_2(z,y)dz\\
&=\sum_\gamma\int_MA_1(x,\gamma z)\chi(z)A_2(\gamma z,y)dz\\
&=\sum_\gamma\int_MA_1(\gamma^{-1}x,z)\chi(z)A_2(z,\gamma^{-1}y)dz.
\end{align*}
Written this way, the integration is over the compact set ${\rm Supp}(\chi)\subset M$ and converges uniformly. The fact that the action of $\Gamma$ is proper implies 
that, for fixed $x,y\in M$, the summation above is finite. 

The aim of this section is to embed the morphism $\tau: \mathscr{C}^k_\gamma(\Gamma)\to C^\bullet_\lambda(\mathbb{C}\Gamma)$ of the previous section 
into a diagram
\begin{equation}
\label{diagram-nc}
\begin{tikzcd}
\mathscr{C}^\bullet_\gamma(\Gamma) \arrow[rr, "\tau"] \arrow[dd, "\Psi_{\rm inv}"']  &                                                                                            & {C^\bullet_\lambda(\mathbb{C}\Gamma)} \arrow[dd, "\Phi"] \\
                                                                                                       &  &                                                     \\
{C^\bullet_{\rm AS, inv}(M,\gamma)} \arrow[rr, "\rho_{\rm inv}"]                      &                                                                                            & C^\bullet_\lambda(\mathcal{A}^c_\Gamma(M))                
\end{tikzcd}
\end{equation}
Here the cochain complex $C^\bullet_{\rm AS, inv}(M,\gamma)$ of invariant $\gamma$-localized Alexander--Spanier cochains will be explained below. 

Let us remark already however that the diagram above does {\em not} commute, an issue that will be taken up in the next section where it is shown that it does commute {\em up to homotopy}. 

To set up the diagram, we now first define the morphism $\Phi:C^\bullet_\lambda(\mathbb{C}\Gamma)\to C^\bullet_\lambda(\mathcal{A}^c_\Gamma(M))$ by
\begin{equation}
\label{morphism-cohomology}
\Phi(\phi)(A_0,\ldots,A_k)=\sum_{\gamma_0,\ldots,\gamma_k}\phi(\delta_{\gamma_0},\ldots,\delta_{\gamma_k})\int_{M^{\times(k+1)}}\chi(x_0)A_0(x_0,\gamma_0\cdot x_1)\cdots\chi(x_k)A_k(x_k,\gamma_k\cdot x_0)dx_0\cdots dx_k,
\end{equation}
where $\phi\in C^k_\lambda(\CC\Gamma)$.
 Remark that, because the kernels $A_i$ have compact $\Gamma$-support, the above sum is finite.
We then have, by a straightforward computation:
\begin{lemma}
The formula above defines a morphism of complexes $\Phi:C^\bullet_\lambda(\mathbb{C}\Gamma)\to C^\bullet_\lambda(\mathcal{A}^c_\Gamma(M))$:
\[
b\circ\Phi=\Phi\circ b.
\]
\end{lemma}
\begin{proof}
The proof is standard and therefore omitted.
\end{proof}

We then come to the lower left corner of the diagram. We define
\begin{definition} 
The \emph{delocalized invariant Alexander--Spanier complex} $C^\bullet_{\rm AS, inv}(M,\gamma)$ is given by the vector space of smooth functions $f\in C^{\infty}(M^{\times (k+1)})$
satisfying
\begin{itemize}
\item $f$ is antisymmetric,
\item $f(zx_0,\ldots,zx_k)=f(x_0,\ldots,x_k),~\forall z\in Z_\gamma$,
\item $f(\gamma x_0,x_1,\ldots,x_k)=f(x_0,\ldots,x_k)$.
\end{itemize}
The differential is given by the usual formula
\[
(\delta f)(x_0,\ldots,x_{k+1}):=\sum_{i=0}^{k+1}(-1)^if(x_0,\ldots,\hat{x}_i,\ldots,x_{k+1}).
\]
\end{definition}
We can now define the two remaining maps. First, $\Psi_{\rm inv}: \mathscr{C}^\bullet_\gamma(\Gamma)\to C^\bullet_{\rm AS, inv}(M,\gamma)$ is defined as
\begin{equation}
\label{eq:psiinv}
\Psi_{\rm inv}(c)(x_0,\ldots,x_k):=\sum_{\gamma_0,\ldots,\gamma_k}c(\gamma_0,\ldots,\gamma_k)\chi(\gamma_0^{-1}x_0)\cdots\chi(\gamma_k^{-1}x_k),
\end{equation}
and it is once again straightforward to verify that 
\[
\delta\circ \Psi_{\rm inv}=\delta\circ \Psi_{\rm inv}.
\]
Finally $\rho_{\rm inv}:C^\bullet_{\rm AS, inv}(M,\gamma)\to C^\bullet_\lambda(\mathcal{A}^c_\Gamma(M))$ is defined as
\begin{align}\label{eq:inv com}
\rho_{\rm inv}(f)(A_0,\ldots,A_k):=\sum_{\eta\in\langle\gamma\rangle}\int_{M^{\times (k+1)}}\chi(x_0)f(x_0,\ldots x_k)A_0(x_0,x_1)\cdots A_k(x_k,\eta x_0)dx_0\cdots dx_k
\end{align}
and a direct computation shows that $\rho_{\rm inv}$ is a morphism of cochain complexes,
\[
\rho_{\rm inv}\circ\delta=b\circ \rho_{\rm inv}.
\]
This completes the definitions of all the ingredients in the diagram \eqref{diagram-nc}. Despite the seemingly straightforward definitions, it is important to notice that the diagram 
{\em does not commute}. Indeed already for a general degree zero cochain $c\in \mathscr{C}_\gamma^0(\Gamma)$ we get
\begin{align*}\Phi(\tau_c)(A)&=\sum_{\eta\in Z_\gamma \backslash\Gamma}c(\eta)\int_M\chi(x)A(x,\eta^{-1}\gamma\eta x)dx,
\end{align*}

whereas
\begin{align*}
\rho_{\rm inv}(\Psi_{\rm inv}(c))(A)&=\sum_{\eta\in\langle\gamma\rangle}\int_{M}\chi(x)\Psi_{\rm inv}(c)(x)A(x,\eta x)dx\\
&=\sum_{\gamma_0\in\Gamma}\sum_{\eta\in\langle\gamma\rangle}c(\gamma_0)\int_{M}\chi(x) \chi(\gamma_0^{-1}x)A(x,\eta x)dx.
\end{align*}
Clearly, the two expressions are different, showing that the diagram indeed does not commute. However notice that imposing the cocycle condition $\delta c=0$ forces $c$ to be constant and we can use the identity \eqref{cut-off} to see that the two expressions do agree on the level of cocycles. This small observation strongly suggests that the two maps $\Phi\circ \tau$ and $\rho_{\rm inv}\circ \Psi_{\rm inv}$, from the left upper to the right lower corner of the diagram \eqref{diagram-nc}, are homotopic to each other. In the next section we shall construct this homotopy.  
\subsection{The extended Alexander--Spanier complex}
In this section we shall discuss the commutativity property of the diagram (\ref{diagram-nc})
as follows: we define a new cochain complex ${\mathsf{E}^\bullet_{\rm AS}(M,\gamma)}$, called the {\em $\gamma$-extended  Alexander--Spanier complex}, and add it to the diagram \eqref{diagram-nc} as follows:
\begin{equation}
\label{diagram}
\begin{tikzcd}
\mathscr{C}^\bullet_\gamma(\Gamma) \arrow[rr, "\tau"] \arrow[dd, "\Psi_{\rm inv}"'] \arrow[rd, "\Psi"] &                                                                                            & {C_\lambda^\bullet(\mathbb{C}\Gamma) } \arrow[dd, "\Phi"] \\
                                                                                                       & {\mathsf{E}^\bullet_{\rm AS}(M,\gamma)}\arrow[rd, "\rho"] \arrow[ld, "P"', shift right=2] &                                                     \\
{C^\bullet_{\rm AS, inv}(M,\gamma)} \arrow[rr, "\rho_{\rm inv}"] \arrow[ru, "I"']                      &                                                                                            & C_\lambda^\bullet(\mathcal{A}^c_\Gamma(M))                
\end{tikzcd}
\end{equation}

We will define the chain maps in the above diagram and prove the following theorem.

\begin{theorem}\label{thm:comm-diagram} The three triangles in (\ref{diagram}) commute. Moreover, there is a chain homotopy equivalence $H$ between $C^\bullet_{\rm AS, inv}(M,\gamma)$ and ${\mathsf{E}^\bullet_{\rm AS}(M,\gamma)}$ satisfying
\[
P\circ I={\rm Id},\qquad 
I\circ P={\rm Id} +\delta\circ H+H\circ \delta.
\]
\end{theorem}

The remainder of this subsection is devoted to the definitions of the $\gamma$-extended Alexander-Spanier complex and the related maps in diagram (\ref{diagram}), and to the proof of Theorem \ref{thm:comm-diagram}. Before that, we mention the following corollary of Theorem \ref{thm:comm-diagram}.
\begin{corollary}
The outer rectangle of diagram (\ref{diagram}), which is diagram (\ref{diagram-nc}), commutes up to homotopy. 
\end{corollary}
 
\begin{proof} Using the maps $I, P, H$ in Theorem \ref{thm:comm-diagram}, we have the following equalities, 
\begin{align*}
\rho_{\rm inv}\circ\Psi_{\rm inv}-\Phi\circ\tau&=\rho\circ I\circ P\circ\Psi-\Phi\circ\tau\\
&=\rho\circ(H\circ\delta+\delta\circ H+{\rm id})\circ \Psi-\Phi\circ\tau\\
&=\rho\circ H\circ\Psi\circ\delta+\delta\circ\rho\circ H\circ \Psi+\rho\circ\Psi-\Phi\circ\tau\\
&=H'\circ\delta+\delta\circ H',
\end{align*}
with 
\begin{equation}
\label{homotopy-diagram}
H':=\rho\circ H\circ\Psi:\mathscr{C}^\bullet_\gamma(\Gamma)\to C_\lambda^{\bullet-1}(\mathcal{A}^c_\Gamma(M)).
\end{equation}
This completes the proof the corollary. 
\end{proof}

We now proceed to define the center piece of the diagram above.
Let $p_0: \Gamma \times M\times M\times \cdots \times M\to  M$ be defined by mapping $(\eta, x_0, \cdots, x_k)$ to $\eta^{-1}x_0$.
\begin{definition}
\label{defn:gammaASch} 
For $\gamma\in \Gamma$, a $\gamma$-extended Alexander-Spanier 
cochain is a function 
\[
f\in \mathsf{E}_{\rm AS}^k(M, \gamma):=\{F\in C^\infty \big(\Gamma\times M^{\times (k+1)} \big),~ p_0\big(\operatorname{supp}(F)\big)\ \text{is compact.}\},
\]
satisfying the following conditions
\begin{enumerate}
\item $f(z\eta, x_0, \cdots, x_k)=f(\eta, z^{-1}x_0, z^{-1}x_1, \cdots, z^{-1}x_{k})$, $\qquad \forall z\in Z_\gamma$
\item $f(\eta, x_0, \cdots, x_k)=f(\gamma\eta, \gamma x_0, x_1, \cdots, x_k)$.
\end{enumerate}
The differential $\epsilon_{\rm AS} : \mathsf{E}^k_{\rm AS}(M, \gamma) \to \mathsf{E}^{k+1}_{\rm AS}(M, \gamma)$ is defined as follows.
\[
\epsilon_{\rm AS}(f)(\eta, x_0, \cdots, x_{k+1}):= \sum_h f(\eta h, x_1, \cdots, x_{k+1})\chi(\eta^{-1}x_0)+\sum_{i\geq 1}(-1)^i f(\eta, x_0, \cdots, \hat{x}_i, \cdots, x_{k+1}).
\]
For $f\in \mathsf{C}_{\rm AS}^k(M, \gamma)$, set 
$E_f(x_0, \cdots, x_k):=\sum_{\eta\in  \Gamma} f(\eta, x_0, \cdots, x_k)$. With the support requirement of $f$ and the assumption that the $\Gamma$-action is proper, we see that $E_f$ is a well defined function on $M^{\times (k+1)}$. We can restrict to consider those $f$ such that $E_f$ is antisymmetric. We denote the subspace of $\mathsf{C}_{\rm AS}^k(M, \gamma)$ consisting of such antisymmetric cochains by $\mathsf{E}_{\rm AS, a}^k(M, \gamma)$. $(\mathsf{E}_{\rm AS, a}^k(M, \gamma), \delta_{\rm AS})$ is a subcomplex of $(\mathsf{E}_{\rm AS}^k(M, \gamma), \epsilon_{\rm AS})$. 
\end{definition}

Let us explain why $\epsilon_{\rm AS} : \mathsf{E}^k_{\rm AS}(M, \gamma) \to \mathsf{E}^{k+1}_{\rm AS}(M, \gamma)$ is well defined.
Recall that in the definition of $\mathsf{E}^k_{\rm AS}(M, \gamma)$, we have assumed that the image of the support of $f\in \mathsf{E}^k_{\rm AS}(M, \gamma)$ under the map $p_0$ is compact. In the first component of $\epsilon_{\rm AS}(f)$, to have a nontrivial contribution, one needs that $h^{-1}\eta ^{-1}x_1$ belongs $p_0(\text{supp}(f))$. As the action is $\Gamma$ action on $X$ is proper, there are in total a finite number of such $h$. This observation assures that the first component of $\epsilon_{\rm AS}(f)$ is a finite sum. Next we explain that $\epsilon_{\rm AS}(f)$ itself satisfies the support requirement. We start with observing that $p_0$ maps the support of the following function
\[
\sum_h f(\eta h, x_1, \cdots, x_{k+1})\chi(\eta^{-1}x_0)
\]
into the support of the function $\chi$, which is a compact set. We then notice that $p_0$ maps the support of the second component of $\epsilon_{\rm AS}(f)$ into $p_0(\text{supp}(f))$, which is compact. Hence, the sum to the two terms in $\epsilon_{\rm AS}(f)$ satisfies the right support condition.  And we conclude that the boundary map $\epsilon_{\rm AS}$ is well defined. 

\begin{lemma}\label{lem:d2} $\epsilon_{\rm AS}^2=0$. 
\end{lemma}

\begin{proof}
We check the identity directly:
\begin{subequations}
\begin{eqnarray}\label{eq:delta2}
&&\epsilon_{\rm AS}^2 (f)(\eta, x_0, \cdots, x_{k+2})\\
&=&\sum_h \epsilon_{\rm AS}(f)(\eta h, x_1, \cdots, x_{k+2})\chi(\eta^{-1}x_0)+\sum_{i\geq 1}(-1)^i \epsilon_{\rm AS}(f)(\eta, x_0, \cdots, \hat{x}_i, \cdots, x_{k+2})\\
\label{eq:delta2term1}
&=& \sum_{h_1, h_2}f(\eta h_1h_2, x_2, \cdots, x_{k+2})\chi(h_1^{-1}\eta^{-1}x_1)\chi(\eta^{-1}x_0)\\
\label{eq:delta2term2}
&+&\sum_{i\geq 2}(-1)^{i-1} f(\eta h, x_1, \cdots, \hat{x}_i, \cdots, x_{k+2}) \chi(\eta^{-1}x_0)\\
\label{eq:delta2term3}
&+&\sum_{i\geq 1} (-1)^i \sum_h f(\eta h, x_1, \cdots, \hat{x}_i, \cdots, x_{k+2})\chi(\eta^{-1}x_0)\\
\label{eq:delta2term4}
&+&\sum_{1\leq i<j}(-1)^{i+j-1}f(\eta, x_0, \cdots, \hat{x}_i, \cdots, \hat{x}_j, \cdots, x_{k+1})\\
\label{eq:delta2term5}&+&\sum_{1\leq j<i}(-1)^{i+j}f(\eta, x_0, \cdots, \hat{x}_i, \cdots, \hat{x}_j, \cdots, x_{k+1})
\end{eqnarray}
\end{subequations}
In the above sum, the two terms (\ref{eq:delta2term4}) and (\ref{eq:delta2term5}) cancel because of the sign difference. The third term (\ref{eq:delta2term3}) can be split into a sum of two terms
\[
-\sum_h f(\eta h, x_2, \cdots, x_{k+2})\chi(\eta^{-1}x_0)+\sum_{i\geq 2} \sum_h f(\eta h, x_1, \cdots, \hat{x}_i, \cdots, x_{k+2})\chi(\eta^{-1}x_0). 
\]
The second term in the above equation cancels with the second term (\ref{eq:delta2term2}) in the expression of $\delta_{\rm AS}^2(f)$. And the first term in the above equation cancels with the first term (\ref{eq:delta2term1}) in the expression of $\delta^2_{\rm AS}(f)$ once we make the following observation, 
\[
\begin{split}
&\sum_{h_1, h_2}f(\eta h_1h_2, x_2, \cdots, x_{k+2})\chi(h_1^{-1}\eta^{-1}x_1)\chi(\eta^{-1}x_0)\\
=&\sum_{h'_1}f(\eta h'_1), x_2, \cdots, x_{k+2}) \chi(\eta^{-1}x_0)\sum_{h_1} \chi(h_1^{-1}x_1)\\
=&\sum_{h'_1}f(\eta h'_1), x_2, \cdots, x_{k+2}) \chi(\eta^{-1}x_0),
\end{split}
\]
where we have used the defining property for the cutoff function $\chi$, i.e.
\[
\sum_{h_1} \chi(h_1^{-1}x)=1,\ \ \forall x. 
\]
Summarizing the above discussion on (\ref{eq:delta2term1})-(\ref{eq:delta2term5}), we conclude that $\delta^2_{\rm AS}(f)=0$. 
\end{proof}

\begin{definition}\label{defn:ascoh} The cohomology of $(\mathsf{E}^\bullet_{\rm AS,a}(M, \gamma), \epsilon_{\rm AS})$ is denoted by $H_{\rm AS,a}^\bullet(M, \gamma)$ and called the $\gamma$-extended Alexander-Spanier cohomology. 
\end{definition}
The maps $P$ and $I$ in diagram \eqref{diagram} are given by the formulae
\[
(If)(\eta,x_0,\ldots,x_k):= f(x_0,\ldots,x_k)\chi(\eta^{-1}(x_0)),
\]
and 
\[
(PF)(x_0,\ldots,x_k) := \sum_{h\in\Gamma} F(h,x_0,\ldots,x_k).
\]
It is straightforward to check that both $P$ and $I$ commute with the differentials. It follows from the identity \eqref{cut-off} that $P\circ I={\rm id}$. 
The opposite composition writes out as
\[
(I\circ P)(F)(\eta,x_0,\ldots, x_k)=\sum_{h\in\Gamma} F(h,x_0,\ldots,x_k)\chi(\eta^{-1}x_0).
\]
We now define the map $H:\mathsf{E}^k_{\rm AS}(M,\gamma)\to \mathsf{E}^{k-1}_{\rm AS}(M,\gamma)$ by
\[
HF(\eta,x_0,\ldots,x_{k-1}):=\sum_{i=0}^{k-1} (-1)^iF(\eta,x_0,\ldots,x_i,x_i,\ldots,x_{k-1}).
\]
\begin{lemma}\label{lem:IPH}
The following identity holds true: 
\[
\epsilon_{\rm AS}\circ H+H\circ \epsilon_{\rm AS}=I\circ P-{\rm id}.
\]
\end{lemma}
\begin{proof}
We start computing on the left hand side:
\begin{align*}
\epsilon_{\rm AS}(H(F))(\eta,x_0,\ldots,x_k)=&\sum_h(HF)(\eta h,x_1,\ldots,x_{k})\chi(\eta^{-1}x_0)+\sum_{i\geq 1}(-1)^i (HF)(\eta,x_1,\ldots,\hat{x}_i,\ldots,x_{k-1})\\
=&\sum_h\sum_{j=1}^{k} (-1)^{j+1}F(\eta h,x_1,\ldots,x_j,x_j,\ldots,x_{k})\chi(\eta^{-1}x_0)\\
&
+\sum_{i\geq 1\atop 1\leq j<i}(-1)^{i+j}F(\eta,x_0,\ldots,x_j,x_j,\ldots,\hat{x}_i,\ldots,x_k)\\
&
+\sum_{i\geq 1\atop i< j\leq k}(-1)^{i+j-1}F(\eta,x_0,\ldots,\hat{x}_i,\ldots,x_j,x_j,\ldots,x_k)
\end{align*}
The second term on the left hand side is
\begin{align*}
H(\epsilon_{\rm AS}(F))(\eta,x_0,\ldots,x_k)=&\sum_{j=0}^k(-1)^j(\delta F)(\eta, x_0,\ldots,x_j,x_j,\ldots,x_k)\\
=& \sum_h F(\eta h,x_0,\ldots,x_k)\chi(\eta^{-1}x_0)+\sum_h\sum_{j=1}^k(-1)^jF(\eta h,x_1,\ldots,x_j,x_j,\ldots,x_k)\chi(\eta^{-1}x_0)\\
&+\sum_{0\leq j\leq k\atop 1\leq i<j} (-1)^{i+j}F(\eta,x_0,\ldots,\hat{x}_i,\ldots,x_j,x_j,\ldots,x_k)\\
&+\sum_{j=1}^k(-1)^{2j}F(\eta,x_0,\ldots,x_k)+\sum_{j=0}^k(-1)^{(2j+1)}F(\eta,x_0,\ldots,x_k)\\
&+\sum_{0\leq j\leq k\atop j <i\leq k} (-1)^{i+j+1}F(\eta,x_0,\ldots,x_j,x_j,\ldots,\hat{x}_i,\ldots,x_k)
\end{align*}
When we add to form $(\epsilon_{\rm AS}\circ H+H\circ \epsilon_{\rm AS})(F)$ all terms from the second part $H(\epsilon_{\rm AS}(F))$ after the last equality sign  cancel against all terms in the first part $\epsilon_{\rm AS}(H(F))$ except for the very first which is exactly $I(P(F))$ and one single term in the third line which equals $-F$. Together, this proves the identity of the lemma.
\end{proof}
To complete the diagram \eqref{diagram}, we have to define the maps $\Psi$ and $\rho$. To define $\Psi$, we associate to $c\in \mathscr{C}^k_\gamma(\Gamma)$ the 
following function
\begin{equation}\label{eq:fc}
f_c(\gamma_0, x_0,\ldots,x_k)=\sum_{\gamma_1,\ldots,\gamma_k}c(\gamma_0,\ldots,\gamma_k)\chi(\gamma_0^{-1}x_0)\cdots\chi(\gamma_k^{-1}x_k)
\end{equation}

\begin{lemma}\label{lem:fc}The function $f_c$ introduced in Eq. (\ref{eq:fc}) is a chain in $ \mathsf{E}^k_{\rm AS,a}(M,\gamma)$. 
\end{lemma}

\begin{proof}We notice from (\ref{eq:fc}) that in order for $f_c(\eta, x_0, \cdots, x_k)$ to be nonzero, $\eta^{-1}x_0$ needs to be in the support of the cut-off function $\chi$. This assures that $p_0$ maps the support of $f_c$ into a compact set. 

We now check that $f_c$ satisfies both (1) and (2) in Definition \ref{defn:gammaASch}. 

To check property (1), we make the following computation: For $z \in Z_\gamma$, 
\begin{eqnarray*}
f_c(z\eta, z x_0, \cdots, z x_k)&=&\sum _{\gamma_1, \cdots, \gamma_k} c(z\eta, \gamma_1, \cdots, \gamma_k)\chi((z\eta)^{-1}z x_0)\chi(\gamma_1^{-1}z x_1)\cdots \chi(
\gamma_k^{-1}z x_k)\\
&=&\sum_{\gamma_1, \cdots, \gamma_k}c(\eta, z^{-1}\gamma_1, \cdots, z^{-1}\gamma_k) \chi(\eta^{-1}x_0)\chi((z^{-1}\gamma_1)^{-1}x_1)\cdots \chi((z^{-1}\gamma)^{-1}x_k)\\
&=&\sum_{\gamma'_1, \cdots,\gamma'_k}c(\eta, \gamma'_1, \cdots, \gamma'_k)\chi(\eta^{-1}x_0)\chi(\gamma'_1x_1)\cdots \chi(\gamma'_kx_k)\\
&=&f_c(\eta, x_0, \cdots, x_k),
\end{eqnarray*}
where in the second equality we have use the property of $c$  in (\ref{clc}a).
To check property (2), we make the following computation. 
\begin{eqnarray*}
f_c(\gamma \eta, x_0, \cdots, x_k)&=&\sum_{\gamma_1, \cdots, \gamma_k} c(\gamma \eta, \gamma_1, \cdots, \gamma_{k}) \chi(\eta^{-1}\gamma ^{-1} x_0)\chi(\gamma^{-1}_1x_1)\ldots \chi((\gamma^{-1} _k x_k)\\
&=& \sum_{\gamma_1, \cdots, \gamma_{k}} c(\eta, \gamma_1, \cdots, \gamma_{k}) \chi(\eta^{-1}\gamma ^{-1} x_0) \chi(\gamma_1^{-1}x_1)\ldots \chi((\gamma_{k})^{-1}x_k)\\
&=&f_c(\eta, \gamma^{-1}x_0, x_1, \cdots, x_k),
\end{eqnarray*} 
where in the second equality we have used the property of $c$ in (\ref{clc}b). 
Finally, the antisymmetry property of $f_c$ follows from direct computation, which is left to the reader. 
\end{proof}

\begin{lemma}\label{lem:psi} The map $\Psi: \mathscr{C}_\gamma^k(\Gamma)\to \mathsf{E}^k_{\rm AS, a}(M, \gamma)$ associating $f_c$ to $c$ is a morphism of cochain complexes. 
\end{lemma}

\begin{proof} 
We have to check that $\Psi(\delta c)=\epsilon_{\rm AS}\big(\Psi(c)\big)$ for all $c\in \mathscr{C}_\gamma^k(\Gamma)$.
We recall that $\Psi(\delta c)$ has the following expression, 
\begin{eqnarray*}
\Psi(\delta c)(\gamma_0, x_0, \cdots, x_{k+1})&=&\sum_{\gamma_1, \cdots, \gamma_{k+1}} \delta c(\gamma_0, \cdots, \gamma_{k+1}) \chi(\gamma_0^{-1}x_0)\cdots \chi(\gamma_{k+1}^{-1} x_{k+1})\\
&=&\sum_{i=0}^{k+1}\sum_{\gamma_1, \cdots, \gamma_{k+1}}(-1)^i c(\gamma_0, \cdots, \hat{\gamma}_i, \cdots, \gamma_{k+1})\chi(\gamma_0^{-1}x_0)\cdots \chi(\gamma_{k+1}^{-1} x_{k+1})\\
&=&\sum_{\gamma_1, \cdots, \gamma_{k+1}}c(\gamma_1, \cdots, \gamma_{k+1})\chi(\gamma_0^{-1}x_0)\cdots \chi(\gamma_{k+1}^{-1} x_{k+1})\\
&+&\sum_{i=1}^{k+1}\sum_{\gamma_1, \cdots,\hat{\gamma}_i,\cdots, \gamma_{k+1}}\sum_{\gamma_0}(-1)^i c(\gamma_0, \cdots, \hat{\gamma}_i, \cdots, \gamma_{k+1})\chi(\gamma_0^{-1}x_0)\cdots \chi(\gamma_{k+1}^{-1} x_{k+1}).
\end{eqnarray*}
With $\Psi(c)$ given as in \eqref{eq:fc}, $\epsilon_{\rm AS}\Psi (c)$ has the following expression:
\begin{align*}
\epsilon_{\rm AS}(\Psi(c))&(\eta, x_0, \cdots, x_{k+1})\\=&\sum_{h}\Psi(c)(\eta h, x_1, \cdots, x_{k+1})\chi(\eta^{-1}x_0)
+\sum_{i\geq 1}(-1)^i \Psi(c)(\eta, x_0, \cdots, \hat{x}_i, \cdots, x_{k+1})\\
=& \sum_h \sum_{\gamma_2, \cdots, \gamma_k} c(\eta h, \gamma_2, \cdots, \gamma_{k+1}) \chi(\eta^{-1}x_0)\chi(h^{-1}\eta^{-1} x_1) \chi(\gamma^{-1}_2x_2)\cdots \chi(\gamma^{-1}_{k+1}x_{k+1})\\
&+\sum_{i\geq 1}\sum_{\gamma_1, \cdots,\hat{\gamma}_i, \cdots, \gamma_{k+1}} (-1)^ic(\eta, \gamma_1, \cdots, \hat{\gamma}_i, \cdots, \gamma_{k+1})\chi(\eta^{-1}x_0) \chi(\gamma_1^{-1}x_1)\cdots\widehat{\chi(\gamma_i^{-1}x_i)}\cdots \chi(\gamma^{-1}_{k+1}x_{k+1})\\
=& \sum_{\gamma_1, \gamma_2, \cdots, \gamma_k} c(\gamma_1, \gamma_2, \cdots, \gamma_{k+1}) \chi(\eta^{-1}x_0)\chi(\gamma^{-1}_1 x_1) \chi(\gamma^{-1}_2x_2)\cdots \chi(\gamma^{-1}_{k+1}x_{k+1})\\
&+\sum_{i\geq 1}\sum_{\gamma_1, \cdots,\hat{\gamma}_i, \cdots, \gamma_{k+1}} (-1)^ic(\eta, \gamma_1, \cdots, \hat{\gamma}_i, \cdots, \gamma_{k+1})
\chi(\eta^{-1}x_0) \chi(\gamma_1^{-1}x_1)\cdots\widehat{\chi(\gamma_i^{-1}x_i)}\cdots \chi(\gamma^{-1}_{k+1}x_{k+1}).\\
\end{align*}
In the third equality, we have changed the variable $\gamma_1=\eta h$. Using equality \eqref{cut-off}, we conclude that $\Psi (\delta c)=\epsilon_{\rm AS}(\Psi(c))$, showing that
$\Psi$ is indeed a morphism of cochain complexes.
\end{proof}
Finally, to complete the diagram (\ref{diagram}), we define  $\rho: \mathsf{E}^\bullet_{\rm AS, a}(M, \gamma)\to C^\bullet_\lambda(\mathcal{A}^c_\Gamma(M))$ as follows, 
\begin{equation}
\label{equation:rho-AS}
\rho(f)(A_0,\dots,A_k):= \sum_{\nu\in Z_\gamma\backslash \Gamma} \int_{M^{k+1}} f(\nu, x_0, \dots, x_k) A_0 (x_0,x_1)
\cdots A_{k-1}(x_{k-1},x_k) A_k (x_k, \gamma x_0),
\end{equation}
for $f\in \mathsf{E}^\bullet_{\rm AS, a}(M, \gamma)$.
By the $\Gamma$-invariance of the kernels $A_i(x,y)$, $i=0, \cdots, k$, we observe that 
\[
A_0(x_0, x_1)\cdots A_{k-1}(x_{k-1}, x_k)A_k(x_k, \gamma x_0)=A_0(\nu^{-1}x_0, \nu^{-1}x_1)\cdots A_{k-1}(\nu^{-1}x_{k-1}, \nu^{-1}x_k)A_k(\nu ^{-1}x_k, \nu^{-1}\gamma \nu \nu^{-1}x_0);
\]
thus the integral  in \eqref{equation:rho-AS} is independent of the choice of $\nu$.
In the definition of $\mathsf{E}^\bullet_{\rm AS}(M, \gamma)$, we have assumed that  for $f\in \mathsf{E}^\bullet_{\rm AS}(M, \gamma)$, $p_0(\text{supp}(f))$ is compact. We can conclude that the above infinite sum over $\nu$ in the definition of $\rho(f)$ is actually finite when the propagation of each $A_i$ ($i=0, \cdots, k$) is finite. Therefore the above pairing $\rho(f)(A_0, \cdots, A_k)$ is well defined. 
\begin{lemma}
\label{lem:tau-chain} 
The map $\rho: \mathsf{E}^\bullet_{\rm AS, a}(M, \gamma)\to C^\bullet_\lambda(\mathcal{A}^c_\Gamma(M))$ is a morphism of cochain complexes.
\end{lemma}
\begin{proof} Again, we have to show that $\rho(\epsilon_{\rm AS}(f))=\delta(\rho(f))$ for all $f\in \mathsf{E}^\bullet_{\rm AS, a}(M, \gamma)$. We compute $\rho(\epsilon_{\rm AS}(f))$ as follows,
\begin{eqnarray*}
\rho(\epsilon_{\rm AS}(f))(A_0, \cdots, A_{k+1})&=& \sum_{\nu\in Z_\gamma \backslash \Gamma}\int_{M^{k+2}} \epsilon_{\rm AS}(f)(\nu, x_0, \cdots, x_{k+1})A_0(x_0, x_1)\cdots A_{k+1}(x_{k+1}, \gamma x_0).
\end{eqnarray*}
Using the definition of $\epsilon_{\rm AS}(f)$, we continue computing $\rho(\epsilon_{\rm AS}(f))(A_0, \cdots, A_{k+1})$,
\begin{align*}
 \sum_{\nu\in Z_\gamma \backslash \Gamma}\int_{M^{k+2}} &\epsilon_{\rm AS}(f)(\nu, x_0, \cdots, x_{k+1})A_0(x_0, x_1)\cdots A_{k+1}(x_{k+1}, \gamma x_0)\\
=& \sum_{\nu\in Z_\gamma \backslash \Gamma}\int_{M^{k+2}}  \sum_{h}f(\nu h, x_1, \cdots , x_{k+1})\chi (\nu^{-1}x_0)A_0(x_0, x_1)\cdots A_{k+1}(x_{k+1}, \gamma x_0)\\
&+\sum_{\nu\in Z_\gamma\backslash \Gamma} \int_{M^{k+2}} \sum_{i\geq 1} (-1)^i f(\nu, x_0, \cdots, \hat{x}_{i}, \cdots, x_{k+1}) A_0(x_0,x_1)\cdots A_{k+1}(x_{k+1}, \gamma x_0)\\
=& \sum_{\nu\in Z_\gamma \backslash \Gamma}\int_{M^{k+2}}  \sum_{h}f(\nu , x_1, \cdots , x_{k+1})\chi (h^{-1}\nu^{-1}x_0)A_0(x_0, x_1)\cdots A_{k+1}(x_{k+1}, \gamma x_0)\\
&+\sum_{\nu\in Z_\gamma\backslash \Gamma} \int_{M^{k+1}} \sum_{i\geq 1} (-1)^i f(\nu, x_0, \cdots, x_i, \cdots, x_{k}) A_0(x_0,x_1)\\
&\qquad \cdots (A_i\ast A_{i+1})(x_i, x_{i+1})\cdots  A_{k+1}(x_{k}, \gamma x_0)\\
=& \sum_{\nu\in Z_\gamma \backslash \Gamma}\int_{M^{k+1}}  f(\nu, x_0, \cdots , x_{k})A_1(x_0, x_1)\cdots (A_{k+1}\ast A_0)(x_{k}, \gamma x_0)\\
&+\sum_{\nu\in Z_\gamma\backslash \Gamma} \int_{M^{k+1}} \sum_{i\geq 1} (-1)^i f(\nu, x_0, \cdots, \hat{x}_{i}, \cdots, x_{k+1}) A_0(x_0,x_1)\\
&\qquad\cdots (A_i\ast A_{i+1})(x_i, x_{i+1})\cdots A_{k+1}(x_{k}, \gamma x_0)\\
=&\delta(\rho(f))(A_0, \cdots, A_{k+1}),
\end{align*}
where in the third equality we have used the property that $\chi$ is a cut-off function. 
\end{proof}

In summary, using the results in Lemma \ref{lem:d2}, \ref{lem:IPH}, \ref{lem:fc}, \ref{lem:psi}, and \ref{lem:tau-chain}, we have completed the proof of Theorem \ref{thm:comm-diagram}.

\subsection{Cohomology computations}
We finally turn to the computations of the cohomologies of the complexes discussed above.
Consider the pair groupoid $M\times M\rightrightarrows M$ and observe that its $q$-th nerve space can be identified with $M^{\times (q+1)}$. It follows that $\{C^\infty(M^{\times (q+1)})\}_{q\geq 0}$ carries the structure of a cosimplicial complex. Let $C^{\infty}_{\gamma, a}(M^{\times (q+1)})$ be the subspace of $C^\infty(M^{\times (q+1)})$ of antisymmetric smooth functions on $M^{\times (q+1)}$ satisfying
\[
f(x_0, \ldots, x_q)=f(\gamma(x_0), \ldots, x_q). 
\]
We check that for all $f\in C^\infty_{\gamma, a}(M^{\times (q+1)})$, antisymmetry implies that
\[
f(x_0, \ldots, \gamma(x_i), \ldots, x_q)=
f(x_0, \ldots, x_i, \cdots, x_q)
\]
for all $1\leq i\leq q$, and therefore also 
\begin{equation}\label{eq:gamma-inv}
f(\gamma(x_0), \gamma(x_1), \ldots, \gamma(x_i), x_{i+1}, \ldots, x_q)=
f(x_0, x_1, \cdots, x_i, \cdots, x_q). 
\end{equation}
Using the above property (\ref{eq:gamma-inv}) of functions in $C^\infty_{\gamma, a}(M^{\times (q+1)})$, we check that $\{ C^\infty_{\gamma, a}(M^{\times (q+1)}) \}_{q\geq 0}$ indeed is a cosimplicial subcomplex of $\{C^\infty(M^{\times (q+1)})\}_{q\geq 0}$. 

Fix $m\in M$. Define a homotopy operator $H^q: C^\infty(M^{\times (q+1)})\to C^\infty(M^{\times q}) $ by 
\[
H^q(f)(x_0, \ldots, x_{q-1})=f(m, x_0, \ldots, x_{q-1}). 
\]
It is straightforward to check that $H^\bullet $ defines a deformation retract of  $\{C^\infty(M^{\times (q+1)})\}_{q\geq 0}$ to the space of constant functions on $M$. Furthermore, by (\ref{eq:gamma-inv}) we directly check that $H^q$ restricts to a deformation retract $H_{\gamma,a}^q: C^\infty_{\gamma, a}(M^{\times (q+1)})\to C^\infty_{\gamma, a}(M^{\times q})$. Therefore, the cochain complex $C^\bullet_{\gamma, a}(M^{\times (\bullet+1)})$ is acyclic. 

By restricting the action of $\Gamma$, the group $Z_\gamma$ acts on $M$ and accordingly acts diagonally on $M^{\times (q+1)}$. So the space $C^\infty(M^{\times (q+1)})$ is a $Z_\gamma$-module. Since $Z_\gamma$ commutes with $\gamma$, $C^\infty_{\gamma, a}(M^{\times (q+1)})$, as a subspace of $C^\infty(M^{\times (q+1)})$, is a $Z_\gamma$-submodule.  Furthermore, $Z_\gamma$ acts on the groupoid $M\times M\rightrightarrows M$ by groupoid automorphisms and makes $\{C^\infty(M^{\times (q+1)})\}_{q\geq 0}$ in a $Z_\gamma$-equivariant cosimplicial complex. 

Fix $q\geq 0$. Consider the group cochain complexes $C^\bullet(Z_\gamma, C^\infty(M^{\times (q+1)}))$ and $C^\bullet (Z_\gamma, C^\infty_{\gamma, a}(M^{\times (q+1)}))$. Since the $\Gamma$ action on $M$  is proper, the restricted $Z_\gamma$ action on $M^{\times (q+1)}$ is also proper. Using a cut-off function $c^q$ on the groupoid $Z_\gamma \ltimes M^{\times (q+1)}\rightrightarrows M^{\times (q+1)}$, in \cite[Proposition 1] {crainic} a deformation retract is constructed from $C^\bullet (Z_\gamma, C^\infty(M^{\times (q+1)}))$ to $ C^\infty(M^{\times (q+1)})^{Z_\gamma}$ as follows: define $H_\gamma: C^p(Z_\gamma,  C^\infty(M^{\times (q+1)}))\to C^{p-1}(Z_\gamma,  C^\infty(M^{\times (q+1)}))$ by
\[
\begin{split}
H_\gamma(\varphi)(\gamma_0, \cdots, \gamma_{p-1})(x_0, \cdots, x_{q}):=\sum_{\alpha\in Z_\gamma}& \varphi(\alpha^{-1}, \gamma_0, \cdots, \gamma_{p-1})(\alpha^{-1}(x_0), \alpha^{-1}(x_1), \cdots \alpha^{-1}(x_q))\\
&\cdot c^q(\alpha^{-1}(x_0), \cdots, \alpha^{-1}(x_q))
\end{split}
\]
Using the commuting property between $Z_\gamma$ and $\gamma$, we can directly check that $H_\gamma$ restricts to a deformation retract from $ C^\infty_{\gamma, a}(M^{\times (q+1)})$ to $ C^\infty_{\gamma, a}(M^{\times (q+1)})^{Z_\gamma}$.

\begin{proposition}\label{prop:psi-inv-quasi-iso}The cochain map $\Psi_{\rm inv}: \mathscr{C}^\bullet_\gamma(\Gamma)\to {\mathsf{C}^\bullet_{\rm AS, inv}(M,\gamma)} $ defined in \eqref{eq:psiinv} is a quasi-isomorphism. 
\end{proposition}

\begin{proof}
We consider the bicomplex with entries $D^{p,q}_\gamma:=C^p(Z_\gamma, C^\infty_{\gamma, a}(M^{\times (q+1)}))$ with differentials given by the group cohomology differential for the $Z_\gamma$-module $C^\infty_{\gamma, a}(M^{\times (q+1)})$ and the groupoid differential for the pair groupoid $M\times M\rightrightarrows M$.  

As explained above, the complex $ C^\infty_{\gamma, a}(M^{\times (\bullet+1)})$ is acyclic. Spectral sequence arguments show that the total cohomology of the  double complex $D^{\bullet,\bullet}$ is equal to the total cohomology of the $E^1$-page associated to the $q$-direction, which is computed as 
\[
E^1_{p,q}=\left\{\begin{array}{ll}0,&q\geq 1\\ C^{p}(Z_\gamma),&q=0\end{array}\right.
\]
So we conclude that the total cohomology of the double complex $D^{\bullet, \bullet}$ is given by $H^\bullet(Z_\gamma)$. 

On the other hand, for every $q\geq 0$, the cohomology of $ C^\bullet(Z_\gamma, C^\infty_{\gamma, a}(M^{\times (q+1)}))$ is concentrated in degree 0 with cohomology equal to $  C^\infty_{\gamma, a}(M^{\times (q+1)})^{Z_{\gamma}}$. A similar spectral sequence argument as above shows that the cohomology of double complex $D^{\bullet, \bullet}$ is computed by $ C^\infty_{\gamma, a}(M^{\times (\bullet+1)})^{Z_\gamma}$, which is exactly the invariant Alexander-Spanier complex $C^\bullet_{\rm AS,inv}(M, \gamma)$. 

Under this bicomplex, the map $\Psi_{\rm inv}$ can be identified with the map $\Phi^\chi_M$ in \cite[Proposition 2.5, and equation (2.4)]{PP1} constructed using the ``zig-zag" splitting trick. So we can conclude that the map $\Psi_{\rm inv}$ is a quasi-isomorphism. 
\end{proof}

We now define the cochain map $\Lambda^\gamma: C^\infty_{\gamma,a}(M^{\times (q+1)})\to \Omega^q(M^\gamma)$ by
\[
\Lambda^\gamma(f)(x)(v_1, \cdots, v_q):=\frac{1}{k!}\sum_{\tau\in S_k}\sgn(\tau) \frac{\partial}{\partial \epsilon_1}\cdots \frac{\partial}{\partial \epsilon_k}f(x, \exp_x(\epsilon_1 v_{\tau(1)}), \cdots, \exp_x(\epsilon_q v_{\tau(q)})|_{\epsilon_1=\cdots =\epsilon_k=0}
\]
for $x\in M^\gamma, v_1, \cdots, v_q\in T_x M^\gamma$ and where $\exp_x:T_xM\to M$ denotes the exponential map defined by a $\Gamma$-invariant riemannian metric. 

It is straightforward to check that the map $\Lambda^\gamma$ is a cochain map and is $Z_\gamma$-equivariant. Therefore $\Lambda^\gamma$ induces a cochain map $\Lambda^\gamma:  C^\infty_{\gamma, a}(M^{\times (q+1)})^{Z_\gamma}\to \Omega(M^\gamma)^{Z_\gamma}$ and therefore a map $\Lambda^\gamma: C^\bullet_{\rm AS, inv}(M, \gamma)\to \Omega(M^\gamma)^{Z_\gamma}$.

Composing $\Lambda^\gamma$ with $\Psi_{\rm inv}$, we obtain the following chain map
\[
\Psi^\gamma:=\Lambda^\gamma\circ \Psi_{\rm inv}: \mathscr{C}^\bullet_\gamma(\Gamma)\to \Omega^\bullet(M^\gamma)^{Z_\gamma}
\]
by
\begin{equation}\label{eq:lambdagammapsi}
\Lambda^\gamma (\Psi_{\rm inv})(c)=\sum_{\gamma_0, \cdots,\gamma_k}c(\gamma_0, \cdots, \gamma_k)\chi(\gamma_0^{-1}x)d\chi(\gamma_1^{-1} x)\cdots d\chi(\gamma_k^{-1}x)|_{M^\gamma}.
\end{equation}

\begin{remark}
\label{rk-te}
Because the action of $\Gamma$ on $M$ is proper, $M^\gamma$ can only be non-empty when $\gamma$ is a torsion element. This is the reason why we only consider
the first component in the decomposition \eqref{eq-dec-b} of the cyclic cohomology of $\CC\Gamma$.
\end{remark}

\begin{remark}
The above argument recovers, in particular, the result explained in detail in [Shaegan, Lemma 3.4], namely that Lott's
complex has the same cohomology of $Z_\gamma$. It is explained in \cite[Cor. 7.9]{PSZ} that since $\gamma$ has finite order, the Leray-Serre spectral sequence for group cohomology associated with the short exact sequence
\[
1\to \gamma^\mathbb{Z} \to Z_\gamma\to N_\gamma\to 1
\]
degenerates at the $E_2$ page and we have the following isomorphism of group cohomology
\[
H^\bullet(N_
\gamma)\cong H^\bullet(Z_\gamma). 
\]
\end{remark}

\section{Finite propagation and monomial  approximation}\label{section:finite propagation}
We now consider, in the setting of the previous section, --$M$ being a manifold equipped with a proper, cocompact action of a discrete group $\Gamma$--, a Dirac operator $D$ on $M$ that is invariant under $\Gamma$. We furthermore assume the dimension of $M$ is even so that the spinor bundle splits into even and odd degree, $S=S^+\oplus S^-$  and denote by $D^\pm$ the restriction of $D$ to $S^\pm$. In this setting, we can construct the index class $\operatorname{Ind_c} (D)\in K_0 (\mathcal{A}_\Gamma^c (M))$ as in \cite{CM}:  Let $Q$ be a $\Gamma$-equivariant parametrix
of $\Gamma$-compact support with remainders $S_\pm$. Consider the $2\times 2$ matrix
\begin{equation}\label{CS-projector}
P:= \left(\begin{array}{cc} S_{+}^2 & S_{+}  (I+S_{+}) Q\\ S_{-} D^+ &
I-S_{-}^2 \end{array} \right).
\end{equation}
This produces a class 
\begin{equation}\label{CS-class}
\operatorname{Ind_c} (D):= [P] - [e_1]\in K_0 (\mathcal{A}_\Gamma^c (M))\;\;\text{with}\;\;e_1:=\left( \begin{array}{cc} 0 & 0 \\ 0&1
\end{array} \right).
\end{equation}
Using the Chern character to cyclic homology, we can now consider the pairing of this index class with a cocycle $c\in\mathscr{C}^{2q}_\gamma(\Gamma)$ defined as 
\[
\langle{\rm Ind}_c(D),c\rangle:=\langle{\rm Ch}_{2q}({\rm Ind}_c(D)),(\rho\circ\Psi)(c)\rangle. 
\]
By the discussion in the previous section, we obtain the following expression for the cyclic 
cocycle $\rho \circ \Psi (c) \in Z_\lambda (\mathcal{A}^c_\Gamma (M))$:
\begin{multline}
\rho \circ \Psi (c)(A_0,\ldots,A_k)\\
= \sum_{\gamma_0\in Z_\gamma\backslash \Gamma} \int_{M^{k+1}} \Psi(c)(\gamma_0, y_0, \dots, y_k) A_0 (y_0,y_1)
\cdots A_{k-1}(y_{k-1},y_k) A_k (y_k, \gamma y_0)  \\
\end{multline}
Using the cutoff function $ \chi^\Gamma_\gamma$ for the action by $Z_\gamma$ on $\Gamma$ by left multiplication and (\ref{eq:fc}),  we rewrite 
\begin{multline}\label{eq:rhopsi}
\rho \circ \Psi(c)(A_0,\ldots,A_k)=\sum_{\gamma_0\in  \Gamma} \int_{M^{k+1}}  \chi^\Gamma_\gamma(\gamma_0) \cdot \Psi(c)(\gamma_0, y_0, \dots, y_k) A_0 (y_0,y_1)
\cdots A_{k-1}(y_{k-1},y_k) A_k (y_k, \gamma y_0)\\
=\int_{M^{k+1}} \left(\sum_{\gamma_0, \cdots, \gamma_{k} \in \Gamma } c(\gamma_0, \cdots, \gamma_{k})\cdot \chi^\Gamma_\gamma(\gamma_0) \cdot \chi(\gamma_0^{-1}y_0)\cdots \chi\left( \gamma_{k}^{-1}y_k\right)\right) \cdot  A_0(  y_0,  y_1)
\ldots A_k( y_k,  \gamma  y_0).
\end{multline}
In this section we discuss in detail this expression when the $A_j$ are the entries of the index class
$\Ind_c (D)$.

\bigskip
We briefly explain the strategy of our study in four steps. \\

\noindent{\bf Step I}: By a suitable choice of a rapid decay function $\psi_t$ whose Fourier transform has specific support properties, we express the index element $\operatorname{Ind}_c(D)$ through a finite-propagation idempotent $P_{\operatorname{fp}}(t)$. 
Recall that
$\operatorname{Ind}_c(D)$  is an element  
in $K_0(\mathcal{A}_\Gamma^c (M))$ and that it therefore pairs naturally with $\rho\circ \Psi(c)$ for $c\in \mathscr{C}^\bullet _\gamma(\Gamma)$ through \eqref{eq:rhopsi}.

\noindent{\bf Step II}: The pairing $\rho\circ \Psi(c) \big(\operatorname{Ind}_c(D)\big)$ is independent of the time variable $t$. We can compute the pairing by 
\[
\lim_{t\to 0} \rho\circ \Psi(R_{\operatorname{fp}}(t), \cdots, R_{\operatorname{fp}}(t)),\qquad R_{\operatorname{fp}}(t)=P_{\operatorname{fp}}(t)-\begin{pmatrix}0&0\\0&1\end{pmatrix}.
\]
By the support property of the Fourier transform of $\psi_t$, the propagation of $R_{\operatorname{fp}}(t)$ is uniformly finite, independent of $t$. This property allows to truncate $\rho\circ \Psi(c)$ in to a finite sum of terms like
\[
c(\gamma_0, \cdots, \gamma_k)\chi_\gamma^\Gamma(\gamma_0)\chi(\gamma_0^{-1}y_0)\cdots \chi(\gamma_k^{-1}y_k)\,;
\] 
the pairing $\rho\circ \Psi (c) (R_{\operatorname{fp}}(t), \cdots, R_{\operatorname{fp}}(t))$  is therefore computed by a finite sum of integrals of the form
\[
\int_{M^{k+1}}  c(\gamma_0, \cdots, \gamma_{k})\cdot \chi^\Gamma_\gamma(\gamma_0) \cdot \chi(\gamma_0^{-1}y_0)\cdots \chi\left( \gamma_{k}^{-1}y_k\right)\cdot  R_{\operatorname{fp}}(t)(  y_0,  y_1)
\ldots R_{\operatorname{fp}}(t)( y_k,  \gamma  y_0).
\]

\noindent{\bf Step III}: Let $R(\sqrt{t}D)$ be the Connes-Moscovici idempotent defined by the heat kernel of $D$, $e^{-tD^2}$. The entries of $R(\sqrt{t}D)$ do not have finite propagation. However, the integral   
\[
\int_{M^{k+1}}  c(\gamma_0, \cdots, \gamma_{k})\cdot \chi^\Gamma_\gamma(\gamma_0) \cdot \chi(\gamma_0^{-1}y_0)\cdots \chi\left( \gamma_{k}^{-1}y_k\right)\cdot  R(\sqrt{t}D)(  y_0,  y_1)
\ldots R(\sqrt{t}D)( y_k,  \gamma  y_0)
\]
is well defined. We prove that the difference of the smoothing kernels $R_{\operatorname{fp}}(t))-R(\sqrt{t}D)$ decays faster than any $t^k$, $\forall k\in \mathbb{N}$ as $t\to 0$. So the integral 
\[
\lim_{t\to 0} \int_{M^{k+1}}  c(\gamma_0, \cdots, \gamma_{k})\cdot \chi^\Gamma_\gamma(\gamma_0) \cdot \chi(\gamma_0^{-1}y_0)\cdots \chi \left( \gamma_{k}^{-1}y_k\right) \cdot  R_{\operatorname{fp}}(t)(  y_0,  y_1)
\ldots R_{\operatorname{fp}} (t)( y_k,  \gamma  y_0).
\]
can computed by
\[
\lim_{t\to 0} \int_{M^{k+1}}  c(\gamma_0, \cdots, \gamma_{k})\cdot \chi^\Gamma_\gamma(\gamma_0) \cdot \chi(\gamma_0^{-1}y_0)\cdots \chi\left( \gamma_{k}^{-1}y_k\right) \cdot  R(\sqrt{t}D)(  y_0,  y_1)
\ldots R(\sqrt{t}D)( y_k,  \gamma  y_0)
\]

\noindent{\bf Step IV}: We use the Volterra pseudodifferential calculus and the Volterra Getzler rescaling mechanism, see \cite{Ponge-Wang-2}, \cite{CM}, in order to compute the above short time limit
\[
\lim_{t\to 0} \int_{M^{k+1}}  c(\gamma_0, \cdots, \gamma_{k})\cdot \chi^\Gamma_\gamma(\gamma_0) \cdot \chi(\gamma_0^{-1}y_0)\cdots \chi \left( \gamma_{k}^{-1}y_k\right) \cdot  R(\sqrt{t}D)(  y_0,  y_1)
\ldots R(\sqrt{t}D)( y_k,  \gamma  y_0).
\]

\medskip
\noindent
We now pass to a detailed explanation of these 4 steps.

\subsection{A finite propagation index class}
Fix $0<\epsilon <\frac{1}{2}$. Choose a compactly supported even continuous function  $\theta$ on $\mathbb{R}$, which is equal to 1 in a neighborhood of $0$.  Set 
\[
\psi(u):=e^{-\frac{u^2}{2}}.
\]
Let $\hat{\psi}$ be the Fourier transform of $\psi$. For $t>0$, define 
\[
a_t:=\frac{1}{2\pi}\int_\mathbb{R} \theta(t^\epsilon s) \hat{\psi}(s){\rm d}s.
\]
We observe that $\theta$ is compactly supported and $\psi(u)$ (and therefore $\hat{\psi}$) is a Schwartz function. Furthermore, 
\[
\theta(t^\epsilon s)\to \theta(0)=1,\ t\to 0. 
\]
Hence, it follows from the dominated convergence theorem that 
\[
a_t\to \frac{1}{2\pi}\int_\mathbb{R}  \hat{\psi}(s){\rm d}s=\psi(0)=1,\ t\to 0. 
\]
So $a_t>0$ when $t$ is sufficiently small. Now define
\[
\psi_t(u):=\frac{a_t^{-1}}{2\pi}\int_\mathbb{R}\theta(t^\epsilon s)\hat{\psi}(s)e^{isu}{\rm d}s.
\]

Throughout the paper, by $O(t^\infty)$, we mean a function $a(t)$ that decays to 0 faster than any  $t^n$ for all $n\in \mathbb{N}$. 
\begin{lemma}\label{lem:psi-t-estimate} In every Schwartz norm, 
  \[
  \psi_t-\psi=O(t^\infty).
  \]
\end{lemma}

\begin{proof}

Put
\[
        \widetilde\psi_t(u)
        :=
        \frac1{2\pi}
        \int_{\mathbb R}\theta(t^\epsilon s)\widehat\psi(s)e^{isu}\,ds .
\]
We shall first prove that
\begin{equation}\label{eq:psi-app}
        \widetilde\psi_t-\psi=O(t^\infty)
\end{equation}
in every Schwartz seminorm.

Since $\theta=1$ in a neighbourhood of $0$, there exists $c>0$ such that
$\theta(s)=1$ for $|s|\leq c$. Hence
\[
        1-\theta(t^\epsilon s)=0
        \qquad\text{if } |s|\leq c t^{-\epsilon}.
\]
Therefore
\[
\begin{aligned}
        \widetilde\psi_t(u)-\psi(u)
        &=
        -\frac1{2\pi}
        \int_{\mathbb R}
        \bigl(1-\theta(t^\epsilon s)\bigr)
        \widehat\psi(s)e^{isu}\,ds .
\end{aligned}
\]
Let
\[
        r_t(u):=\widetilde\psi_t(u)-\psi(u).
\]
We show that for every pair of nonnegative integers \(m,\ell\),
\[
        \sup_{u\in\mathbb R}
        |u^m\partial_u^\ell r_t(u)|
        =
        O(t^\infty).
\]
Differentiating under the integral gives
\[
        \partial_u^\ell r_t(u)
        =
        -\frac1{2\pi}
        \int_{\mathbb R}
        (is)^\ell
        \bigl(1-\theta(t^\epsilon s)\bigr)
        \widehat\psi(s)e^{isu}\,ds .
\]
Multiplying by \(u^m\) and integrating by parts \(m\) times in the
variable \(s\), using \(u^m e^{isu}=i^{-m}\partial_s^m e^{isu}\), yields
\[
\begin{aligned}
        u^m\partial_u^\ell r_t(u)
        =
        C_{m,\ell}
        \int_{\mathbb R}
        \partial_s^m
        \left[
            s^\ell
            \bigl(1-\theta(t^\epsilon s)\bigr)
            \widehat\psi(s)
        \right]
        e^{isu}\,ds ,
\end{aligned}
\]
for a constant \(C_{m,\ell}\). Hence
\[
        \sup_{u\in\mathbb R}
        |u^m\partial_u^\ell r_t(u)|
        \leq
        C_{m,\ell}
        \left\|
        \partial_s^m
        \left[
            s^\ell
            \bigl(1-\theta(t^\epsilon s)\bigr)
            \widehat\psi(s)
        \right]
        \right\|_{L^1_s}.
\]
We claim that the right-hand side is \(O(t^\infty)\). Indeed, after
expanding the derivative on the product $ s^\ell
            \bigl(1-\theta(t^\epsilon s)\bigr)
            \widehat\psi(s)
$, we get a finite linear combination of
expressions of the form
\[
        t^{j\epsilon}
        \theta^{(j)}(t^\epsilon s)
        \partial_s^{m-j}\bigl(s^\ell\widehat\psi(s)\bigr),
        \qquad j\geq 1,
\]
or of the form
\[
        \bigl(1-\theta(t^\epsilon s)\bigr)
        \partial_s^m\bigl(s^\ell\widehat\psi(s)\bigr),
\]
where $\theta^{(j)}$ is the $j$th derivative of $\theta$.

We observe that both the function
\(\theta^{(j)}(t^\epsilon s)\) (\(j\geq 1\)) and $1-\theta(t^\epsilon s)$ are supported where
\[
        |s|\geq c t^{-\epsilon}.
\]
Since \(\widehat\psi\) is a Schwartz function, every derivative
\(\partial_s^a(s^\ell\widehat\psi(s))\) decays faster than any power of
\(|s|\). Therefore, for every \(N\),
\[
        \left\|
        \partial_s^m
        \left[
            s^\ell
            \bigl(1-\theta(t^\epsilon s)\bigr)
            \widehat\psi(s)
        \right]
        \right\|_{L^1_s}
        =
        O(t^N).
\]
As \(N\) is arbitrary, this proves
\[
        \widetilde\psi_t-\psi=O(t^\infty)
\]
in every Schwartz seminorm.

We look at the normalizing factor \(a_t^{-1}\). Since
\[
        a_t-\psi(0)
        =
        \frac1{2\pi}
        \int_{\mathbb R}
        \bigl(\theta(t^\epsilon s)-1\bigr)
        \widehat\psi(s)\,ds ,
\]
the above tail estimate (\ref{eq:psi-app}) with \(m=\ell=0\) gives
\[
        a_t-1=O(t^\infty).
\]
In particular, for \(t>0\) sufficiently small, \(a_t\) is bounded away
from zero, and hence
\[
        a_t^{-1}-1=O(t^\infty).
\]
Finally,
\[
        \psi_t-\psi
        =
        a_t^{-1}\widetilde\psi_t-\psi
        =
        a_t^{-1}(\widetilde\psi_t-\psi)
        +
        (a_t^{-1}-1)\psi .
\]
The first term is \(O(t^\infty)\) in every Schwartz seminorm by the
estimate already proved, while the second term is \(O(t^\infty)\) because
\(\psi\in\mathcal S(\mathbb R)\) is fixed and
\(a_t^{-1}-1=O(t^\infty)\). We conclude that 
\[
        \psi_t-\psi=O(t^\infty)
\]
in every Schwartz seminorm.
\end{proof}

Notice that by the definition $\psi_t$, for sufficiently small $t$, $\psi_t(0)=1$. 
\begin{definition}\label{defn:st_bt_qt}Define
\[
  s_t(u):=\psi_t(u)^2,
  \qquad
  b_t(u):=\frac{1-\psi_t(u)}{u^2},
  \qquad
  q_t(u):=\frac{1-s_t(u)}{u^2}=\bigl(1+\psi_t(u)\bigr)b_t(u).
\]
\end{definition}

For any tempered distribution $f$, we denote by 
\[
\widehat   f(u)=\frac1{2\pi}\langle f(s),e^{isu}\rangle .
\]
its Fourier transform. 
\begin{lemma}\label{lem:Paley-Wiener}
A tempered distribution
\(\widehat f\) is supported in \([-R,R]\) if and only if its inverse Fourier
transform is the restriction to the real line of an entire function satisfying an
estimate of the form
\[
  |f(z)|\le C(1+|z|)^N e^{R|\operatorname{Im}z|},
  \qquad z\in\mathbb C.
\]
A function $f$ is said to be of exponential type at most $R$ if $f$ satisfies the above inequality. 
\begin{proof}
This is Paley--Wiener--Schwartz theorem. See \cite{Rudin}
\end{proof}
\end{lemma}

\begin{lemma}\label{lem:fp-D-calculus}
If
$\operatorname{supp}\theta\subset[-c,c]$,  then
\[
        \operatorname{supp}\widehat{\psi_t}\subset[-R_t,R_t],
        \qquad
        \operatorname{supp}\widehat{s_t}\subset[-2R_t,2R_t],
\]
and, as tempered distributions,
\[
        \operatorname{supp}\widehat{b_t}\subset[-R_t,R_t],
        \qquad
        \operatorname{supp}\widehat{q_t}\subset[-2R_t,2R_t],
\]
where $R_t:=ct^{-\epsilon}$. 
\end{lemma}
\begin{proof}
Because $\theta$ has compact support, we can find constant $c$ so that 
\[
\operatorname{supp}\theta\subset[-c,c].
\] 
By the definition of $\psi_t$,
\[
        \widehat{\psi_t}(s)=a_t^{-1}\theta(t^\epsilon s)\widehat\psi(s).
\]
Accordingly, we get
\[
        \operatorname{supp}\widehat{\psi_t}
        \subset [-R_t,R_t].
\]
with $R_t = ct^{-\epsilon}$. Hence, by Lemma \ref{lem:Paley-Wiener}, $\psi_t$ extends to an
entire function of exponential type at most $R_t$; equivalently, for some
constants $C_t,N_t$,
\[
        |\psi_t(z)|
        \le C_t(1+|z|)^{N_t}e^{R_t|\operatorname{Im}z|},
        \qquad z\in\mathbb C.
\]
Also, since $\widehat{\psi_t}\in C_c^\infty(\mathbb R)$, the restriction of
$\psi_t$ to the real line belongs to $\mathcal S(\mathbb R)$.

The function
\[
        s_t=\psi_t^2
\]
is entire of exponential type at most $2R_t$.  On the real line it is a
Schwartz function.  Therefore Lemma \ref{lem:Paley-Wiener} gives
\[
        \operatorname{supp}\widehat{s_t}\subset[-2R_t,2R_t].
\]

We next consider
\[
        b_t(u)=\frac{1-\psi_t(u)}{u^2}.
\]
Set
\[
        F_t(z):=1-\psi_t(z).
\]
Then $F_t$ is entire of exponential type at most $R_t$.  Moreover,
$F_t(0)=0$, 
because $\psi_t(0)=1$, and $F_t'(0)=0$ 
because $\psi_t$ is even.  Hence $F_t$ vanishes to order at least $2$ at
$0$, and $F_t(z)/z^2$ has a removable singularity there.  Thus $b_t$ is
entire.

We now check that division by $z^2$ does not increase the exponential
type.  For $|z|\ge 1$ this follows directly from the estimate for $F_t$:
\[
        |b_t(z)|
        =
        \frac{|F_t(z)|}{|z|^2}
        \le
        C_t(1+|z|)^{N_t}e^{R_t|\operatorname{Im}z|}.
\]
For $|z|\le 1$, Taylor's formula with integral remainder gives
\[
        F_t(z)
        =
        z^2\int_0^1(1-r)F_t''(rz)\,{\rm d}r,
\]
and therefore
\[
        b_t(z)
        =
        \int_0^1(1-r)F_t''(rz)\,{\rm d}r.
\]
By Cauchy's estimate, applied on the circle of radius $1$ centered at $rz$,
\[
        |F_t''(rz)|
        \le
        2!\sup_{|\zeta-rz|=1}|F_t(\zeta)|
        \le
        C_t'e^{R_t|\operatorname{Im}z|}.
\]
We obtain
\[
 |b_t(z)|\leq \int_0^1(1-r)|F_t''(rz)|\,{\rm d}r
 \leq \int_0^1(1-r) C_t'e^{R_t|\operatorname{Im}z|} \; dr
     =
        \frac{C_t'}{2}\; e^{R_t|\operatorname{Im}z|},
        \qquad z\in\mathbb C.	
\]
Thus $b_t$ is entire of exponential type at most $R_t$.  Applying
Lemma \ref{lem:Paley-Wiener} again gives
\[
        \operatorname{supp}\widehat{b_t}\subset[-R_t,R_t].
\]

Finally,
\[
        q_t(u)
        =
        \frac{1-s_t(u)}{u^2}
        =
        (1+\psi_t(u))b_t(u).
\]
The factor $1+\psi_t$ is entire of exponential type at most $R_t$, and
$b_t$ is entire of exponential type at most $R_t$.  Hence $q_t$ is entire of
exponential type at most $2R_t$. Therefore
\[
        \operatorname{supp}\widehat{q_t}\subset[-2R_t,2R_t].
\]
This proves the claim.
\end{proof}

Building on \cite{CS, CM, moscovici-wu}, we are now ready to introduce a specific time-dependent projection in $M_2(\mathcal{A}^c_\Gamma(M))$ that 
gives a very useful representative of  the index element $\operatorname{Ind}_c(D)$ in $K_0(\mathcal{A}^c_\Gamma(M))$.  

For a sufficiently small $t>0$, define 
\[
Q(t):=\sqrt{t}D^- b_t(\sqrt{t\Delta_-})=b_t(\sqrt{t\Delta_+})\sqrt{t}D^+
\] 
with 
$$\Delta_+:= D^+ D^-, \quad\quad \Delta_-:= D^+ D^-\,.$$
Using the general formula 
\[
f(D)=\frac{1}{2\pi}\int \hat{f}(\lambda)e^{i\lambda D} d\lambda,
\]
for any Schwarz function $f\in S(\mathbb{R})$, the propagation speed of $f(D)$ is controlled by the support of its Fourier transform because the propagation speed of $e^{i\lambda D}$ is $|\lambda |$. With this, and  Lemma \ref{lem:fp-D-calculus}, we see that $Q(t)$ has finite propagation $O(t^{\frac{1}{2}-\epsilon})$.
Define
\[
S_+(t):=\psi_t(\sqrt{t\Delta_+}),\qquad S_-(t):=\psi_t(\sqrt{t\Delta_-}).
\]

We consider the idempotent
\begin{equation}\label{eq:Pfp}
  P_{\operatorname{fp}}(t):=
  \begin{pmatrix}
    S_+(t)^2 & S_+(t)(1+S_+(t))Q(t) \\
    S_-(t)\sqrt{t}D_+ & 1-S_-(t)^2
  \end{pmatrix}.
\end{equation}
Let
\[
  e_1:=\begin{pmatrix}0&0\\0&1\end{pmatrix}.
\]
\begin{definition}\label{defn:fp-index}
The finite propagation index element $\operatorname{Ind}_c(D)\in K_0(\mathcal{A}^c_\Gamma(M))$ is defined to be 
\[
\operatorname{Ind}_c(D):=[P_{\operatorname{fp}}(t)]-[e_1]. 
\] 
\end{definition}

\begin{lemma}\label{lem:chern-reduction}
Let $\varphi$ be a cyclic cocycle on $\mathcal{A}^c_\Gamma(M)$. Extend \(\varphi\) to matrix algebras by the ordinary matrix trace.  We use the
reduced normalization of the degree \(2q\) component of the Chern character in which the relative
chain of an idempotent difference \([P]-[e]\) is represented by
\((P-e)^{\otimes(2q+1)}\).  The Chern character pairing between $\operatorname{Ind}_c(D)$ and $\varphi$ can be computed by the following formula,
\begin{equation}\label{eq:exact-reduction}
  \big\langle \operatorname{Ch}_{2q}(\operatorname{Ind}_c(D)),\varphi\big\rangle
  =
  \varphi\bigl(R_{\operatorname{fp}}(t),\ldots,R_{\operatorname{fp}}(t)\bigr),
\end{equation}
where $R_{\operatorname{fp}}(t)$ is defined to be 
\begin{equation}\label{eq:Rfp}
  R_{\operatorname{fp}}(t):=
  \begin{pmatrix}
    S_+(t)^2 & S_+(t)(1+S_+(t))Q(t) \\
    S_-(t)\sqrt{t}D_+ & -S_-(t)^2
  \end{pmatrix}.
\end{equation}
\end{lemma}
\begin{proof}This is a straightforward algebraic computation, which is left to the reader. 
\end{proof}

For $c\in \mathscr{C}^k_\gamma(\Gamma)$,  in the Chern character pairing $\langle \operatorname{Ind}_c(D), \rho\circ \Psi(c)\rangle$ has the following expression,
\[
\int_{M^{k+1}} \left(\sum_{\gamma_0, \cdots, \gamma_{k} \in \Gamma } c(\gamma_0, \cdots, \gamma_{k})\cdot \chi^\Gamma_\gamma(\gamma_0) \cdot \chi(\gamma_0^{-1}y_0)\cdots \chi\left( \gamma_{k}^{-1}y_k\right)\right) \cdot  A_0(  y_0,  y_1)
\ldots A_k( y_k,  \gamma  y_0),
\]
where $A_i(-,-)$ ($i=0,...,k$) are entries of the matrix $  R_{\operatorname{fp}}(t)$. 

By Lemma \ref{lem:fp-D-calculus}, each entry of the idempotent $P_{\operatorname{fp}}(t)$ has propagation $Ct^{\frac{1}{2}-\epsilon}$. Therefore,  in the above integral expression each $A_i(-,-)$ has propagation $Ct^{\frac{1}{2}-\epsilon}$. 
Let $S$ be the support of the function $\chi$. Then the integrand $\left(\sum_{\gamma_0, \cdots, \gamma_{k} \in \Gamma } c(\gamma_0, \cdots, \gamma_{k})\cdot \chi^\Gamma_\gamma(\gamma_0) \cdot \chi(\gamma_0^{-1}y_0)\cdots \chi\left( \gamma_{k}^{-1}y_k\right)\right) \cdot  A_0(  y_0,  y_1)
\ldots A_k( y_k,  \gamma  y_0),$ is supported on  
\[
\{(y_0, \cdots, y_k)|\ \gamma_0^{-1}y_0 \in S, ...,  \gamma_k^{-1}y_k\in S, d(y_0, y_1)\leq Ct^{\frac{1}{2}-\epsilon},..., d(y_{k-1}, y_k)\leq Ct^{\frac{1}{2}-\epsilon}, d(y_k, \gamma y_0)\leq Ct^{\frac{1}{2}-\epsilon}\}.
\]

\begin{proposition}\label{prop:finitesum}There exists a finite subset $\mathcal{I} _k\subset \Gamma^{\times (k+1)}$ 
	\begin{equation}\label{eqn:finitesum}
	\begin{split}
         &\rho \circ \Psi(c)(A_0,\ldots,A_k)\\
         = &\int_{M^{k+1}} \left(\sum_{(\gamma_0, \cdots, \gamma_{k}) \in \mathcal{I}_k } c(\gamma_0, \cdots, \gamma_{k})\cdot \chi^\Gamma_\gamma(\gamma_0) \cdot \chi(\gamma_0^{-1}y_0)\cdots \chi\left( \gamma_{k}^{-1}y_k\right)\right) \cdot  A_0(  y_0,  y_1)
\ldots A_k( y_k,  \gamma  y_0).	
\end{split}
\end{equation}
\end{proposition}
\begin{proof}Since $\gamma_i^{-1}y_i\in S$ and $d(y_i, y_{i+1})\leq Ct^{\frac{1}{2}-\epsilon}$, for $i=1, ..., k-1$ and $d(y_k, \gamma y_0)\leq Ct^{\frac{1}{2}-\epsilon}$, we have
\[
d(\gamma_i S, \gamma_{i+1}S)\leq Ct^{\frac{1}{2}-\epsilon},\ i=0, ..., k-1,\ d(\gamma_k S, \gamma \gamma_0 S)\leq Ct^{\frac{1}{2}-\epsilon}. 
\]
Define $F\subset \Gamma$ by
\[
F:=\{\alpha\in \Gamma |\ d(S, \alpha S )\leq Ct^{\frac{1}{2}-\epsilon}\}. 
\]
By the assumption that the $\Gamma$ action on $M$ is proper and the set $S$ is compact, $F$ is finite. Accordingly, 
\[
\gamma_i^{-1}\gamma_{i+1}\in F,\ i=0, ..., k-1,\ \gamma_k^{-1}\gamma \gamma_0\in F. 
\]
Multiplying the above relations, we conclude 
\[
\gamma_0^{-1}\gamma \gamma_0=(\gamma_0^{-1} \gamma_1) (\gamma_1^{-1}\gamma_2)...(\gamma_{k-1}^{-1}\gamma_k)\gamma_k^{-1}\gamma \gamma_0 \in F^{\times (k+1)}. 
\]
So $[\gamma_0^{-1}\gamma \gamma_0]$ belongs to a finite subset in $\Gamma/Z_\gamma$. By the construction, $\chi^\Gamma_\gamma(\gamma_0) $ has finite supports along each coset of $\Gamma/Z_\gamma$, $\gamma_0$ must belong to a finite set. 
Accordingly, $\gamma_1, ..., \gamma_k$ all must belong to a finite set, which proves the existence of the set $\mathcal{I}_k$. 
\end{proof}

As a function on $M^{\times (k+1)}$, the following
\begin{align}\label{anti-eq}
\sum_{I \in \Gamma^{\times (k+1)}} \left(c(\gamma_0, \cdots, \gamma_{k})\cdot\chi^\Gamma_\gamma(\gamma_0) \cdot  \chi(\gamma_0^{-1}y_0)\cdots \chi\left( \gamma_{k}^{-1}y_k\right)\right)	= \sum_{\gamma_0\in Z_\gamma \backslash \Gamma } \Psi(c)(\gamma_0, y_0, \ldots y_k)
\end{align}
is anti-symmetric  by Lemma \ref{lem:fc}.  For subset $\mathcal{I}_k \in \Gamma^{\times (k+1)}$ in Proposition \ref{prop:finitesum}, put
\begin{align}\label{def:CI}
\mathcal{C}_{\mathcal{I}_k}(y_0, \ldots, y_k) := \sum_{(\gamma_0, ..., \gamma_k)\in \mathcal{I}_k }c(\gamma_0, \cdots, \gamma_{k})\cdot\chi^\Gamma_\gamma(\gamma_0) \cdot  \chi(\gamma_0^{-1}y_0)\cdots \chi\left( \gamma_{k}^{-1}y_k\right)	
\end{align}
and its anti-symmetrization by 
\begin{align}\label{def:CIa}
\mathcal{C}_{\mathcal{I}_k, a}(y_0, \ldots, y_k) := \sum_{\sigma \in S_{k+1}} \text{sign} (\sigma) \cdot \mathcal{C}_{\mathcal{I}_k}(y_{\sigma(0)}, \ldots, y_{\sigma(k)}).
\end{align}
Because the right side of (\ref{anti-eq}) is anti-symmetric, we have that
\begin{align}\label{equ:antisym}
\rho\big(\mathcal{C}_{\mathcal{I}_k, a}\big)( \operatorname{Ind}_c(D), ..., \operatorname{Ind}_c(D))= 
(k+1)! \rho\big(\mathcal{C}_{\mathcal{I}_k}\big)( \operatorname{Ind}_c(D), ..., \operatorname{Ind}_c(D))=(k+1)!\langle \operatorname{Ind}_c(D), \rho\circ \Psi(c)\rangle.
\end{align}
We can assume without loss of generality that $\mathcal{C}_{\mathcal{I}_k, a}(y_0, \ldots, y_k) $ is anti-symmetric.

\subsection{Passing to the standard heat kernel}
In the following we explain the strategy that we will take to compute the short time limit of $\langle \operatorname{Ind}_c(D), \rho\circ \Psi(c)\rangle$ via the heat kernel method. By Proposition \ref{prop:finitesum}, as the pairing $\rho \circ \Psi(c)(A_0,\ldots,A_k)$ 
is a finite sum, it is sufficient to consider  $f_0, ..., f_k\in C^\infty_c(M)$ and the pairing 
\[
\lim_{t\to 0}\rho(f_0\otimes \cdots \otimes f_k)(R_{\operatorname{fp}}(t), \cdots, R_{\operatorname{fp}}(t)).
\]

We consider
\begin{equation}\label{eq:R-heat}
  R(\sqrt tD)=
  \begin{pmatrix}
    e^{-t\Delta_+}
    & \left(\frac{1-e^{-t\Delta_+}}{t\Delta_+}\right)e^{-t\Delta_+/2}\sqrt tD^- \\
    e^{-t\Delta_-/2}\sqrt tD^+
    & -e^{-t\Delta_-}
  \end{pmatrix}.
\end{equation}

We observe that those $R(\sqrt{t}D)$ does not have finite propagation, its pairing with $\rho(f_0\otimes \cdots \otimes f_k)$ is well defined as $f_0, ..., f_k$ all have compact supports. More explicitly,  $\rho(f_0\otimes \cdots \otimes f_k)(R(\sqrt tD), \cdots, R(\sqrt tD))$ is a finite sum of integrals of the form, 
\[
\rho(f_0\otimes \cdots \otimes f_k)(A_0, ..., A_k)=
\int_{M^{k+1}} f_0(y_0)A_0(  y_0,  y_1)f_1(y_1)
\ldots f_k(y_k)A_k( y_k,  \gamma  y_0)
\]
which is finite.

The following is the main technical property we will need to compute the short time limit of $R_{\operatorname{fp}}(t)$ via $R(\sqrt{t}D)$. 

\begin{definition}
Let $A$ be a smoothing operator on $M$, with Schwartz kernel
$K_A(x,y)$.  For compact subsets $K,L\subset M$ and differential
operators $P_x,Q_y$ acting respectively in the left and right variables,
define
\[
        p_{K,L,P,Q}(A)
        :=
        \sup_{x\in K,\ y\in L}
        \left\|P_xQ_yK_A(x,y)\right\|.
\]
The family of seminorms $p_{K,L,P,Q}$ is called the family of local
smoothing seminorms.

For a family of smoothing operators $A_t$, we write
$A_t=O(t^N)$ in every local smoothing seminorm if, for all
$K,L,P,Q$, there exists $C_{K,L,P,Q,N}>0$ such that
\[
        p_{K,L,P,Q}(A_t)
        \le C_{K,L,P,Q,N}t^N
\]
for $t$ sufficiently small.  We write $A_t=O(t^\infty)$ in every local
smoothing seminorm if the same estimate holds for every $N\ge 0$.
\end{definition}

\begin{lemma}\label{lem:entry}
Let $r_t\in \mathcal S(\mathbb R)$ satisfy
\[
        r_t=O(t^\infty)
\]
in every Schwartz seminorm. Then 
\[
r_t(\sqrt{t}D) = O(t^\infty)
\]
in all local smoothing seminorms. 
\end{lemma}
\begin{proof}

We first prove the rapid decay property in Sobolev operator norms.  Let
\[
        \Lambda=(1+D^2)^{1/2}.
\]
Since $r_t(\sqrt tD)$ is a function of $D$, it commutes with $\Lambda$.
Let $a,b\geq 0$.  Then
\[
\begin{aligned}
 \|r_t(\sqrt tD)\|_{H^{-a}\to H^b}
 &=
 \|\Lambda^b r_t(\sqrt tD)\Lambda^a\|_{L^2\to L^2}  \\
 &=
 \|\Lambda^{a+b}r_t(\sqrt tD)\|_{L^2\to L^2}.
\end{aligned}
\]
By the spectral theorem, this last norm is bounded by
\[
        \sup_{\lambda\in\mathbb R}
        (1+\lambda^2)^{(a+b)/2}
        |r_t(\sqrt t\lambda)|.
\]
Put $u=\sqrt t\,\lambda$.  Then
\[
        (1+\lambda^2)^{(a+b)/2}
        =
        \left(1+\frac{u^2}{t}\right)^{(a+b)/2}
        \leq
        C_{a,b}\,t^{-(a+b)/2}(1+|u|)^{a+b}.
\]
Therefore
\[
\begin{aligned}
 \|r_t(\sqrt tD)\|_{H^{-a}\to H^b}
 &\leq
 C_{a,b}\,t^{-(a+b)/2}
 \sup_{u\in\mathbb R}(1+|u|)^{a+b}|r_t(u)|.
\end{aligned}
\]
Since $r_t=O(t^\infty)$ in every Schwartz seminorm, for every $N>0$ we have
\[
        \sup_{u\in\mathbb R}(1+|u|)^{a+b}|r_t(u)|
        =
        O(t^N).
\]
Because $N$ is arbitrary, the factor $t^{-(a+b)/2}$ can be absorbed, and
we obtain
\[
        \|r_t(\sqrt tD)\|_{H^{-a}\to H^b}
        =
        O(t^\infty).
\]
Thus
\[
        r_t(\sqrt tD)=O(t^\infty)
\]
as an operator from $H^{-a}$ to $H^b$, for all $a,b\geq 0$.

We now pass from Sobolev estimates to local smoothing seminorms.  Let
$K,L\subset M$ be compact subsets, and let $P_x$ and $Q_y$ be differential
operators, acting respectively in the left and right variables of the
Schwartz kernel.  Choose $\chi_1,\chi_2\in C_c^\infty(M)$ such that
$\chi_1\equiv 1$ in a neighbourhood of $K$ and $\chi_2\equiv 1$ in a
neighbourhood of $L$.  Set
\[
        T_t:=\chi_1 r_t(\sqrt tD)\chi_2 .
\]
Since multiplication by compactly supported smooth functions is bounded
between all Sobolev spaces, the estimate above implies
\[
        \|T_t\|_{H^{-a}\to H^b}=O(t^\infty)
\]
for all $a,b\geq 0$.

We claim that this implies rapid decay of all derivatives of the Schwartz
kernel of $T_t$ on $K\times L$.  Let $K_t(x,y)$ denote the Schwartz kernel
of $T_t$.  Fix the differential operator $Q_y$ of order $q$.  Choose
$a>q+n/2$, where $n=\dim M$.  Let $\delta_y$ be the Dirac distribution at $y$. Then the distributions 
\[
        Q_y^*\delta_y,\qquad y\in L,
\]
form a bounded family in $H^{-a}(M)$.  Indeed, by Sobolev embedding,
evaluation of derivatives up to order $q$ is continuous on $H^a(M)$ when
$a>q+n/2$, and the continuity is uniform for $y$ in the compact set $L$.

Next choose $b$ so large that Sobolev embedding controls the derivatives
appearing in $P_x$; more explicitly, if $\operatorname{ord}(P_x)=p$, take
$b>p+n/2$.  Then, uniformly for $y\in L$,
\[
\begin{aligned}
 \sup_{x\in K}
 |P_x Q_y K_t(x,y)|
 &=
 \sup_{x\in K}
 \left|
 P_x\bigl(T_t(Q_y^*\delta_y)\bigr)(x)
 \right|  \\
 &\leq
 C\,
 \|T_t(Q_y^*\delta_y)\|_{H^b}  \\
 &\leq
 C\,
 \|T_t\|_{H^{-a}\to H^b}
 \|Q_y^*\delta_y\|_{H^{-a}}  \\
 &=
 O(t^\infty).
\end{aligned}
\]
The constants are uniform for $x\in K$ and $y\in L$.  Hence
\[
        \sup_{x\in K,\ y\in L}
        |P_xQ_yK_t(x,y)|
        =
        O(t^\infty).
\]
Since $\chi_1=\chi_2=1$ near $K$ and $L$, the kernel of $T_t$ agrees with
the kernel of $r_t(\sqrt tD)$ on $K\times L$.  Therefore
\[
        \sup_{x\in K,\ y\in L}
        |P_xQ_yK_{r_t(\sqrt tD)}(x,y)|
        =
        O(t^\infty).
\]
This is precisely the assertion that $r_t(\sqrt tD)=O(t^\infty)$ in every
local smoothing seminorm.  The claim is proved.
\end{proof}

\begin{proposition}\label{prop:fp-heat}In any local smoothing seminorms, 
\[
R_{\operatorname{fp}}(t)-R(\sqrt{t}D)=O(t^\infty). 
\]
\end{proposition}

\begin{proof} We compare the entries of $R_{\mathrm{fp}}(t)$ with those of the heat
matrix $R(\sqrt t D)$.  Recall that
\[
        \psi(u)=e^{-u^{2}/2},\qquad
        \kappa_t(u):=\psi_t(u)-\psi(u).
\]

By Lemma~4.3,
\[
        \kappa_t=O(t^\infty)
\]
in every Schwartz seminorm.

We now apply Lemma \ref{lem:entry} entry by entry.  For the diagonal entries we have
\[
        \psi_t(u)^2-\psi(u)^2
        =
        \kappa_t(u)(\psi_t(u)+\psi(u)).
\]
Since $\psi_t=\psi+O(t^\infty)$ in every Schwartz seminorm, the factor
$\psi_t+\psi$ is uniformly bounded in all Schwartz seminorms.  Hence
\[
        \psi_t^2-\psi^2=O(t^\infty)
\]
in every Schwartz seminorm.  Therefore
\[
        S_\pm(t)^2-e^{-t\Delta_\pm}=O(t^\infty)
\]
in all local smoothing seminorms.

For the lower-left entry, the difference is represented by the odd
Schwartz function
\[
        u\bigl(\psi_t(u)-\psi(u)\bigr)=u\kappa_t(u).
\]
Multiplication by $u$ preserves the property of being $O(t^\infty)$ in
Schwartz seminorms.  Hence, by the claim,
\[
        S_-(t)\sqrt tD^+
        -
        e^{-t\Delta_-/2}\sqrt tD^+
        =
        O(t^\infty)
\]
in local smoothing seminorms.

It remains to compare the upper-right entry.  Using the intertwining
relation
\[
        \sqrt tD^-\,f(\sqrt{t\Delta_-})
        =
        f(\sqrt{t\Delta_+})\,\sqrt tD^-,
\]
we may write
\[
\begin{aligned}
 S_+(t)(1+S_+(t))Q(t)
 &=
 \psi_t(\sqrt{t\Delta_+})
 \bigl(1+\psi_t(\sqrt{t\Delta_+})\bigr)
 b_t(\sqrt{t\Delta_+})\sqrt tD^-  \\
 &=
 \psi_t(\sqrt{t\Delta_+})
 \frac{1-\psi_t(\sqrt{t\Delta_+})^2}{t\Delta_+}
 \sqrt tD^- .
\end{aligned}
\]
Thus this entry is represented by the odd scalar function
\[
        F_t(u):=\frac{\psi_t(u)(1-\psi_t(u)^2)}{u},
\]
with the removable singularity at $u=0$ understood by smooth extension.
The corresponding heat entry is represented by
\[
        F(u):=\frac{\psi(u)(1-\psi(u)^2)}{u}.
\]
Since $\psi_t(0)=\psi(0)=1$, both functions are smooth at the origin.
Moreover,
\[
\begin{aligned}
 F_t(u)-F(u)
 &=
 \frac{\psi_t(u)-\psi(u)}{u}
 \Bigl(1-\psi_t(u)^2-\psi_t(u)\psi(u)-\psi(u)^2\Bigr).
\end{aligned}
\]
The second factor is uniformly bounded in every Schwartz seminorm.  

We claim that
\[
        \frac{\psi_t-\psi}{u}=O(t^\infty)
\]
in every Schwartz seminorm.  Near $u=0$ the $O(t^\infty)$ estimate follows from
$\kappa_t(0)=0$ and the identity
\[
        \frac{\kappa_t(u)}{u}
        =
        \int_0^1 \kappa_t'(\lambda u)\,d\lambda ,
\]
together with the corresponding differentiated identities.  Away from
$u=0$, division by $u$ is a fixed smooth multiplier with the same symbol
estimates. And the same Schwartz estimates for $\kappa_t$ give the conclusion for $\frac{\kappa_t(u)}{u}$.  Therefore
\[
        F_t-F=O(t^\infty)
\]
in every Schwartz seminorm.  Applying  Lemma \ref{lem:entry} to the odd functional
calculus gives
\[
\begin{aligned}
&F_t(\sqrt{t}D) - F(\sqrt{t}D) = S_+(t)(1+S_+(t))Q(t)
-
\frac{1-e^{-t\Delta_+}}{t\Delta_+}
e^{-t\Delta_+/2}\sqrt tD^-        \\
&\hspace{4cm}
=
O(t^\infty)
\end{aligned}
\]
in local smoothing seminorms. 

Combining the estimates for the four matrix entries, we conclude that
\[
        R_{\mathrm{fp}}(t)-R(\sqrt tD)=O(t^\infty)
\]
in every local smoothing seminorm.
\end{proof}

\begin{corollary}\label{cor:rfptoR}
\[
\lim_{t\to 0}\rho(f_0\otimes \cdots \otimes f_k)(R_{\operatorname{fp}}(t), \cdots, R_{\operatorname{fp}}(t))=\lim_{t\to 0}\rho(f_0\otimes \cdots \otimes f_k)(R(\sqrt tD), \cdots, R(\sqrt tD))).
\]
\end{corollary}
\begin{proof} The equation follows from Proposition \ref{prop:fp-heat} and the following equation. 
\[
\begin{split}
&\rho(f_0\otimes \cdots \otimes f_k)(R_{\operatorname{fp}}(t), \cdots, R_{\operatorname{fp}}(t))-\rho(f_0\otimes \cdots \otimes f_k)(R(\sqrt tD), \cdots, R(\sqrt tD)))\\
=&\sum \rho(f_0\otimes \cdots \otimes f_k)(R(\sqrt tD), \cdots, R_{\operatorname{fp}}(t)-R(\sqrt tD), \cdots R_{\operatorname{fp}}(t))\\
\end{split}
\]
\end{proof}

\section{New results on Volterra pseudodifferential operators}\label{section:volterra-main}
This section and the Appendix are devoted  to the theory of Volterra pseudodifferential operators. Most of the results presented in the Appendix are classic whereas the material presented in this section is new.
In particular, we present in this section
results about Getzler rescaling and commutators; this technical results  will play a major role
in the proof of our higher index theorem.

Let $M$ be a smooth compact manifold without boundary and let $D$ be an $L^2$-invertible
 Dirac-type operator acting on the sections of a bundle of Clifford modules $E$. Then it is well known that $D^{-1}\in  \Psi^{-1} (M,E)$. One way to introduce the Volterra calculus is to consider the heat operator $\partial_t + D^2$. Acting on $C^\infty_+ (M,E)$, the subspace of 
 $C^\infty_+ (M,E)$ consisting of sections supported on $M\times (-\infty,c]$  for some $c\in 
 \mathbb{R}$, this operator is invertible; its inverse   $(\partial_t + D^2)^{-1}$ is a Volterra
 pseudodifferential operator of order $(-2)$. Put it differently, in the same way that classic
 pseudodifferential operators are a receptacle for the inverses of elliptic operators, the Volterra calculus is a receptacle for the inverse of the parabolic operator $\partial_t + D^2$.
 
 We refer the reader to the Appendix for
 
 \begin{itemize}
 \item an introduction to the basic concepts;
 \item the definition of $\Gamma$-equivariant Volterra calculus;
 \item results relative to Getzler rescaling in the context of the Volterra calculus.
 \end{itemize}
  In this Section we only concentrate 
 on results that are not available in the literature.

\subsection{Getzler order and simplification}\label{subsect:simplification}
In the definition of the index pairing we shall need to consider the composition of heat-type operators with $t$ fixed. Such composition is of course different from the composition
of the corresponding Volterra operators within the  Volterra calculus, where composition in the $t$-variable, by convolution, appears. Such a difference prevents us from directly applying the Getzler-Volterra calculus to the short-time asymptotics computation of the index pairing. To overcome this problem, we develop in this subsection the tools that allow us to move operators in the index pairing 
and to bring them in a sort of normal form, a form to which we can apply the results
of Ponge and Wang. One computational complication is the following, already present in \cite{Ponge-Wang-2} and \cite{LYZ}:  if $\gamma\in \Gamma$
we can usually consider two coordinate systems near a point $x\in M^\gamma$; one is the normal coordinates system centered at $x$ and the other is the coordinate system defined by tubular neighborhood coordinates for $M^\gamma$, involving the normal bundle to $M^\gamma$ near $x$. While the usual Volterra calculus works more conveniently with the tubular neighborhood coordinates, the Volterra-Getzler calculus works better with normal coordinates. Thanks to the formula connecting symbols in different coordinates systems
we can relate the two approaches; while this is conceptually clear, it becomes computationally
rather involved, and this explains the length of this subsection.

\bigskip
Let $\sigma: \operatorname{Cliff}(T M) \rightarrow \Lambda^\bullet T_{\mathbb{C}} M$ be the symbol map. We take $x \in M$. Let $Q_i, i = 1, \cdots, l$ be Volterra operators. Moreover, let $\Psi$ be a function supported in a small neighborhood of $x$. Put
\[
\widetilde{J}(x,\gamma,  t):=\int_{M^{\times(k-1)}} K_{Q_1}\left(x, x_1, t\right)\Psi(x_1)  K_{Q_2}\left(x_1, x_2, t\right) \cdots \Psi(x_{k-1}) K_{Q_k}\left(x_{k-1}, \gamma x, t\right) d x_1 \cdots d x_{k-1}
\]
As in \cite{Ponge-Wang-2} we fix tubular coordinates near $x_0\in M_a^\gamma$ (where $M_a^\gamma$ is a connected component of $M^\gamma$ of dimension $a$); we write $x$ in the coordinates associated with such a $\gamma$-invariant chart as $(y,w)\in \mathbb{R}^a\times \mathbb{R}^{n-a}$ via $\phi_{\rm tub}(y,w)$, with $\phi_{\rm tub}$ the inverse of the tubular coordinate chart.
Because $\Psi$ is supported in a neighbourhood $U_0$ of  $x$ corresponding to  a neighbourhood in $\mathbb{R}^n$, we can write
\begin{multline*}
\widetilde{J}(x, \gamma, t):=\int_{\left( \mathbb{R}^n\right)^{\times(k-1)}}K_{Q_1}\left(x, \phi_{\rm tub}(y_1,w_1), t\right)\Psi(\phi_{\rm tub}(y_1,w_1))\\  K_{Q_2}\left(\phi_{\rm tub}(y_1,w_1),\phi_{\rm tub}(y_2,w_2), t\right) \cdots \Psi(\phi_{\rm tub}(y_{k-1},w_{k-1})) K_{Q_k}\left(\phi_{\rm tub}(y_{k-1},w_{k-1}), \gamma x, t\right) \;d y_1 dw_1\cdots d y_{k-1}dw_{k-1},
\end{multline*}
 where $dydw$ is the pullback volume form on the chart via the map $\phi_{\rm tub}$. Assume that  $Q_i$ is a Volterra operator of order $s_i$. Then in the tubular coordinates chart we have fixed at $x$ this operator $Q_i$ will correspond to an operator 
in the euclidean space  $\mathbb{R}^n$ that we call $Q^{\rm tub}_i$ with symbol $q$; we can then write, with a small
but common abuse of notation, 
\begin{equation}\label{equ:KQi}
K_{Q_i}(\phi_{\rm tub}(y,w), \phi_{\rm tub}(y',w'), t) \sim \sum_{j \geq 0} \check{q}_{s_i-j}(y,w, y'-y,w'-w, t).
\end{equation}

Consider $x\in M^\gamma_a$ and write accordingly $x=\phi_{\rm tub}(y,0)$.
We then consider 
\begin{multline}
\label{equ:Jterm}
J(x, \gamma, t):=\int_{\mathbb{R}^{n-a}} \int_{{\mathbb{R}^n}^{\times(k-1)}}K_{Q_1}\left(\phi_{\rm tub}(y,w), \phi_{\rm tub}(y_1,w_1), t\right)\Psi(\phi_{\rm tub}(y_1,w_1))\\  K_{Q_2}\left(\phi_{\rm tub}(y_1,w_1),\phi_{\rm tub}(y_2,w_2), t\right) \cdots \Psi(\phi_{\rm tub}(y_{k-1},w_{k-1})) K_{Q_k}\left(\phi_{\rm tub}(y_{k-1},w_{k-1}), \phi_{\rm tub}(y,d\gamma|_{y}(w)), t\right) \\ 
d y_1dw_1 \cdots d y_{k-1}dw_{k-1}	dw,
\end{multline}
where in the above we have used the property that $\gamma(\phi_{\rm tub}(y,w))=\phi_{\rm tub}(y,d\gamma_y(w))$.

For notational convenience, we will in the following developments combine the function $\Psi(\phi_{\rm tub}(y_{i-1},w_{i-1})) $ with $ K_{Q_i}\left(\phi_{\rm tub}(y_{i-1},w_{i-1}), \phi_{\rm tub}(y_{i},w_i)), t\right) $ by replacing $Q_i$ with $\Psi(\phi_{\rm tub}(y_{i-1},w_{i-1})) Q_i$ so that the kernel of $Q_i$ has compact support in the first component. And the integral (\ref{equ:Jterm}) in the definition of $J(x,\gamma, t)$ is reduced to a product of $\epsilon$-balls with a proper choice of $\epsilon>0$, i.e. 
\begin{multline}
\label{equ:Jterm-epsilon}
J(x, \gamma, t):=\int_{\mathbb{B}^{n-a}(\epsilon)} \int_{{\mathbb{B}^n(\epsilon)}^{\times(k-1)}}K_{Q_1}\left(\phi_{\rm tub}(y,w), \phi_{\rm tub}(y_1,w_1), t\right)\\  K_{Q_2}\left(\phi_{\rm tub}(y_1,w_1),\phi_{\rm tub}(y_2,w_2), t\right)\cdots K_{Q_k}\left(\phi_{\rm tub}(y_{k-1},w_{k-1}), \phi_{\rm tub}(y,d\gamma|_{y}(w)), t\right) \\ 
d y_1dw_1 \cdots d y_{k-1}dw_{k-1}	dw,
\end{multline}

\begin{lemma}\label{lem:Jy} Consider Volterra operators $Q_i$ with (Volterra) order $m'_i$, $i=1,\cdots, k$ with symbol in tubular coordinates around $x\in M_a^\gamma$,
\[
q_i\sim \sum_{s_i} q_{i, s_i},\ s_i\leq m'_i. 
\] 
Assume that the Schwartz kernel of each $Q_i$ has compact support in the first component. Recall that
 for $x\in M^\gamma_a$  we have $x=\phi_{\rm tub}(y,0)$.
As $t \downarrow 0$, we have the following asymptotic expansion:
\begin{multline}\label{eq:Jxgammat}
J(x, \gamma, t)\equiv J(y, \gamma, t) \sim  \sum_{s_i \leq m'_i} \sum_{\alpha_{i, s_i}, \beta_{i, s_i}} t^{\frac{1}{2}\left(-\sum_{i=1}^k s_i+\sum_{i=1}^{k}(\left|\alpha_{i, s_i}\right|+|\beta_{i, s_i}|)-a\right)-k} \\
\int_{\mathbb{R}^{n-a}}\int_{\left(\mathbb{R}^n\right)^{\times(k-1)}} \frac{y^{\alpha_{1, s_1}}w^{\beta_{1, s_1}}}{\alpha_{1, s_1}!\beta_{1, s_1}!}\partial^{\alpha_{1, s_1}}_y\partial_w^{\beta_{1, s_1}}\check{q}_{1, s_1}\left(0,0, y_1-y,w_1-w, 1\right) \\
\cdot \frac{y_1^{\alpha_{2, s_2}}w_1^{\beta_{2, s_2}}}{\alpha_{2, s_2}!\beta_{2, s_2}!} \partial^{\alpha_{2, s_2}}_y \partial_w^{\beta_{2, s_2}}\check{q}_{2, s_2}\left(0, 0,y_2-y_1, w_2-w_1, 1\right) \cdots\\ 
\cdots \frac{y_{k-1}^{\alpha_{k, s_k}}w_{k-1}^{\beta_{k, s_k}}}{\alpha_{k, s_k}!\beta_{k, s_k}!} \partial^{\alpha_{k, s_k}}_y\partial_w^{\beta_{k, s_k}} \check{q}_{k, s_k}\left(0, 0, y -y_{k-1},d\gamma_y (w)-w_{k-1}, 1\right) dy_1dw_1 \cdots d y_{k-1} dw_{k-1}dw.	
\end{multline}
Here $\partial_y^{\alpha}$ (and $\partial_w^\beta$) denotes differentiation with respect to the $y$ (and $w$) variables and in  $\check{q}$ for a multi-indices $\alpha$ (and $\beta$).
\begin{proof}
By (\ref{equ:KQi}) and (\ref{equ:Jterm-epsilon}), we have the following expression of $J(y, \gamma, t)$, 
\begin{multline}
J(y, \gamma, t) \sim\int_{\mathbb{B}^{n-a}(\epsilon)} \int_{{\mathbb{B}^n(\epsilon)}^{\times(k-1)}} \left(\sum_{j_1 \geq 0} \check{q}_{1,m'_1-j_1}(y, w, y_1- y, w_{1}-w, t)\right)\\
\left(\sum_{j_2 \geq 0} \check{q}_{2, m'_2-j_2}(y_1, w_1, y_2-y_1, w_{2}-w_1, t)\right)\ldots\\
\left(\sum_{j_k \geq 0} \check{q}_{k,m'_k-j_k}(y_{k-1},v_{k-1}, y-y_{k-1}, d\gamma_y (w)-w_{k-1}, t)\right) \;dy_1dw_1 \cdots d y_{k-1} dw_{k-1}dw
\end{multline}	

We need to work with general $(y, w)$ instead of $(0,w)$ to show that our $t$-growth estimate on $J(y, \gamma, t)$ is continuous (uniformly) with respect to $y$ in a tubular neighborhood of $x_0$. We replace $w_i$ by $\sqrt{t}w_i$,$w$ by $\sqrt{t}w$, $y_i$ by $\sqrt{t}y_i$, and $y$ by $\sqrt{t}y$. Then  
\begin{multline}\label{equ:Jy}
J(y, \gamma, t) \sim t^{\frac{kn -a}{2}}\int_{\mathbb{B}^{n-a}(\frac{\epsilon}{\sqrt{t}})} \int_{{\mathbb{B}^n((\frac{\epsilon}{\sqrt{t}})}^{\times(k-1)}}
 \left(\sum_{j_1 \geq 0} \check{q}_{1,m'_1-j_1}(\sqrt{t}y,\sqrt{t} w, \sqrt{t}y_1 -\sqrt{t}y, \sqrt{t}w_{1}-\sqrt{t}w, t)\right)\\
\left(\sum_{j_2 \geq 0} \check{q}_{2, m'_2-j_2}(\sqrt{t}y_1, \sqrt{t}w_1,\sqrt{t} y_2-\sqrt{t}y_1,\sqrt{t} w_{2}-\sqrt{t}w_1, t)\right)\ldots\\
 \left(\sum_{j_k \geq 0} \check{q}_{k,m'_k-j_k}(\sqrt{t}y_{k-1},\sqrt{t}v_{k-1},\sqrt{t}y-\sqrt{t}y_{k-1}, \sqrt{t}d\gamma_y (w)-\sqrt{t}w_{k-1}, t)\right) \;dy_1dw_1 \cdots d y_{k-1} dw_{k-1}dw
 \end{multline}
 By the homogeneity of the symbols $q_{i, m'_i-j_i}$ ($i=1, ..., k$), we see
 \begin{multline}
 J(y, \gamma, t) \sim t^{\frac{\sum_{i=1}^k (j_i-m'_i)-a}{2}-k }\int_{\mathbb{B}^{n-a}(\frac{\epsilon}{\sqrt{t}})} \int_{{\mathbb{B}^n((\frac{\epsilon}{\sqrt{t}})}^{\times(k-1)}}
 \left(\sum_{j_1 \geq 0} \check{q}_{1,m'_1-j_1}(\sqrt{t}y,\sqrt{t} w, y_1-y , w_{1}-w, 1)\right)\\
\left(\sum_{j_2 \geq 0} \check{q}_{2, m'_2-j_2}(\sqrt{t}y_1, \sqrt{t}w_1, y_2-y_1, w_{2}-w_1, 1)\right)\ldots\\
 \left(\sum_{j_k \geq 0} \check{q}_{k,m'_k-j_k}(\sqrt{t}y_{k-1},\sqrt{t}v_{k-1}, y-y_{k-1}, d\gamma_y (w)-w_{k-1}, 1)\right) \;dy_1dw_1 \cdots d y_{k-1} dw_{k-1}dw.
\end{multline}
By Taylor's formula, we can expand $\check{q}_{i, m'_i-j_i}$ ($i=1, ..., k$) as follows, 
\begin{align}\label{equ:taylor-1}
\check{q}_{1,s_1}(\sqrt{t}y,\sqrt{t} w, y_1-y , w_{1}-w, 1)=\sum_{\alpha_{1, s_1}, \beta_{1, s_1}} t^{\frac{\left|\alpha_{1, s_1}\right| + \left|\beta_{1, s_1}\right|}{2}} \cdot \frac{y^{\alpha_{1, s_1}}w^{\beta_{1, s_1}}}{\alpha_{1, s_1}!\beta_{1, s_1}!}\partial^{\alpha_{1, s_1}}_y\partial_w^{\beta_{1, s_1}}\check{q}_{1, s_1}\left(0,0, y_1-y,w_1-w, 1\right),
\end{align}
and for $i=2, \cdots, k-1$
\begin{align}\label{equ:taylor-2}
&\check{q}_{i, s_i}(\sqrt{t}y_{i-1}, \sqrt{t}w_{i-1}, y_i-y_{i-1}, w_{i}-w_{i-1}, 1)\\
=&\sum_{\alpha_{i, s_i}, \beta_{i, s_i}}  t^{\frac{\left|\alpha_{i, s_i}\right| + \left|\beta_{i, s_i}\right|}{2}} \cdot\frac{y_{i-1}^{\alpha_{i, s_i}}w_{i-1}^{\beta_{i, s_i}}}{\alpha_{i, s_i}!\beta_{i, s_i}!}\partial^{\alpha_{i, s_i}}_y\partial_w^{\beta_{i, s_i}}\check{q}_{i, s_i}\left(0,0, y_i - y_{i-1},w_i-w_{i-1}, 1\right),\end{align}
and
\begin{multline}\label{equ:taylor-3}
 \check{q}_{k,s_k}(\sqrt{t}y_{k-1},\sqrt{t}v_{k-1}, -y_{k-1}, d\gamma_y (w)-w_{k-1}, 1)\\=\sum_{\alpha_{k, s_k}, \beta_{k, s_k}} t^{\frac{\left|\alpha_{k, s_k}\right| + \left|\beta_{k, s_k}\right|}{2}} \cdot \frac{y_{k-1}^{\alpha_{k, s_k}}w_{k-1}^{\beta_{k, s_k}}}{\alpha_{k, s_k}!\beta_{k, s_k}!}\partial^{\alpha_{k, s_k}}_y\partial_w^{\beta_{k, s_k}}\check{q}_{k, s_k}\left(0,0, -y_{k-1}, d\gamma_y (w)-w_{k-1}, 1\right).
\end{multline}
Combining (\ref{equ:Jy}) - (\ref{equ:taylor-3}), we obtain  the following 
\begin{multline*}
J(x, \gamma, t) \sim  \sum_{s_i \leq m'_i} \sum_{\alpha_{i, s_i}, \beta_{i, \beta_i}} t^{\frac{1}{2}\left(-\sum_{i=1}^k s_i+\sum_{i=1}^{k}(\left|\alpha_{i, s_i}\right|+|\beta_{i, s_i}|)-a\right)-k} \\
\int_{\mathbb{R}^{n-a}}\int_{\left(\mathbb{R}^n\right)^{\times(k-1)}} \frac{y^{\alpha_{1, s_1}}w^{\beta_{1, s_1}}}{\alpha_{1, s_1}!\beta_{1, s_1}!}\partial^{\alpha_{1, s_1}}_y\partial_w^{\beta_{1, s_1}}\check{q}_{1, s_1}\left(0,0, y_1,w_1-w, 1\right) \cdot 
\\
 \frac{y_1^{\alpha_{2, s_2}}w_1^{\beta_{2, s_2}}}{\alpha_{2, s_2}!\beta_{2, s_2}!} \partial^{\alpha_{2, s_2}}_y \partial_w^{\beta_{2, s_2}}\check{q}_{2, s_2}\left(0, 0,y_2-y_1, w_2-w_1, 1\right) \cdots\\ 
 \cdots \frac{y_{k-1}^{\alpha_{k, s_k}}w_{k-1}^{\beta_{k, s_k}}}{\alpha_{k, s_k}!\beta_{k,s_k}!} \partial^{\alpha_{k, s_k}}_y\partial_w^{\beta_{k, s_k}} \check{q}_{k, s_k}\left(0, 0, y-y_{k-1},d\gamma_y (w)-w_{k-1}, 1\right) dy_1dw_1 \cdots d y_{k-1} dw_{k-1}dw	
\end{multline*}	
which is what we wanted to prove.
\end{proof}

\end{lemma}
\begin{remark}
We point out that it follows from the symbol growth property (\ref{eq:symbol-growth}) and the homogeneity property of the anti-Fourier transform of the symbol that the function $y^\alpha w^\beta\partial_y^\alpha \partial_w^\beta \check{q}_{i, s_i}(0,0, y, w)$ has rapid decay property away from the origin. So the above integral  in Lemma \ref{lem:Jy} converges absolutely. 
\end{remark}

\begin{lemma}\label{lem:symbolest}
If $Q_i$ has Getzler order $m_i$ along $M^\gamma$ and  Volterra order $m'_i$, then 
$$
\sigma\left[\gamma^{-1} J(y, \gamma, t)\right]^{(j)} \sim O(t^{\frac{-\sum_{i}m_i+j-a-2k}{2}}).
$$	
\begin{proof}

To compute the Getzler order of $Q_i$, we look at the symbol\footnote{Notice that our notation for symbols ($q$ in tubular coordinates and $\tilde{q}$ in normal coordinates)
is reversed with respect to \cite{Ponge-Wang-2} } $\tilde{q}_i(z, \eta)$  of $Q_i$ in the normal coordinate chart $(U, \varphi_U)$ centered at $x\in M^\gamma$ via the exponential map $\exp_x: \mathbb{R}^n=T_xM=\mathbb{R}^a\times \mathbb{R}^{n-a}\to M$. Let $\kappa=\varphi_U\circ \phi_{\rm tub}  : \mathbb{R}^n=\mathbb{R}^a\times \mathbb{R}^{n-a} \to \mathbb{R}^n$. Then $d\kappa_{(y, w)}: T_{(y,w)}\mathbb{R}^n\to T_{\kappa(y,w)} \mathbb{R}^n$ is the induced map between tangent spaces. We use $\theta_{(y,w)}$ to denote the corresponding map ${d\kappa^*}^{-1}_{(y,w)} : T^*_{(y,w)}\mathbb{R}^n\to T^*_{\kappa(y,w)} \mathbb{R}^n$. 

The change of variable formula for symbols \cite[Theorem 18.1.17]{hormander-3} gives 
\begin{equation}\label{eq:qi}
\begin{aligned}
q_i&=\sum_{s_i}q_{i, s_i},\\
\ \ \ q_{i, s_i}(y,w,\xi, \nu,\tau)&=\sum_{{\tiny \begin{array}{c}l_i-|\beta^1_{i, l_i}|-|\beta^2_{i, l_i}|+|\gamma_{i, l_i}|=s_i,\\ 2(|\beta^1_{i, l_i}|+|\beta^{2}_{i, l_i} |)\leq |\gamma_i|\end{array}}}a_{i, l_i, \beta^{j}_{i, l_i}, \gamma_{i, l_i}}(y, w)\xi^{\beta^1_{i, l_i}}\nu^{\beta^2_{i, l_i}} \partial_\eta ^{\gamma_{i, l_i}}\tilde{q}_{i, l_i}(\kappa(y,w), \theta_{(y,w)}(\xi, \nu), \tau),
\end{aligned}
\end{equation}
where $a_{i, l_i, \beta^j_{i, l_i}, \gamma_{i, l_i}}$ is a smooth function such that $a_{i, l_i, \beta^j_{i, l_i}, \gamma_{i, l_i}}=1$ when $\beta^j_{i, l_i}=\gamma_{i, l_i}=0$ and $Q_i$ has order $m_i'$.

We compute the inverse Fourier transform of $q_{i, s_i}$ in the $\xi, \nu, \tau$ direction,
\begin{equation}\label{eq:inverse-q}
\begin{aligned}
&\check{q}_{i, s_i}(y, w,y', w',t )\\
=&\sum_{{\tiny \begin{array}{c}l_i-|\beta^1_{i, l_i}|-|\beta^2_{i, l_i}|+|\gamma_{i, l_i}|=s_i,\\ 2(|\beta^1_{i, l_i}|+|\beta^2_{i, l_i}|)\leq |\gamma_{i, l_i}|\end{array}}} a_{i, l_i, \beta^{j}_{i, l_i}, \gamma_{i, l_i}}(y, w)\\
&\hspace{2cm} \big(\xi^{\beta^1_{i, l_i}}\nu^{\beta^2_{i, l_i}} \partial_\eta ^{\gamma_{i, l_i}}\tilde{q}_{i, l_i}(\kappa(y,w), \theta_{(y,w)}(\xi, \nu), \tau)\big)^\vee(y', w', t)\\
\sim &\sum_{{\tiny \begin{array}{c}l_i-|\beta^1_{i, l_i}|-|\beta^2_{i, l_i}|+|\gamma_{i, l_i}|=s_i,\\ 2(|\beta^1_{i, l_i}|+|\beta^2_{i, l_i}|)\leq |\gamma_{i, l_i}|\end{array}}} \partial_y^{\alpha^1_{i, l_i}}\partial_w ^{\alpha^2_{i, l_i}}b_{i, l_i, \alpha^j_{i, l_i}, \beta^{j}_{i, l_i}, \gamma_{i, l_i}}(0,0)\frac{y^{\alpha^1_{i, l_i}}}{\alpha^1_{i, l_i}!}\frac{w^{\alpha^2_{i, l_i}}}{\alpha^2_{i, l_i}!}\frac{y^{\alpha^3_{i, l_i}}}{\alpha^3_{i, l_i}!}\frac{w^{\alpha^4_{i, l_i}}}{\alpha^4_{i, l_i}!}\\
&\hspace{2cm}\cdot \big(\xi^{\beta^1_{i, l_i}}\nu^{\beta^2_{i, l_i}} \partial_y^{\alpha^3_{i, l_i}}\partial_w^{\alpha^4_{i, l_i}}\partial_\eta ^{\gamma_{i, l_i}}\tilde{q}_{i, l_i}(\kappa(y,w), \theta_{(y,w)}(\xi, \nu), \tau)|_{(y,w)=(0,0)}\big)^{\vee}(y', w', t) \\
=&\sum_{{\tiny \begin{array}{c}l_i-|\beta^1_{i, l_i}|-|\beta^2_{i, l_i}|+|\gamma_{i, l_i}|=s_i,\\ 2(|\beta^1_{i, l_i}|+|\beta^2_{i, l_i}|)\leq |\gamma_{i, l_i}|\end{array}}} \tilde{b}_{i, l_i, \alpha^j_{i, l_i}, \beta^{j}_{i, l_i}, \gamma_{i, l_i}}(0,0)y^{\alpha^1_{i, l_i}+\alpha^3_{i, l_i}}w^{\alpha^2_{i, l_i}+\alpha^4_{i, l_i}}\\
&\hspace{2cm}\cdot \big(\xi^{\beta^1_{i, l_i}}\nu^{\beta^2_{i, l_i}} \partial_y^{\alpha^3_{i, l_i}}\partial_w ^{\alpha^4_{i, l_i}}\partial_\eta ^{\gamma_{i, l_i}}\tilde{q}_{i, l_i}(\kappa(y,w), \theta_{(y,w)}(\xi, \nu), \tau)|_{(y,w)=(0,0)}\big)^{\vee}(y', w', t).
\end{aligned}
\end{equation}

Consider the following expression in the above expansion,
\[
\partial_y^{\alpha^3_{i, l_i}}\partial_w^{\alpha^4_{i, l_i}}\partial_\eta ^{\gamma_{i, l_i}}\tilde{q}_{i, l_i}(\kappa(y,w), \theta_{(y,w)}(\xi, \nu), \tau)|_{(y,w)=(0,0)}. 
\] 
Applying the chain rule to compute the partial differentiation with respect to $y$ and $w$, we obtain the following formula,
\begin{equation}\label{eq:delydelw}
\begin{split}
&\partial_y^{\alpha^3_{i, l_i}}\partial_w^{\alpha^4_{i, l_i}}\partial_\eta ^{\gamma_{i, l_i}}\tilde{q}_{i, l_i}(\kappa(y,w), \theta_{(y,w)}(\xi, \nu), \tau)|_{(y,w)=(0,0)}\\
=&\sum_{|\sigma^1_{i, l_i}|+|\sigma^2_{i, l_i}|= |\delta^2_{i, l_i}|}g_{\delta^j_{i, l_i}, \sigma^i_{i, l_i}} \xi^{\sigma^1_{i, l_i}} \nu^{\sigma^2_{i, l_i}} \big(\partial^{\delta^1_{i, l_i}}_z\partial_\eta^{\delta^2_{i, l_i}+\gamma_{i, l_i}}\tilde{q}_{i, l_i}(\kappa(y,w), \theta_{(y,w)}(\xi, \nu), \tau)\big) |_{(y,w)=(0,0)},
\end{split}
\end{equation}
where $g_{\delta^j_{i, l_i}, \upsilon_{i, l_i}} $ is a constant. The inequality $|\sigma^1_{i, l_i}|+|\sigma^2_{i, l_i}|=|\delta^2_{i, l_i}|$ is from the observation that the $\xi^{\sigma^1_{i, l_i}}\nu^{\sigma^2_{i,l_i}}$ on the right side of (\ref{eq:delydelw}) comes from the chain rule and differentiating the Jacobian $\theta_{(y,w)}(\xi, \nu)$ with respect to the variables $y$ and $w$. Furthermore, we notice that $\kappa(0,0)=0$ and $d\kappa|_{(0,0)}={\rm Id}$. 

Substituting the simplified expression (\ref{eq:delydelw}) into the expansion (\ref{eq:inverse-q}) of $\check{q}_{i, s_i}(y,w,y', w',t )$, we obtain
\begin{equation}
\label{eq:inverse-q-expansion}
\begin{split}
&\check{q}_{i, s_i}(y, w,y', w',t )\\
=&\sum_{{\tiny \begin{array}{c}l_i-|\beta^1_{i, l_i}|-|\beta^2_{i, l_i}|+|\gamma_{i, l_i}|=s_i,\\ 2(|\beta^1_{i, l_i}|+|\beta^{2}_{i, l_i} |)\leq |\gamma_{i, l_i}|\end{array}}} a_{i, l_i, \alpha^j_{i, l_i}, \beta^{j}_{i, l_i}, \gamma_{i, l_i}, \delta^j_{i, l_i}, \sigma^j_{i, l_i}} y^{\alpha^1_{i, l_i}+\alpha^3_{i, l_i}}w^{\alpha^2_{i, l_i}+\alpha^4_{i, l_i}}\\
&\cdot \big( \xi^{\beta^1_{i, l_i}+\sigma^1_{i, l_i}} \nu^{\beta^2_{i, l_i}+\sigma^2_{i, l_i}} \big(\partial^{\delta^1_{i, l_i}}_z\partial_\eta^{\delta^2_{i, l_i}+\gamma_{i, l_i}}\tilde{q}_{i, l_i}((0, 0), (\xi, \nu), \tau)\big)^\vee (y', w', t).
\end{split}
\end{equation}

Applying (\ref{eq:inverse-q-expansion}) to the expression of $J(x, \gamma, t)$ in (\ref{eq:Jxgammat}), we obtain 
\begin{equation}
\label{eq:sigma-I}
\begin{split}
\sigma\left[\gamma^{-1} J(y, \gamma, t)\right]^{(j)}\sim &\sum_{{\tiny \begin{array}{cc}2(|\beta^1_{i, l_i}|+|\beta^2_{i, l_i}|)\leq |\gamma_{i, l_i}| \end{array}}} t^{\frac{\sum_{i=1}^{k}( -l_i+|\beta^1_{i, l_i}|+|\beta^2_{i ,l_i}|-|\gamma_{i, l_i}|+|\alpha^1_{i, l_i}|+|\alpha^2_{i,l_i}|+|\alpha^3_{i, l_i}|+|\alpha^4_{i, l_i}|)-(a+2k)}{2}} \\
&\hspace{4cm}I^{\underline{j}}_{ \underline{l}, \underline{\alpha^1}, \underline{\alpha^2},\underline{\alpha^3}, \underline{\alpha^4}, \underline{\beta^1}, \underline{\beta^2}, \underline{\delta^1}, \underline{\delta^2}, \underline{\gamma}, \underline{\sigma^1}, \underline{\sigma^2}}, 
\end{split}
\end{equation}
where $\underline{j}=(j_1, \cdots, j_k)$, $\underline{l}=(l_1, \cdots, l_k),\ \underline{\alpha^{j}}=(\alpha^j_{1, l_1},\cdots, \alpha^j_{k, l_k}),\ j=1,2,3,4,\ \underline{\beta^j}=(\beta^j_{1, l_1}, \cdots, \beta^j_{k, l_k}),\ l=1,2,\ \underline{\gamma}=(\gamma_{1, l_1}, \cdots, \gamma_{k, l_k}),\ \underline{\delta^j}=(\delta^j_{1, l_1}, \cdots, \delta^j_{k, l_k}),\  j=1,2,\ \underline{\sigma^j}=(\sigma^j_{1, l_1}, \cdots, \sigma^j_{k, l_k})$, $j=1,2$, and 
\begin{multline*}
I_{ \underline{l}, \underline{\alpha^1}, \underline{\alpha^2},\underline{\alpha^3}, \underline{\alpha^4}, \underline{\beta^1}, \underline{\beta^2}, \underline{\delta^1}, \underline{\delta^2}, \underline{\gamma}, \underline{\sigma^1}, \underline{\sigma^2}}\\
:=
\int_{\mathbb{R}^{n-a}}\int_{(\mathbb{R}^n)^{\times (k-1)} } a_{1, l_1, \alpha^j_{1, l_1}, \beta^{j}_{1, l_1}, \gamma_{1, l_1}, \delta^j_{1, l_1}, \sigma^j_{1, l_1}} y^{\alpha^1_{1, l_1}+\alpha^3_{1, l_1}}w^{\alpha^2_{1, l_1}+\alpha^4_{1, l_1}}\\
\cdot \big( \xi^{\beta^1_{1, l_1}+\sigma^1_{1, l_1}} \nu^{\beta^2_{1, l_1}+\sigma^2_{1, l_1}} \big(\partial^{\delta^1_{1, l_1}}_z\partial_\eta^{\delta^2_{1, l_1}+\gamma_{1, l_1}}\tilde{q}_{1, l_1}((0, 0), (\xi, \nu), \tau)^{(j_1)}\big)^\vee (y_1-y, w_1-w, 1).
\\
\prod_{i=2} ^{k-1} a_{i, l_i, \alpha^j_{i, l_i}, \beta^{j}_{i, l_i}, \gamma_{i, l_i}, \delta^j_{i, l_i}, \sigma^j_{i, l_i}} y_{i-1}^{\alpha^1_{i, l_i}+\alpha^3_{i, l_i}}w_{i-1}^{\alpha^2_{i, l_i}+\alpha^4_{i, l_i}}\\
\cdot \big( \xi^{\beta^1_{i, l_i}+\sigma^1_{i, l_i}} \nu^{\beta^2_{i, l_i}+\sigma^2_{i, l_i}} \big(\partial^{\delta^1_{i, l_i}}_z\partial_\eta^{\delta^2_{i, l_i}+\gamma_{i, l_i}}\tilde{q}_{i, l_i}((0, 0), (\xi, \nu), \tau)^{(j_i)}\big)^\vee (y_i-y_{i-1}, w_i-w_{i-1}, 1).
\\
a_{k, l_k, \alpha^j_{k, l_k}, \beta^{j}_{k, l_k}, \gamma_{k, l_k}, \delta^j_{k, l_k}, \sigma^j_{k, l_k}} y_{k-1}^{\alpha^1_{k, l_k}+\alpha^3_{k, l_k}}w_{k-1}^{\alpha^2_{k, l_k}+\alpha^4_{k, l_k}}\\
\cdot \big( \xi^{\beta^1_{i, l_i}+\sigma^1_{i, l_i}} \nu^{\beta^2_{i, l_i}+\sigma^2_{i, l_i}} \big(\partial^{\delta^1_{i, l_i}}_z\partial_\eta^{\delta^2_{i, l_i}+\gamma_{i, l_i}}\tilde{q}_{i, l_i}((0, 0), (\xi, \nu), \tau)^{(j_k)}\big)^\vee (y-y_{k-1}, d\gamma_y(w)-w_{k-1}, 1).
\\
dy_1dw_1\cdots dy_{k-1}dw_{k-1}dw, 
\end{multline*}
and $j=j_1+\cdots+j_k$. 

We observe that the term 
\[
y_{i-1}^{\alpha^1_{i, l_i}+\alpha^3_{i, l_i}}w_{i-1}^{\alpha^2_{i, l_i}+\alpha^4_{i, l_i}}\\
\cdot \big( \xi^{\beta^1_{i, l_i}+\sigma^1_{i, l_i}} \nu^{\beta^2_{i, l_i}+\sigma^2_{i, l_i}} \big(\partial^{\delta^1_{i, l_i}}_z\partial_\eta^{\delta^2_{i, l_i}+\gamma_{i, l_i}}\tilde{q}_{i, l_i}((0, 0), (\xi, \nu), \tau)\big)^\vee (y_i-y_{i-1}, w_i-w_{i-1}, 1)
\]
is Getzler homogeneous of order 
$$
-|\alpha^1_{i, l_i}|-|\alpha^2_{i, l_i}|-|\alpha^3_{i, l_i}|-|\alpha^4_{i, l_i}|+l_i+j_i.
$$ 
For any $i, l_i, \alpha^j_{i, \lambda_i}, \beta^j_{i, \lambda_i}, \gamma_{i, l_i}, \delta^j_{i, l_i}, \sigma^j_{i, l_i}$ with 
\[
-|\alpha^1_{i, l_i}|-|\alpha^2_{i, l_i}|-|\alpha^3_{i, l_i}|-|\alpha^4_{i, l_i}|+l_i+j_i>m_i,
\] 
the sum of  terms 
\[
\begin{split}
&\sum_{-|\alpha^1_{i, l_i}|-|\alpha^2_{i, l_i}|-|\alpha^3_{i, l_i}|-|\alpha^4_{i, l_i}|+l_i+j_i>m_i} a_{i,l_i, \alpha^j_{i, l_i}, \beta^j_{i, l_i}, \gamma_{k, l_k}, \delta^j_{i, l_i}, \sigma^j_{i, l_i} }y_{i-1}^{\alpha^1_{i, l_i}+\alpha^3_{i, l_i}}w_{i-1}^{\alpha^2_{i, l_i}+\alpha^4_{i, l_i}}\\
&\cdot \big( \xi^{\beta^1_{i, l_i}+\sigma^1_{i, l_i}} \nu^{\beta^2_{i, l_i}+\sigma^2_{i, l_i}} \big(\partial^{\delta^1_{i, l_i}}_z\partial_\eta^{\delta^2_{i, l_i}+\gamma_{i, l_i}}\tilde{q}_{i, l_i}((0, 0), (\xi, \nu), \tau)\big)^\vee (y_i-y_{i-1}, w_i-w_{i-1}, 1)
\end{split}
\]
must therefore vanish since the Getzler order of $q_i$ is $m_i$.

In (\ref{eq:sigma-I}), we look at the power of $t$, that is
$$
t^{\frac{\sum_{i=1}^{k}( -l_i+|\beta^1_{i, l_i}|+|\beta^2_{i ,l_i}|-|\gamma_{i, l_i}|+|\alpha^1_{i, l_i}|+|\alpha^2_{i,l_i}|+|\alpha^3_{i, l_i}|+|\alpha^4_{i, l_i}|)-(a+2k)}{2}}. 
$$
By the above observation using the Getzler order, we can assume to only consider those terms with 
$$
-|\alpha^1_{i, l_i}|-|\alpha^2_{i, l_i}|-|\alpha^3_{i, l_i}|-|\alpha^4_{i, l_i}|+l_i+j_i\leq m_i.
$$
Using the property that $2(|\beta^1_{i, l_i}|+|\beta^2_{i, l_i}|)\leq |\gamma_{i, l_i}|$ from (\ref{eq:qi}), we have
$$
|\beta^1_{i, l_i}|+|\beta^2_{i, l_i}|-|\gamma_{i,l_i}|\leq -\frac{1}{2}(|\beta^1_{i, l_i}|+|\beta^2_{i, l_i}|)\leq 0. 
$$
Hence, we only need to consider terms with the following bound on the sum,
$$
\frac{\sum_{i=1}^{k}( -l_i+|\beta^1_{i, l_i}|+|\beta^2_{i ,l_i}|-|\gamma_{i, l_i}|+|\alpha^1_{i, l_i}|+|\alpha^2_{i,l_i}|+|\alpha^3_{i, l_i}|+|\alpha^4_{i, l_i}|)-(a+2k)}{2}\geq \frac{\sum_{i=1}^k (-m_i+j_i) -(a+2k)}{2}.
$$
\end{proof}
\end{lemma}

\begin{definition}\label{eq:tr-geo}
Let $A$ be a smoothing operator with kernel $A(x,y)$. The geometric trace of $A$ is defined by 
\begin{align}
\Tr^{{\rm geo}}\left(A \right): = \int_M A(y, y)\; dy,	
\end{align}
if the integral on the right side is absolutely convergent. If $A$ acts on  the sections of a vector bundle $E$, then 
\begin{align}
\Tr^{{\rm geo}}\left(A \right): = \int_M {\rm tr}_{E_y} A(y, y)\; dy,	
\end{align}

\end{definition}

\begin{lemma}\label{lem:cutoff}
If along $M^\gamma$ the Getzler order of $Q_i$ satisfies $\sum_{i=1}^k m_i < -2k$, then 
\[
\lim_{t \to 0^+}\Tr^{{\rm geo}}\left(\gamma^{-1} \Psi Q_1 \Psi Q_2 \Psi \ldots \Psi Q_k^\gamma\right) = 0.
\]
\begin{proof}
Recall that 
\begin{equation*}
 \lim_{t \to 0^+}\Tr^{{\rm geo}}\left(\gamma^{-1} \Psi Q_1 \Psi Q_2 \Psi \ldots \Psi Q_k^\gamma\right)  = 
\lim_{t \to 0^+}\int_M \sigma\left[\gamma^{-1} \widetilde{J}(x,\gamma,  t)\right]^{(n)}\; dx =  \lim_{t \to 0^+} \int_{M^\gamma} \sigma \left[\gamma^{-1} J(y, \gamma, t)\right]^{(n)}\; dy 
\end{equation*}
The first equality follows from the definition of the function $\widetilde{J}(x,\gamma,  t)$. The second equality follows from $J(y, \gamma, t)$ by the definition of $J(y, \gamma, t)$ and the property of $\Psi$, a function supported in a neighborhood of $y\in M^\gamma$. 
In order to explain the last equality we make use of the last formula in page 337 in \cite{Ponge-Wang-2}
where our $\gamma $ corresponds to $\phi^{S}$ in that formula and $\mathfrak{a}(t)$ 
corresponds to our $J(y,\gamma,t)$. Using also Lemma 4.7 in \cite{Ponge-Wang-2} , and in particular formula (4.13) there, we obtain that
$$
\sigma \left[\gamma^{-1} J(y, \gamma, t)\right]^{(n)}= \sigma [\gamma^{-1} ]^{(0,n-a)} \wedge \sigma [J(y, \gamma, t]^{(a,0)}
$$
modulo terms involving $\sigma J(y, \gamma, t)^{(a,l)}$ with $l\geq 1$. 

\noindent
Recall that \cite[Lemma 4.7]{Ponge-Wang-2} operator $(\gamma^{-1})^{\mathcal{S}}\otimes (\gamma^{-1} )^{E}$ is a pointwise operator on the fiber $\mathcal{S}|_y\otimes E|_y$ of the form 
\[
(\gamma^{-1} )^{\mathcal{S}}=\coprod_{\frac{a}{2}<j\leq \frac{n}{2}}\left(\cos\left(\frac{\theta_j}{2}\right)+\sin\left(\frac{\theta_j}{2}\right)c(v^{2j-1})c(v^{2j})\right),
\]
where $v_1, \cdots, v_{a}$ is an orthonormal basis of $\mathbb{R}^a$, and $v_{a+1}, \cdots, v_{n}$ is an orthonormal basis of $\mathbb{R}^{n-a}$, and $e^{i\theta_{a+1}}, \cdots, e^{i\theta_n}$ are the eigenvalues of the action of $\gamma^{-1} $ on $\mathcal{S}|_y$.  Hence, the expression $\sigma \left[\gamma^{-1}  J(y, \gamma, t)\right]^{(n)}$ has the following form
\[
\sigma \left[\gamma^{-1} J(y, \gamma, t)\right]^{(n)}=\sigma [\gamma^{-1} ]^{(0,n-a)} \wedge \sigma [J(y, \gamma, t]^{(a,0)}+\sum_{0<l\leq n-a} \sigma[\gamma^{-1} ]^{(0,n-l)}\wedge \sigma[J(y, \gamma, t)]^{(a, l)}.
\]

\noindent
Recall now Lemma \ref{lem:symbolest} stating that
$$
\sigma\left[J(y, \gamma, t)\right]^{(j)} \sim O(t^{\frac{-\sum_{i}m_i+j-a-2k}{2}}).
$$	
This means that \[
\sigma\left[J(y, \gamma, t)\right]^{(a, l)}\sim O(t^{\frac{-\sum_{i}m_i+l -2k}{2}})\;\;\;\text{where}\;\;\; l\geq 0.
\]
 Hence, if the Getzler order $\sum_{i=1}^k m_i < -2k$, then 
\[
\sigma\left[J(y, \gamma, t)\right]^{(a, l)} = \mathcal{O}(\sqrt{t}) \;\;\;\text{for all }\;\;\; l\geq 0.
\]
This shows that $\sigma \left[\gamma^{-1} J(y, \gamma, t)\right]^{(n)}\to 0$ as $t\to 0$, which completes proof of the lemma. 
\end{proof}
\end{lemma}

The following proposition will play a key role in moving terms under the sign 
of $\Tr^{{\rm geo}}$. First we give two Definitions.

\begin{definition}\label{def: Volterra-related}
We shall say that a family of smoothing kernels $\{Q(t)\}_{t\in\mathbb{R}^+}$ on $M\times M$ defined by a smoothing kernel $K_Q(x,y,t)$ on $M\times M\times \mathbb{R}^+$ is  {\bf Volterra-related} if  there exists a Volterra operator  with Volterra kernel
equal to $K_Q$ when restricted to $M\times M\times \mathbb{R}^+$. The Volterra order (Getzler order at $x_0$) of $\{Q(t)\}_{t\in \mathbb{R}^+}$
is by definition the Volterra order (Getzler order at $x_0$) of the corresponding Volterra operator.\end{definition}

\noindent
For example, the heat-kernel is Volterra-related; its Volterra order is -2 and its Getzler order 
at any point of $M$ is also -2. More examples will be given in the next subsection.

\begin{definition}\label{def: exponential-control-no-appendix}
We shall say that a family of smoothing kernels $\{Q(t)\}_{t\in\mathbb{R}^+}$ on $M\times M$ defined by a smoothing kernel $K_Q(x,y,t)$ on $M\times M\times \mathbb{R}^+$ is of {\bf exponential control} if there exist constants $\alpha, \beta, \eta > 0$ such that for sufficiently small $t$,
\begin{align}\label{kernel decay bis-tris}
\Big| K_Q (x, y, t)\Big|\leq  \alpha \cdot t^{-\beta} \cdot e^{-\eta \cdot \frac{d(x, y)}{t}}.	
\end{align}
\end{definition}

\noindent
In what follows, we sometimes will abuse the notation of $\{Q(t)\}_{t\in\mathbb{R}^+}$ and the employ the simplified symbol $Q$. 

\begin{lemma}[Milnor-\v{S}varc, see for example \cite{Roe-lectures}]
Let $\Gamma$ be a group acting by isometries on a Riemannian manifold $X$ such that the action is properly discontinuous and cocompact. Then the group $\Gamma$ is finitely generated and for every finite generating set $S$ of $G$ and every point $x \in X$ the orbit map
\[
\left(\Gamma, l\right) \rightarrow X, \quad \gamma \mapsto \gamma x
\]
is a quasi-isometry, where $l$ is the word metric on $\Gamma$
corresponding to $S$. 
\end{lemma}

\begin{lemma}
\label{lem dis}
Let $F\subset X$ be a compact set. There exist positive constants $\tau, \kappa$ such that 
\[
\frac{1}{\tau} \cdot l(\gamma) + \kappa >  d(x, \gamma y) \geq \tau \cdot l(\gamma) - \kappa, \quad \forall x, y \in F\,.
\]
\begin{proof}
By the Milnor-\v{S}varc lemma, we can find constants $\tau_x, \kappa_x>0$ such that 
\[
\frac{1}{\tau_x} \cdot l(\gamma) + \kappa_x > d(x, \gamma y) \geq \tau_x \cdot l(\gamma) - \kappa_x. 
\]
Since $F$ is compact, we can take $\tau = \min_{x\in F} \{\tau_x\}$ and $\kappa = \max_{x\in F} \{\kappa_x\}$. 
\end{proof}
\end{lemma}

\begin{lemma}\label{est:exp}
Let $F$ be a compact subset of $M$. We can find a constant $C$ large enough so that  
\[
\int_{F \times M^k} e^{-C\left(d(x, y_1) + \ldots d(x, y_k)\right)} \; dx dy_1\ldots dy_k
\]
is  convergent. 
\begin{proof}
Because the $\Gamma$ action on $M$ is proper and cocompact, there is an open subset $\mathcal{F}$ of $M$ such that the closure $\overline{\mathcal{F}}$ is a compact subset of $M$  and
\[
\bigcup_{\gamma \in \Gamma} \gamma \mathcal{F}=M. 
\] 
By Lemma \ref{lem dis}
\begin{multline*}
\int_{F \times \mathcal{F}^k } e^{-C\left(d(x, y_1) + \ldots d(x, y_k)\right)} \; dx dy_1\ldots dy_k
\leq  \sum_{\gamma_i \in \Gamma}\int_{F \times \mathcal{F}^k } e^{-C \cdot \left(d(x, \gamma_1 y_1) + \ldots d(x, \gamma_k y_k)\right)} \; dx dy_1\ldots dy_k	\\
\leq \sum_{\gamma_i \in \Gamma}\int_{F \times \mathcal{F}^k }e^{-C\left(\sum_i \tau_i \cdot l(\gamma_i) - \kappa_i\right)} \; dx dy_1\ldots dy_k\\
\leq C_N+\vol(F \times \mathcal{F}^k) \sum_{\gamma_i \in \Gamma, l(\gamma_i) > N}e^{-C\left(\sum_i \tau_i \cdot l(\gamma_i) - \kappa_i\right)} 
\end{multline*}
Because $\Gamma$ is finitely generated, there are constants $C_\Gamma$, $K_\Gamma$ satisfying
\[
\#\left\{\gamma \in \Gamma, l(\gamma) \leq k\right\} \leq C_\Gamma \cdot e^{K_\Gamma k}. 
\]
Thus, 
\[
\sum_{\gamma_i \in \Gamma, l(\gamma_i) > N}e^{-C\left(\sum_i \tau_i \cdot l(\gamma_i) - \kappa_i\right)} 
 \leq  \sum_{k=N}^\infty  C_\Gamma \cdot e^{\sum_i \left((K_\Gamma - C\tau_i) k_i + C\kappa_i \right)}
 \]
We first choose $C$ large enough so that 
\[
K_\Gamma - C\tau_i <0. 
\] 
If we further choose $N$ so that 
\[
N \cdot \sum_i\tau_i  > C\sum_i \kappa_i, 
\]
then the above summation is finite. 	
\end{proof}
\end{lemma}

\begin{lemma}\label{lem:exponential control}
If $\{Q_1(t)\}_{t\in \mathbb{R}^+}$ and $\{Q_2(t)\}_{t\in \mathbb{R}^+}$ have smoothing kernels $K_{Q_i}(x,y,t)$ that are of exponential control, the composition of $\{Q_1(t)\circ Q_2(t)\}_{t\in \mathbb{R}^+}$ is well defined and has a smoothing kernel of exponential control.
\end{lemma}
\begin{proof}
Since the smoothing kernels $K_{Q_i}(x,y,t)$ ($i=1,2$) are of exponential control, there are $\alpha_i, \beta_i, \eta_i > 0$, $i=1,2$, 
\[
\Big| K_{Q_i} (x, y, t)\Big|\leq  \alpha_i \cdot t^{-\beta_i} \cdot e^{-\eta_i \cdot \frac{d(x, y)}{t}}.
\]
Take $\eta=\min(\eta_1, \eta_2)$. By Lemma \ref{est:exp}, for sufficiently small $t$, 
\[
\int_M e^{-\eta_1\frac{d(x,z)}{2t}} e^{-\eta_2\frac{d(z,y)}{2t}} dz\leq \int_M e^{-\eta\frac{d(x,z)}{2t}} e^{-\eta\frac{d(z,y)}{2t}} dz
\]
is integrable. We can derive the exponential control property of the composition by the following observation
\[
|\int_M K_{Q_1}(x,z,t)K_{Q_2}(z, y,t)d y|\leq \int_M |K_{Q_1}(x,z,t)K_{Q_2}(z, y,t)|d y\leq C t^{-\beta_1-\beta_2} e^{-\eta\frac{d(x,y)}{2t}} \int_M e^{-\eta\frac{d(x,z)}{2t}} e^{-\eta\frac{d(z,y)}{2t}} dz.
\]
\end{proof}

\begin{lemma}\label{prop:moving}
Let $f_0 \in C^\infty_c(M)$. We assume that we have 
Volterra-related  smoothing kernels $Q_i$ satisfying the exponential control property
and such that
\[
\lim_{t \to 0^+}\Tr^{{\rm geo}}\left(\gamma^{-1} f_0Q_1 \Psi Q_2  \ldots  \Psi Q_k^\gamma\right) = 0,
\]
for a compactly supported smooth function $\Psi$, which is equal to $1$ on the support of $f_0$.  Then
\[
\lim_{t \to 0^+}\Tr^{{\rm geo}}\left(\gamma^{-1} f_0Q_1  Q_2  \ldots  Q_k^\gamma\right) = 0.
\]
\end{lemma}
\begin{proof}
Using partition of unit, we can write 
\[
f_0(y) = \sum_i \rho_i(y) f_0(y). 
\]
We can choose the support of $\rho_i$ to be arbitrarily small. Thus, we can assume that $f_0$ is supported in a small neighborhood $U_y$ for some $y \in M$. Shrinking the support of $f_0$ if necessary, we can find a small $\epsilon > 0$ and a function $\Psi \in C^\infty_c(M)$ such that 
\begin{equation}\label{condition:supp}
\supp(\Psi) \subseteq U_y, \quad \Psi \equiv 1 \text{ on } B\left(\supp(f_0),  \epsilon \right).	
\end{equation}
By assumption, we have
\begin{equation}\label{equ:geo0}
\lim_{t \to 0^+}\Tr^{{\rm geo}}\left(\gamma^{-1} f_0Q_1 \Psi Q_2  \Psi\ldots \Psi Q_k^\gamma\right) = 0. 
\end{equation}
Moreover, we know that ,
\[
\left\|K_{Q_i}(x, y,t)\right\|\leq \alpha_i  \cdot t^{-\beta_i} \cdot  e^{-\eta_i \cdot \frac{d(x, y)}{t}} 
\]
for some positive constants $\alpha_i , \beta_i, \eta_i$. By the condition in \eqref{condition:supp}, 
\[
f_0(x) (1-\Psi )(y) = 0
\]
whenever $\operatorname{dist}(x, y) < \epsilon$. Thus 
\[
\left\|f_0(x) \cdot K_{Q_1}(x, y,t) \cdot  (1-\Psi )(y)\right\|\leq \alpha_1 \cdot t^{-\beta_1} \cdot  e^{-\eta_1 \cdot \frac{\epsilon + d(x, y)}{2t}} 
\]
It follows that 
\begin{multline}\label{equ:epsilon}
\left|\Tr^{{\rm geo}}\left(\gamma^{-1} f_0Q_1 (1-\Psi )Q_2 \Psi  \ldots \Psi Q_k^\gamma\right) \right|\\
\leq  \int_{U_y \times M \times U_y^{\times (k-2)}} \left\| f_0(x_0) \cdot K_{Q_1}(x_0, y,t) \cdot  (1-\Psi )(y)  K_{Q_2}(y, x_2,t) \Psi(x_2) \ldots \Psi(x_l) K_{Q_l}(x_l, \gamma x_0,t)\right\| dx_0 dy dx_2 \ldots dx_k\\
\leq C \int_{U_y \times M } \left\| f_0(x_0) \cdot K_{Q_1}(x_0, y,t) (1-\Psi )(y) \right\| dx_0 dy \\
\leq C\alpha t^{- \beta}  \cdot e^{-\eta \cdot \frac{\epsilon}{2t}}\int_{U_y \times M }e^{-\eta \cdot \frac{d(x_0, y)}{t}} dx_0 dy.
\end{multline} 
By Lemma \ref{est:exp}, we know that 
\[
\int_{U_y \times M }e^{-\eta \cdot \frac{d(x_0, y)}{t}} dx_0 dy 
\]
is uniformly bounded in the $t$-variable, for $t\leq 1$. 
Because of the presence of the factor $e^{-\eta \frac{\epsilon}{2t}}$ in (\ref{equ:epsilon}), we have shown 
\[
\lim_{t \to 0^+}\Tr^{{\rm geo}}\left(\gamma^{-1} f_0Q_1 (1-\Psi )Q_2 \Psi  \ldots \Psi Q_k^\gamma\right) = 0.
\]
Together with (\ref{equ:geo0}) we obtain 
\[
\lim_{t \to 0^+}\Tr^{{\rm geo}}\left(\gamma^{-1} f_0Q_1Q_2 \Psi  \ldots \Psi Q_k^\gamma\right) = 0,
\]
Note now that it follows from Lemma \ref{lem:exponential control} that the composition  $Q_1 Q_2$ satisfies the exponential-control condition as well, thus we can repeat the argument and conclude by induction that 
\[
\lim_{t \to 0^+}\Tr^{{\rm geo}}\left(\gamma^{-1} f_0Q_1  Q_2  \ldots  Q_l^\gamma\right) = 0.
\]
\end{proof}

\begin{theorem}\label{theo:moving}
Suppose that $Q_1,\cdots, Q_k$ are Volterra-related smoothing kernels
on $M\times M\times \mathbb{R}^+$ that are of exponential control and that are of order $m'_1, \cdots, m'_k$ and Getzler order along the fixed point set $M^\gamma$ equal to $m_1, \cdots, m_k$.  If $m_1+\cdots +m_k<-2k$,  then for any $f_0\in C^\infty_c(M)$, 
\[
\lim_{t \to 0^+}\Tr^{{\rm geo}}\left(\gamma^{-1} f_0Q_1  Q_2  \ldots  Q_k^\gamma\right) = 0.
\]
\end{theorem}
\begin{proof}
This result follows from Lemmas \ref{lem:Jy}, \ref{lem:symbolest}, \ref{lem:cutoff}, \ref{prop:moving}.
\end{proof}

\subsection{Some useful Volterra operators}\label{subsection:useful}
In this subsection we prove that the kernels appearing in the index pairing via the Connes-Moscovici  
projector for $tD$ are indeed Volterra-related and we compute their Getzler order.
Our results are summarized in the table at the end of the subsection.

\medskip
\noindent
 Let us consider the Volterra operator
\begin{equation}\label{the-volterra}
Q = \left(\frac{1}{2}D^2+ \frac{\partial}{\partial t}\right)^{-1}  \circ   \left(\frac{3}{2}D^2+ \frac{\partial}{\partial t}\right)^{-1}.
\end{equation}
Notice that this operator  has Getzler order -4.
We denote by $K_{Q}(x, y, t)$ the distributional kernel of $Q$. We also consider
the operator
	\[
	t\left( \frac{I- e^{-t D^2}}{tD^2} \right) e^{-\frac{t}{2}D^2}
	\]
for $t>0$ and we denote by $\kappa_t(x, y)$ its kernel.

\begin{lemma}
\label{lemma-useful-identity}
The identity
\begin{equation}
\label{useful-identity}
t \,\left( \frac{I- e^{-t D^2}}{tD^2} \right) e^{-\frac{t}{2}D^2} =  \int_{\frac{1}{2}}^{\frac{3}{2}} t\,e^{-s_1tD^2}\; ds_1,
\end{equation}
holds true and its kernel $\kappa_t(x,y)$ is of rapid exponential decay.
\end{lemma}
\begin{proof}
Let
\[
h_t(x):= \left(\frac{1-e^{-tx}}{tx}\right)e^{-\frac{t}{2}x}.
\]
We then have that the operator on the left hand side of \eqref{useful-identity} is equal to $h_t(D^2)$, the bounded operator obtained by applying the Borel functional calculus to the selfadjoint  operator $D^2$.
That is, 
\[
\left<\left(\frac{I-e^{-tD^2}}{tD^2}\right)e^{-\frac{t}{2}D^2}\psi_1,\psi_2\right>=\int_0^\infty h_t(\lambda)dE_{\psi_1,\psi_2}(\lambda),\qquad \psi_1,\psi_2\in L^2(M), 
\]
where $E$ denotes the resolution of the identity on $\mathbb{R}$ deteremined by $D^2$. Because of the identity
\begin{equation}
\label{eq-efr}
\left(\frac{1-e^{-tx}}{tx}\right)e^{-\frac{t}{2}x}=\int_{\frac{1}{2}}^{\frac{3}{2}}e^{-stx}ds, 
\end{equation}
we therefore have 
\[
\left<\left(\frac{I-e^{-tD^2}}{tD^2}\right)e^{-\frac{t}{2}D^2}\psi_1,\psi_2\right>=\int_0^\infty \int_{\frac{1}{2}}^{\frac{3}{2}}e^{-st\lambda}ds dE_{\psi_1,\psi_2}(\lambda).
\]
Write $\mu=ds\times dE_{\psi_1,\psi_2}$ for the product measure on $\mathbb{R}^+\times[\frac{1}{2},\frac{3}{2}]$ and notice that 
\[
\int_{\mathbb{R}^+\times[\frac{1}{2},\frac{3}{2}]}e^{-st\lambda}d|\mu(x,y)|\leq \int_{\mathbb{R}^+\times[\frac{1}{2},\frac{3}{2}]}e^{-\frac{t\lambda}{2}}d|\mu(x,y)|
= \int_0^\infty e^{-\frac{t\lambda}{2}}d|E_{\psi_1,\psi_2}|=\left|\left<e^{-\frac{t}{2}D^2}\psi_1,\psi_2\right>\right|<\infty.
\]
We can therefore apply Fubini's theorem and change the order of integration:
 \[
\left<\left(\frac{I-e^{-tD^2}}{tD^2}\right)e^{-\frac{t}{2}D^2}\psi_1,\psi_2\right>=\int_{\frac{1}{2}}^{\frac{3}{2}}\int_0^\infty e^{-st\lambda} dE_{\psi_1,\psi_2}(\lambda)ds
=\int_{\frac{1}{2}}^{\frac{3}{2}}\left<e^{-stD^2}\psi_1,\psi_2\right>ds.
\]
Denoting by $K_t(x,y)$ the usual heat kernel, we see from this equality that 
\begin{align*}
\left(\frac{I-e^{-tD^2}}{tD^2}\right)e^{-\frac{t}{2}D^2}\psi(x)&=\int_{\frac{1}{2}}^{\frac{3}{2}}(e^{-stD^2}\psi)(x) ds\\
&=\int_{\frac{1}{2}}^{\frac{3}{2}}\int_MK_{st}(x,y)\psi(y)dyds\\
&=\int_M\int_{\frac{1}{2}}^{\frac{3}{2}}K_{st}(x,y)\psi(y)dsdy,
\end{align*}
where we used Fubini again, this time using  the fact that $K_t(x,y)$ is of exponential rapid decay in $y$ and $[1/2,3/2]$ is compact.
This shows that the kernel of the operator in \eqref{useful-identity} is given by
\[
\kappa_t(x,y)=\int_{\frac{1}{2}}^{\frac{3}{2}}K_{st}(x,y)ds.
\]
From this we deduce that
\begin{align*}
\sup_{x,y}\left|e^{qd(x,y)}K'_t(x,y)\right|&=\sup_{x,y}\left|e^{qd(x,y)}\int_{\frac{1}{2}}^{\frac{3}{2}}K_{st}(x,y)ds\right|\\
&\leq \sup_{x,y}\int_{\frac{1}{2}}^{\frac{3}{2}}\left|e^{qd(x,y)}K_{st}(x,y)\right|ds<\infty.
\end{align*}
This proves that the kernel $\kappa_t(x,y)$ is of rapid exponential decay.
\end{proof}

\noindent
Let us go back to the Volterra operator $Q$ in \eqref{the-volterra} and let $K_Q$ be its Volterra
kernel.

\begin{lemma}\label{kernel lem 0.5}
We have that for $t>0$
\[
t  \cdot \kappa_t(x, y) = K_{Q}(x, y, t). 
\]
\begin{proof}
 For any $u \in C^\infty_+(\mathbb{R}, L^2(M, S))$, 
\begin{align*}
  &\left(\frac{1}{2}D^2+ \frac{\partial}{\partial t}\right)^{-1}  \cdot   \left(\frac{3}{2}D^2+ \frac{\partial}{\partial t}\right)^{-1}  u(x, t)\\
  =&\int_0^\infty \int_0^\infty  e^{-t_0 \frac{1}{2}D^2} \cdot e^{-t_1 \frac{3}{2}D^2} u(t - t_0-t_1)dt_0 dt_1 \\	
  &\text{let } \sigma = t_0+t_1, s_0 = \frac{t_0}{\sigma}, s_1 = \frac{t_1}{\sigma} \\
  =&\int_0^\infty \int_0^1 \sigma \cdot  e^{-(1-s_1) \frac{\sigma}{2}D^2} \cdot e^{- s_1 \frac{3\sigma}{2}D^2}  u(t - \sigma) \; ds_1d\sigma\\
    =&\int_0^\infty \int_{\frac{1}{2}}^{\frac{3}{2}} \sigma \cdot  e^{-s_1\sigma D^2}  u(t - \sigma) \; ds_1d\sigma
\end{align*}
The lemma follows from Lemma \ref{lemma-useful-identity}.
\end{proof}
\end{lemma}

\begin{lemma}\label{lem:getzler}
At any $x_0\in M$, the Getzler order of $[D^2, f]$ is 1 and that of $[D^2, c(df)]$ is 2. 
\begin{proof}
In terms of synchronous normal coordinates, we write 
\[
D = \sum c(e_i) \nabla_{e_i}.
\]
It follows that 
\begin{eqnarray*}
[D^2, f] &=& Dc(df) + c(df)D= 	 \sum c(e_i) \nabla_{e_i} c(df) + c(df)  \sum c(e_i) \nabla_{e_i}\\
&=& 	 \sum c(e_i) c(\nabla_{e_i}df) +   \sum \left(c(df)c(e_i) + c(e_i) c(df)\right) \nabla_{e_i}\\ 
&=&	\sum c(e_i)c(e_j)e_ie_j(f)+\sum e_j(f)\left(c(e_j)c(e_i)+c(c_i)c(e_j)\right)\nabla_{e_j}\\
&=& 	-\operatorname{Hess}(f) -2\sum \left(\partial_i f\right) \cdot \nabla_{e_i}.
\end{eqnarray*}
Here $\operatorname{Hess}(f)=\sum_i e_ie_i(f)$ is a scalar multiplication with Getzler order 0 and $\left(\partial_i f\right) \cdot \nabla_{e_i}$ has order one.\\
For the term $[D^2, c(df)]$, we recall the Weitzenb\"ock identity, 
\[
D^2 =-\sum \nabla_{e_{i}} \nabla_{e_i} + \frac{1}{4}\kappa + F^S. 
\]	
The later two terms commutes with $c(df)$. Thus, 
\begin{multline*}
[D^2, c(df)] = -\sum [\nabla_{e_{i}} \nabla_{e_i}, c(df)]= -\sum \nabla_{e_{i}} \nabla_{e_i}c(df) +\sum c(df)\nabla_{e_{i}} \nabla_{e_i} \\
=-\sum \nabla_{e_{i}} c(df)\nabla_{e_i} - \sum \nabla_{e_{i}} c(\nabla_{e_i}df) +\sum c(df)\nabla_{e_{i}} \nabla_{e_i}\\
=-\sum c(\nabla_{e_i}df)\nabla_{e_{i}}  -\sum \nabla_{e_{i}} c(\nabla_{e_i}df),
\end{multline*}
which shows it has order two. 
\end{proof}

\end{lemma}

\begin{lemma}\label{lemma:table}
At every $x_0\in M$, the following  table holds true: 
\[
\begin{tabu}{|c|c|c|}
\hline
{\rm Volterra-related ~ operator} &{\rm Volterra ~ operator}& {\rm Getzler ~ order}\\
\hline
e^{-tD^2}f& (D^2+\partial_t)^{-1}\circ f&-2\\
\hline
[e^{-tD^2}, f]&(D^2+\partial_t)^{-1}\circ [D^2, f] \circ (D^2+\partial_t)^{-1}&-3\\
\hline
  t\; \left( \frac{I- e^{-t D^2}}{tD^2} \right) e^{-\frac{t}{2}D^2} f&(\frac{1}{2}D^2+\partial_t)^{-1}\circ (\frac{3}{2} D^2+\partial_t)^{-1}\circ f&-4\\
\hline
&\left(\frac{1}{2}D^2+ \frac{\partial}{\partial t}\right)^{-1}  \circ   \left(\frac{3}{2}D^2+ \frac{\partial}{\partial t}\right)^{-1}\circ &\\
 \left[ t\; \left( \frac{I- e^{-t D^2}}{tD^2} \right) e^{-\frac{t}{2}D^2},f\right]&\circ 
\left(\left[D^2 ,  f \right]\circ \left(\frac{3}{4}D^2 + \frac{\partial}{\partial t}\right) 
+  \left(\frac{3}{4}D^2 + \frac{\partial}{\partial t}\right) \circ \left[D^2 ,  f\right] \right) \circ &-5\\
&
\circ  \left(\frac{1}{2}D^2+ \frac{\partial}{\partial t}\right)^{-1}  \circ   \left(\frac{3}{2}D^2+ \frac{\partial}{\partial t}\right)^{-1}	&\\
\hline
e^{-tD^2} c(df)& (D^2+\partial_t)^{-1}\circ  c(df)&-1\\
\hline
[e^{-tD^2},  c(df)]&(D^2+\partial_t)^{-1}\circ [D^2,  c(df)] \circ D^2+\partial_t)^{-1}&-2\\
\hline
t\;\frac{1-e^{-tD^2}}{tD^2} e^{-\frac{t}{2} D^2} c(df)&(\frac{1}{2}D^2+\partial_t)^{-1}\circ (\frac{3}{2} D^2+\partial_t)^{-1}\circ  c(df)&-3\\
\hline
&\left(\frac{1}{2}D^2+ \frac{\partial}{\partial t}\right)^{-1}  \circ   \left(\frac{3}{2}D^2+ \frac{\partial}{\partial t}\right)^{-1}\circ &\\
 \left[ t\; \left( \frac{I- e^{-t D^2}}{tD^2} \right) e^{-\frac{t}{2}D^2}, c(df)\right]&\circ 
\left(\left[D^2 ,   c(df) \right]\circ \left(\frac{3}{4}D^2 + \frac{\partial}{\partial t}\right) 
+  \left(\frac{3}{4}D^2 + \frac{\partial}{\partial t}\right) \circ \left[D^2 ,   c(df)\right] \right) \circ &-4\\
&
\circ  \left(\frac{1}{2}D^2+ \frac{\partial}{\partial t}\right)^{-1}  \circ   \left(\frac{3}{2}D^2+ \frac{\partial}{\partial t}\right)^{-1}	&\\
\hline
\end{tabu}
\]
\end{lemma}
\begin{proof}
The table follows from the fact that $(D^2+\partial_t)^{-1}$ has the Getzler order -2 and Lemma \ref{lem:getzler}. 
\end{proof}
\section{The higher Lefschetz formula}\label{section:higher lefschetz}
\subsection{Index pairing}
Now we  consider 
\[
\Ind_{c} (D):= [P_{\operatorname{fp}}(t)]-[e_1] \in K_0 (\mathcal{A}^c_\Gamma (M,E)). 
\]
We are interested in proving a formula for the pairing of the class $\Ind_ {c} (D)$ with the cyclic cocycle 
\[
\rho(\Psi(c))\in HC^{2p} (\mathcal{A}^{c}_\Gamma (M)). 
\]
Since the $K$-theory class $[P_{\operatorname{fp}}(t)]-[e_1] $ is independent of $t$, we can compute the above pairing by taking the limit $t\to 0$.
\[
\lim_{t\to 0}\langle[P_{\operatorname{fp}}(t)]-[e_1] ,\rho \circ \Psi(c) \rangle
\]
Our computation mainly adapts the method in \cite{CM} to the equivariant setting on Volterra operators.

Recall from Moscovici-Wu that the class defined by $V_{{\rm CM}} (tD)\oplus (V_{{\rm CM}} (tD))^*$ and $e_1\oplus e_1$, $t>0$, is equal to the class
defined by the $4\times 4$ matrix 
\begin{equation*}
R(tD):=\left( \begin{array}{cc} e^{-t D^2} \epsilon & 
\left( \frac{I- e^{-t D^2}}{tD^2} \right) e^{-\frac{t}{2}D^2} \sqrt{t}D \epsilon\\
e^{-\frac{t}{2}D^2} \sqrt{t}D\epsilon &  e^{-tD^2}\epsilon
\end{array} \right)
\end{equation*}
with $\epsilon$ the grading operator on the Clifford bundle $E$.

By (\ref{equ:antisym}) and  Corollary \ref{cor:rfptoR}, 
\begin{multline}\label{equtrgeo}
\left\langle \rho(\Psi(c)), \Ind_{c} (D) \right\rangle = \lim_{t\to 0}\langle\rho \circ \Psi(c), [P_{\operatorname{fp}}(t)]-[e_1]  \rangle\\
=\frac{1}{(2p +1)!}\;  \lim_{t\to 0} \rho\big(\mathcal{C}_{\mathcal{I}_{2p+1}, a}\big)( P_{\operatorname{fp}}(t)]-[e_1], ..., P_{\operatorname{fp}}(t)]-[e_1])	\\
=\frac{1}{(2p +1)!}\;  \lim_{t\to 0}\rho\big(\mathcal{C}_{\mathcal{I}_{2p+1}, a}\big)( R(tD), ..., R(tD))\\
=\frac{1}{(2p +1)!}  \; \sum_{I = (\gamma_0, \ldots, \gamma_{2p}) \in \mathcal{I}_{2p+1}}c(\gamma_0, \cdots, \gamma_{2p})\cdot\chi_\gamma(\gamma_0)\; \lim_{t\to 0}
 \Tr^{{\rm geo}}_a\left(\gamma^{-1} f_{0} R( \sqrt{t}D) f_{1} R( \sqrt{t}D)\cdots f_{2p}R( \sqrt{t}D)^\gamma\right)
\end{multline}
where $A^\gamma(x,y):=A(x,\gamma y), f_i(y_i) = \chi(\gamma_i^{-1}y_i)$, $\mathcal{I}_{2p+1}$ is a finite subset defined in Proposition \ref{prop:finitesum}, and $ \Tr^{{\rm geo}}_a$ is the anti-symmetrized trace defined by 
\begin{multline}\label{anti-expression}
\Tr^{{\rm geo}}_a\left(\gamma^{-1} f_{0} R( \sqrt{t}D) f_{1} R( \sqrt{t}D)\cdots f_{2p}R( \sqrt{t}D)^\gamma\right)	\\
:=  \sum_{\sigma \in S_{2p+1}} \text{sign}(\sigma) \Tr^{{\rm geo}}\left(\gamma^{-1} f_{\sigma(0)} R( \sqrt{t}D) f_{\sigma(1)} R( \sqrt{t}D)\cdots f_{\sigma(2p)}R( \sqrt{t}D)^\gamma\right). 	
\end{multline}
\begin{lemma}\label{lem:symmetry}
The following identity holds:
\begin{multline}
\Tr^{{\rm geo}}_a\left(\gamma^{-1} f_0 R( \sqrt{t}D) f_1 R( \sqrt{t}D)\cdots f_{2p}R( \sqrt{t}D)^\gamma\right) \\
= 	\sum_{\sigma \in S_{2p+1}} \text{sign}(\sigma) \cdot \Tr^{{\rm geo}}\left( \gamma^{-1} f_{\sigma(0)}\left[ R( \sqrt{t}D), f_{\sigma(1)}\right]\ldots \left[ R( \sqrt{t}D), f_{\sigma(2p)}\right]R( \sqrt{t}D)^\gamma\right) 
\end{multline}
\end{lemma}
\begin{proof}
Since 
\begin{multline}
 \Tr^{{\rm geo}}\left(\gamma^{-1} f_{\sigma(0)}  f_{\sigma(1)}R( \sqrt{t}D) R( \sqrt{t}D)f_{\sigma(2)}\cdots f_{\sigma(2p)}R( \sqrt{t}D)^\gamma\right)\\=  \Tr^{{\rm geo}}\left(\gamma^{-1} f_{\sigma(1)}  f_{\sigma(0)}R( \sqrt{t}D) R( \sqrt{t}D)f_{\sigma(2)}\cdots f_{\sigma(2p)}R( \sqrt{t}D)^\gamma\right), 
\end{multline}
it follows from  the anti-symmetrization that 
\begin{align}
\sum_{\sigma \in S_{2p+1}} \text{sign}(\sigma) \cdot \Tr^{{\rm geo}}\left(\gamma^{-1} f_{\sigma(0)}  f_{\sigma(1)}R( \sqrt{t}D) R( \sqrt{t}D)f_{\sigma(2)}\cdots f_{\sigma(2p)}R( \sqrt{t}D)^\gamma\right)	= 0. 
\end{align}
Thus, 
\begin{multline}
\Tr^{{\rm geo}}_a\left(\gamma^{-1} f_0 R( \sqrt{t}D) f_1 R( \sqrt{t}D)\cdots f_{2p}R( \sqrt{t}D)^\gamma\right)\\
=\sum_{\sigma \in S_{2p+1}} \text{sign}(\sigma) \cdot \Tr^{{\rm geo}}\left(\gamma^{-1} f_{\sigma(0)} R( \sqrt{t}D) f_{\sigma(1)} R( \sqrt{t}D)\cdots f_{\sigma(2p)}R( \sqrt{t}D)^\gamma\right) \\
= 	\sum_{\sigma \in S_{2p+1}} \text{sign}(\sigma) \cdot \Tr^{{\rm geo}}\left(\gamma^{-1}  f_{\sigma(0)}\left[ R( \sqrt{t}D), f_{\sigma(1)}\right]R( \sqrt{t}D) \ldots R( \sqrt{t}D) f_{\sigma(2p)}R( \sqrt{t}D)^\gamma\right). 
\end{multline}
By repeating the procedure, we can prove the lemma.  
\end{proof}

Our goal is to show that we can take
the limit of this expression at $t\downarrow 0$ and that this limit is given by a higher Atiyah-Segal-Singer formula.
Following \cite{moscovici-wu} we introduce the following 
Volterra-related operators:
$$A^+ := e^{-\frac{t}{2} D^2} D\, \text{ with  Getzler order 0},\quad A^- := 
t\left( \frac{I- e^{-t D^2}}{tD^2} \right) e^{-\frac{t}{2}D^2} D  \text{ with  Getzler order -2}$$
$$ B^+_i :=  e^{-\frac{t}{2} D^2} [D,f_i]\, \text{ with  Getzler order -1},\quad B^-_i:= t\left( \frac{I- e^{-t D^2}}{tD^2} \right) e^{-\frac{t}{2}D^2}[D,f_i] \text{ with  Getzler order -3}$$
$$C^+_i =  [e^{-\frac{t}{2} D^2},f_i] D\,,  \text{ with  Getzler order -1}\quad C^-_i:= \left[  t\left( \frac{I- e^{-t D^2}}{tD^2} \right) e^{-\frac{t}{2}D^2},f_i\right] D  \text{ with Getzler order -3}$$
and then 
$$A= \left( \begin{array}{cc} 0&\frac{A^-}{\sqrt{t}}\\ \sqrt{t}A^+&0
\end{array} \right)\,, \quad B_i= \left( \begin{array}{cc} 0&\frac{B^-_i}{\sqrt{t}}\\ \sqrt{t}B^+_i&0
\end{array} \right)
\,,\quad C_i= \left( \begin{array}{cc} 0&\frac{C^-_i}{\sqrt{t}}\\\sqrt{t}C^+_i&0
\end{array} \right).
$$
We notice once again that $A^\pm$, $B_i^\pm$, and $C_i^\pm$ are all 
Volterra-related operators. In what follows we will use $\operatorname{ord}_G$ to denote the Getzler order of a Volterra operator. We point out that the Getzler order is, in general,  a number depending on $x_0\in M^\gamma$  but for  $A^\pm$, $B_i^\pm$, and $C_i^\pm$ the Getzler order is  actually constant, independent of $x_0$ in $M^\gamma$.

We then have 
$$\Tr^{{\rm geo}}(\gamma^{-1} f_0  [R( \sqrt{t}D), f_1] \cdots [R( \sqrt{t}D),f_{2p}]R( \sqrt{t}D)^\gamma)= \Tr^{{\rm geo}} (\Pi( \sqrt{t}D))$$
with 
$$\Pi ( \sqrt{t}D):= \gamma^{-1}  f_0 \left( \Pi_{i=1}^{2p} \left(\left[   e^{-t D^2} ,f_i\right]\cdot I_2\epsilon +B_i \epsilon+C_i \epsilon\right) \right) ( e^{-t D^2} \cdot I_2\epsilon + A\epsilon)^\gamma,$$ and 
$I_2=\left(\begin{array}{cc}1&0\\ 0&1\end{array}\right)$.

\begin{proposition} \label{prop:trPi}Let $\Tr_s^{\rm geo}$ denote the supertrace on the space of smoothing kernels with finite propagation. 
The terms that contribute to 
\[
\lim_{t \to 0^+}\Tr^{\mathrm{geo}}\!\big(\Pi(\sqrt{t}D)\big)
\]
are precisely the following:
\begin{enumerate}
    \item $\displaystyle 
    \Tr^{\mathrm{geo}}_s\left(\gamma^{-1} f_0\, B_1\epsilon B_2\epsilon \cdots B_{2q-1} \epsilon B_{2q}\epsilon \, (e^{-t D^2})^\gamma \right)=(-1)^q\Tr^{\mathrm{geo}}_s\left(\gamma^{-1} f_0\, B_1B_2\cdots B_{2q-1}B_{2q}\, (e^{-t D^2})^\gamma\right);$

    \item \[
    \begin{split}
    &\Tr^{\mathrm{geo}}_s\left(\gamma^{-1} f_0\, B_1\epsilon B_2\epsilon \cdots  C_j\epsilon \cdots B_{2q-1}\epsilon B_{2q}\epsilon \, (e^{-t D^2})^\gamma \right)\\
    =&(-1)^q \Tr^{\mathrm{geo}}_s\left(\gamma^{-1} f_0\, B_1 B_2 \cdots  C_j\cdots B_{2q-1} B_{2q} \, (e^{-t D^2})^\gamma\right);
    \end{split}
    \]

    \item 
    \[
    \begin{split}
    &\Tr^{\mathrm{geo}}_s\left(\gamma^{-1} f_0\, B_1\epsilon B_2\epsilon \cdots  [e^{-t D^2},f_j]I_2\epsilon \cdots B_{2q-1}\epsilon B_{2q}\epsilon\, A^\gamma \right)\\
    =&(-1)^{q+j+1}\Tr^{\mathrm{geo}}_s\left(\gamma^{-1} f_0\, B_1B_2 \cdots  [e^{-t D^2},f_j]I_2 \cdots B_{2q-1} B_{2q}\, A^\gamma\right).
    \end{split}
    \]
\end{enumerate}

\begin{proof}
In the expansion of $\Tr^{\mathrm{geo}}\!\big(\Pi(\sqrt{t}D)\big)$, 
the \emph{off-diagonal terms} are given by $B_i^\pm$, $C_i^\pm$, and $A^\pm$, 
while the \emph{diagonal terms} are
\[
[e^{-t D^2},f_i] \quad \text{of Getzler order } -3,
\qquad 
e^{-t D^2} \quad \text{of Getzler order } -2.
\]
As there are equal number of $+$ and $-$ terms in the product expansion, we observe that the factor of $t$ in $A$, $B$, $C$ cancel. We will use Theorem \ref{theo:moving} to show that there are only a few types of expression that could survive in the small time limit $t\to 0$. 

\noindent
We distinguish two possibilities.

\smallskip
\noindent\textbf{(I) Terms ending with a diagonal factor $e^{-tD^2}$.}
We observe that such a term must contain an \emph{even}, $2p$, number of diagonal commutators $[e^{-t D^2},f_i]$, whose Getzler order of each commutator is -3, and in terms from $B$ and $C$ the number of $+$ terms appearing must equal the number of $-$ terms appearing. Notice that the $B_i^+$ and $C_i^+$ are 
Volterra-related operators with Getzler order -1 and $B_i^-$ and $C_i^-$ are Volterra-related operators  with Getzler order -3. The total sum of Getzler order of such a term is 
\[
(-3)\times 2p + (-2)\times (2q-2p)-2=-2p -4q-2. 
\]
To maximize the total order, we therefore consider terms with $p=0$. \\
Hence the only possible terms are of the form
\[
\gamma^{-1} f_0 \, (B_1 \text{ or } C_1)\cdots (B_{2q} \text{ or } C_{2q})\, (e^{-tD^2})^\gamma.
\]
Note that each $C_j$ contains one factor of $D$.  
If more than one $C_j$ appears, there will be at least two $D$'s.  
By Theorem~\ref{theo:moving}, we may move these $D$'s together without increasing the Getzler order.  
Since
\[
\operatorname{ord}_G(D^2)=2 
< \operatorname{ord}_G(D)+\operatorname{ord}_G(D)=4,
\]
such terms are strictly of lower order.  
Hence, at most one $C_j$ may appear.  
The total order of the contributing terms is therefore
\[
q\times(-4) + (-2) = -4q - 2.
\]

\smallskip
\noindent\textbf{(II) Terms ending with an off-diagonal factor $A^\pm$.}
Here, the expression must contain an \emph{odd} number, $2p+1$, of diagonal commutators $[e^{-t D^2},f_i]$.  
By the same reasoning as Case (I), the maximal order occurs when exactly one such commutator appears:
\[
\gamma^{-1} f_0 \, (B_1 \text{ or } C_1)\cdots
\big([e^{-tD^2},f_j]\big)\cdots
(B_{2q} \text{ or } C_{2q})\, A^\gamma.
\]
Since $A$ already contains a factor of $D$, moving $D$ using Theorem \ref{theo:moving} shows that the presence of any $C_j^{\pm}$ between $\big([e^{-t D^2},f_j]\big)$ and $A$ would further lower the Getzler order.  Here if $C_j$ appears before $\big([e^{-t D^2},f_j]\big)$, we use the fact that 
\[
\left[D, \left[e^{-t D^2},f_j\right]\right] = \left[e^{-t D^2},\left[D, f_j\right]\right] = \left[e^{-t D^2}, c(df_j)\right]
\]
and 
\[
-2 = \operatorname{ord}_G\left(  \left[e^{-t D^2}, c(df_j)\right]\right) < \operatorname{ord}_G \left(D\left[e^{-t D^2},f_j\right]\right) = -1.
\]
Therefore, only the case with no $C_j$'s contributes:
\[
\gamma^{-1} f_0\, B_1 B_2 \cdots \big([e^{-t D^2},f_j]\big) \cdots B_{2q-1}B_{2q}\, A^\gamma.
\]
Its total order is
\[
-3 + (q-1)\times(-4) + 1 + (-4)
= -3 + (q-1)\times(-4) + 0 + (-3)
= -4q - 2.
\]

\smallskip
Hence, the only terms that contribute to the limit are exactly those listed above.
\end{proof}
\end{proposition}

\begin{definition}\label{defn:4types}
We define the following four types of integrals:
\begin{enumerate}
	\item Type I: 
\[
T_I(\sqrt{t}D) \colon = \gamma^{-1} f_0 B^\mp_1 B^\pm_2 \cdots B_{2q-1}^\mp B_{2q}^\pm  \left(e^{-t D^2}\right)^\gamma;
\]
\item Type II:
\[
T_{II,j}(\sqrt{t}D) \colon=\gamma^{-1} f_0 B^\pm_1 B^\mp_2 \cdots C_j^- \cdots B^\pm_{2q-1} B^\mp_{2q} \left(e^{-t D^2}\right)^\gamma;
\]

\item Type III:
\[
T_{III,j}(\sqrt{t}D) \colon=\gamma^{-1} f_0 B^\pm_1 B^\mp_2 \cdots C_j^+ \cdots B^\pm_{2q-1} B^\mp_{2q}  \left(e^{-t D^2}\right)^\gamma;
\]
\item Type IV:
\[
T_{IV,j}(\sqrt{t}D) \colon=\gamma^{-1}  f_0 B^\mp_1 B^\pm_2 \cdots \left[   e^{-t D^2} ,f_j\right]\cdots B^-_{2q-1} B^+_{2q} \left(A^\pm \right)^\gamma.
\]
\end{enumerate}	
We want to compute the limit of the trace of each term as $t \downarrow 0$.
\end{definition}

\subsection{Computation of type I}
Put
\[
B^+= e^{-\frac{t}{2}D^2}, \quad B^- = t\frac{1-e^{-tD^2}}{tD^2} e^{-\frac{t}{2} D^2}. 
\]
We have that 
\[
B_i^-=B^-\sigma(df_i),\ \ \ B_i^+=B^+ \sigma(df_i). 
\]
\begin{lemma}
\label{moving lemma}
We have that 
\begin{align}
\lim_{t \downarrow 0^+} \Tr_s^{{\rm geo}}\left(T_I(\sqrt{t}D)\right) = \lim_{t \downarrow 0^+} \Tr_s^{{\rm geo}}\left(\gamma^{-1}  f_0c(df_1)\ldots c(df_{2q}) \cdot B^\pm B^\mp \ldots B^\mp \cdot \left(e^{-t D^2}\right)^\gamma\right), 
\end{align}	
\begin{proof}
We write 
\begin{multline}
\Tr_s^{{\rm geo}}\left(\gamma^{-1} f_0  B_1^\pm  B_2^\mp \ldots B_{2q}^\mp  \left(e^{-t D^2}\right)^\gamma\right)=\Tr_s^{{\rm geo}}\left(\gamma^{-1} f_0  B^\pm c(df_1)  B_2^\mp \ldots B_{2q}^\mp  \left(e^{-t D^2}\right)^\gamma\right)\\
=	\Tr_s^{{\rm geo}}\left(\gamma^{-1} f_0  c(df_1)B^\pm   B_2^\mp \ldots B_{2q}^\mp  \left(e^{-t D^2}\right)^\gamma\right) +\Tr_s^{{\rm geo}}\left(\gamma^{-1} f_0  \left[B^\pm, c(df_1)\right]   B_2^\mp \ldots B_{2q}^\mp  \left(e^{-t D^2}\right)^\gamma\right).
\end{multline}
Since 
\[
\operatorname{ord}_G\left(\left[ B^\pm, c(df_1)\right]\right)\leq \operatorname{ord}_G\left(B^\pm\right), 
\]
the second term in the above equation satisfies the assumptions in Theorem \ref{theo:moving}, 
\[
\lim_{t \downarrow 0^+}\Tr_s^{{\rm geo}}\left(\gamma^{-1}f_0  \left[B^\pm, c(df_1)\right]   B_2^\mp \ldots B_{2q}^\mp  \left(e^{-t D^2}\right)^\gamma\right) = 0. 
\]
Hence, 
\[
\lim_{t \downarrow 0^+} \Tr_s^{{\rm geo}}\left(\gamma^{-1}f_0  B^\pm_1  B_2^\mp \ldots B_{2q}^\mp  \left(e^{-t D^2}\right)^\gamma\right) = \lim_{t \downarrow 0^+}\Tr_s^{{\rm geo}}\left(\gamma^{-1}f_0  c(df_1)B^\pm   B_2^\mp \ldots B_{2q}^\mp  \left(e^{-t D^2}\right)^\gamma\right) 
\]
We prove the lemma by repeating the argument. 
\[
\begin{split}
&\lim_{t \downarrow 0^+} \Tr_s^{{\rm geo}}\left(\gamma^{-1}f_0  B^\pm_1  B_2^\mp \ldots B_{2q}^\mp  \left(e^{-t D^2}\right)^\gamma\right) = \lim_{t \downarrow 0^+}\Tr_s^{{\rm geo}}\left(\gamma^{-1}f_0  c(df_1)c(df_2)B^\pm   B^\mp B^{\pm}_3\ldots B_{2q}^\mp  \left(e^{-t D^2}\right)^\gamma\right) \\
=&\cdots=\lim_{t \downarrow 0^+}\Tr_s^{{\rm geo}}\left(\gamma^{-1}f_0  c(df_1)c(df_2)\cdots c(df_{2q})B^\pm   B^\mp B^{\pm}\ldots B^\mp  \left(e^{-t D^2}\right)^\gamma\right)
\end{split}
\]
\end{proof}
\end{lemma}

\begin{lemma}
\label{lem:remove}
We have that 
\begin{align}
\lim_{t \downarrow 0^+} \Tr^{{\rm geo}}\left(T_I(\sqrt{t}D)\right) = \lim_{t \downarrow 0^+}t^q  \Tr_s^{{\rm geo}}\left( \gamma^{-1} f_0c(df_1)\cdots c(df_{2q})B^\gamma\right),
\end{align}	
where
\begin{align}
B = 	\int_{1}^{2}\ldots\int_{1}^{2}e^{-t\left(1+s_1+\ldots+s_q\right)D^2}ds_1\cdots ds_q. 
\end{align}
\end{lemma}

\begin{proof} Recall (\ref{useful-identity}) that 
\[
t \,\left( \frac{I- e^{-t D^2}}{tD^2} \right) e^{-\frac{t}{2}D^2} =  \int_{\frac{1}{2}}^{\frac{3}{2}} t\,e^{-s_1tD^2}\; ds_1. 
\]

Since $B^-$ commutes with $B^+$, we can rewrite 
\begin{multline*}
B^\pm   B^\mp  \ldots B^\mp  \cdot \left(e^{-t D^2}\right)^\gamma\\
=t^{q}\left( \int_{1/2}^{3/2}\ldots\int_{1/2}^{3/2}e^{-t\left(\frac{q}{2}+1+s_1+\ldots+s_q\right)D^2}ds_1\cdots ds_q\right)^\gamma	\\
= t^q\left(\int_{1}^{2}\ldots\int_{1}^{2}e^{-t\left(1+s_1+\ldots+s_q\right)D^2}ds_1\cdots ds_q\right)^\gamma. 	
\end{multline*}
\end{proof}

\begin{lemma}\label{prop:IQ(s)}
\[
\begin{split}
&\lim_{t \downarrow 0^+}t^q  \Tr_s^{{\rm geo}}\left( \gamma^{-1} f_0c(df_1)\cdots c(df_{2q})B^\gamma\right)\\
=&\int_{1}^{2}\ldots\int_{1}^{2} \lim_{t \downarrow 0^+}t^q\Tr_s^{{\rm geo}}\left(\gamma^{-1} f_0c(df_1)\cdots c(df_{2q})e^{-t\left(1+s_1+\ldots+s_q\right)D^2}\right)^\gamma ds_1\cdots ds_q. 
\end{split}
\]
\end{lemma}
\begin{proof}
Given $(s_1, \cdots, s_q)\in [1,2]^{\times q}$, consider the operator $\widetilde{Q}(s_1, \cdots, s_q)$ by 
\[
\widetilde{Q}(s_1, \cdots, s_q)= P \Big((1+s_1+\cdots +s_q)D^2+\partial_t\Big)^{-1},
\]
where $P = f_0c(df_1) \cdots c(df_{2q})$. Since $f_0, \cdots, f_q$ are compactly supported, the operator $\widetilde{Q}(s_1, \cdots, s_q)$ is a trace class operator. Let $\mathcal{L}^1$ be the Banach space of trace class operators. We observe that the $\mathcal{L}^1$-valued function $\widetilde{Q}: [1,2]^{\times q}\to \mathcal{L}^1$ is continuous. Hence, 
\[
\Tr_s^{{\rm geo}}\left( \gamma^{-1} f_0c(df_1)\cdots c(df_{2q})B^\gamma\right)=\int_{1}^{2}\ldots\int_{1}^{2} t^q\Tr^{{\rm geo}}\left(\gamma^{-1} f_0c(df_1)\cdots c(df_{2q})e^{-t\left(1+s_1+\ldots+s_q\right)D^2}\right)^\gamma ds_1\cdots ds_q. 
\]

Notice that the $t$-asymptotic estimate using the Getzler order is uniform on any compact set, c.f. Theorem \ref{thm:PW-local}, \ref{thm local trace}, \ref{GR thm}, Lemma \ref{lem:Jy}, \ref{lem:symbolest}.  Consider $Q=P(D^2+\partial_t)^{-1}$. Then the smoothing (Schwartz) kernels of $Q(s)$ and $Q$ satisfy the following relation, 
\[
Q(s)(x,y,t)=s^{-1}Q(x,y, st), \quad s, t > 0. 
\] 
Accordingly, the $t$-asymptotic estimate of $\widetilde{Q}(s_1, \cdots, s_q)(x,\gamma x,t)$ is uniform on $(s_1, \cdots, s_q)\in [1,2]^{\times q}$ and $x$. This allows us to commute the limit of $t\downarrow 0$ with the integral over $[0,1]^{\times q}$ and get
\[
\begin{split}
&\lim_{t \downarrow 0^+}t^q  \Tr_s^{{\rm geo}}\left( \gamma^{-1} f_0c(df_1)\cdots c(df_{2q})B^\gamma\right)\\
=&\int_{1}^{2}\ldots\int_{1}^{2} \lim_{t \downarrow 0^+}t^q\Tr_s^{{\rm geo}}\left(\gamma^{-1} f_0c(df_1)\cdots c(df_{2q})e^{-t\left(1+s_1+\ldots+s_q\right)D^2}\right)^\gamma ds_1\cdots ds_q. 
\end{split}
\]
\end{proof}

We now evaluate the limit $t\downarrow 0$ of the term
\[
t^q\Tr_s^{{\rm geo}}\left(\gamma^{-1} f_0c(df_1)\cdots c(df_{2q})e^{-t\left(1+s_1+\ldots+s_q\right)D^2}\right)^\gamma.
\]

Set $s=1+s_1+\cdots+s_q$. Consider
\[
Q(s) = P (sD^2+\partial_t)^{-1}, 
\]
where $P = f_0c(df_1) \cdots c(df_{2q})$ is a differential operator with Getzler order $2q$ and 
\[
P_{(2q)} = f_0df_1 \wedge \ldots \wedge f_{2q}. 
\]
Write 
\[
K_{Q(s)}(x, y, t)
\]
the smoothing kernel of $Q(s)$. 

For any $x \in M^\gamma_a$, we have defined 
\[
I_{Q(s)}(x, t, \gamma ):= \int_{N_x^\gamma(\epsilon)}  K_{Q(s)}\left(\exp _x v, \exp _x\left(\gamma v\right), t\right)|d v|.
\]

\begin{proposition} \label{lem:IQ(s)}For $a\geq  2q$,
\[
\begin{aligned}
\sigma\left[\gamma^{-1} I_{Q(s)}(x,t, \gamma)\right]^{(a, 0)}=& (st)^{-q}\frac{(4 \pi )^{-\frac{q}{2}}}{\operatorname{det}^{\frac{1}{2}}\left(1-\gamma|_{N^\gamma}\right)} \cdot\left[f_0df_1\wedge\ldots\wedge df_{2q}\right]^{(2q,0)}
\\& \wedge \left[\operatorname{det}^{\frac{1}{2}}\left(\frac{ R^{\prime} / 2}{\sinh \left( R^{\prime} / 2\right)}\right) \operatorname{det}^{-\frac{1}{2}}\left(1-\gamma|_{N^\gamma} e^{- R^{\prime \prime}}\right)\wedge \gamma^{V} e^{-F^V} \right]^{(a-2q, 0)}+O(t^{-q+\frac{1}{2}}).	
\end{aligned}
\]
\end{proposition}

\begin{proof}By Theorem \ref{thm local trace}, $[\gamma^{-1} I_{Q(s)}(x,t,\gamma)]^{(a,0)}$ can be computed using $\left[\gamma^{-1} I_{P_{(2q)}(H_R+\partial_t)^{-1}}(x,t,\gamma)\right]^{(a,0)}$. Thanks to Theorem \ref{GR thm}, we compute $\sigma\left[\gamma^{-1} I_{Q(s)}(x,t,\gamma)\right]^{(a,0)}$ to be 
\[
\left[fdf_1\cdots df_{2q}\wedge \frac{(4 \pi st)^{-\frac{a}{2}}}{\operatorname{det}^{\frac{1}{2}}\left(1-\gamma|_{N^\gamma}\right)}\operatorname{det}^{\frac{1}{2}}\left(\frac{st R^{\prime} / 2}{\sinh \left(st R^{\prime} / 2\right)}\right) \operatorname{det}^{-\frac{1}{2}}\left(1-\gamma|_{N^\gamma} e^{-st R^{\prime \prime}}\right)\right]^{(a,0)}.
\]
We observe that Theorem \ref{GR thm} is proved on compact manifolds. The same proof generalizes to prove this proposition as $f_0df_1\cdots df_{2q}$ is of compact support. 
\end{proof}

Set 
\[
\beta_q = \int_{1}^{2}\ldots\int_{1}^{2}(1 +  s_1+ \cdots s_q)^{-q} \; ds_1\ldots ds_q.
\]
By  Lemma \ref{moving lemma}, \ref{lem:remove}, and Proposition \ref{prop:IQ(s)}, we conclude the following result:
\begin{theorem}For the Type I terms in Definition \ref{defn:4types}, 
\begin{multline}\label{eq term 1}
\lim_{t\o 0^+}\Tr^{\rm geo}_s(T_I(\sqrt{t}D))\\
=(-i)^{\frac{n}{2}} (2 \pi )^{-\frac{a}{2}}\cdot \beta_q \cdot \left[f_0df_1\wedge\ldots\wedge df_{2q}\right]^{(2q,0)}\wedge \left[\operatorname{det}^{\frac{1}{2}}\left(\frac{ R^{\prime} / 2}{\sinh \left( R^{\prime} / 2\right)}\right) \operatorname{det}^{-\frac{1}{2}}\left(1-\gamma|_{N^\gamma} e^{- R^{\prime \prime}}\right)\wedge \Tr(\gamma e^{-F^V}) \right]^{(a-2q, 0)}.
\end{multline}	
\end{theorem}

\subsection{Computation of type II}
In this subsection, we will compute 
\[
\Tr^{\rm geo}_s\Big(T_{II, j}(\sqrt{t}D) \Big)\colon = 
\Tr_s^{{\rm geo}}\left(\gamma^{-1}  f_0 B^- c(df_1) B^+ c(df_2) \cdots   C^-_j  \cdots  B^+ c(df_{2q}) \left(e^{-t D^2}\right)^\gamma\right),
 \]
where  $u(x) = \int_{\frac{1}{2}}^{\frac{3}{2}}e^{-sx}ds = \frac{1 - e^{-x}}{x}e^{-\frac{x}{2}}$ and 
 \begin{align*}
  C_j^-  \colon =& t[u (tD^2), f_j]D	\\
  =& t\left[\frac{1}{2\pi i} \int_C u(t\lambda)(\lambda + D^2)^{-1}  \; d\lambda, f_j \right] D\\
  =&\frac{t}{2\pi i} \int_C u(t\lambda)\left[(\lambda + D^2)^{-1}, f_j \right] \; d\lambda \cdot D \\
  =&-\frac{t}{2\pi i} \int_C u(t\lambda)(\lambda + D^2)^{-1} \left[D^2, f_j \right] D  (\lambda + D^2)^{-1} \; d\lambda.
 \end{align*}
 Since $f_0$ has compact support,  the family 
 $$
 \lambda\mapsto f_0 B^- c(df_1) B^+ c(df_2)\cdots	 u(t\lambda)(\lambda + D^2)^{-1} \left[D^2, f_j \right] D  (\lambda + D^2)^{-1} B^+ c(df_{j+1}) \ldots  B^+ c(df_{2q}) e^{-t D^2}
 $$
 is an integrable family of trace class operators on $C$. Hence, we can compute the trace of $T_{II, j}(\sqrt{t}D)$ by switching the order between trace and the contour integral, i.e.  
\begin{align*}
&\Tr_s^{\rm geo}\Big(T_{II,j}(\sqrt{t}D)\Big)=-\frac{t}{2\pi i} \int_C \Tr_s^{{\rm geo}}\Big(\gamma^{-1} f_0 B^- c(df_1) B^+ c(df_2)\\
&\hspace{3cm} \cdots	 u(t\lambda)(\lambda + D^2)^{-1} \left[D^2, f_j \right] D  (\lambda + D^2)^{-1} B^+ c(df_{j+1}) \ldots  B^+ c(df_{2q}) \left(e^{-t D^2}\right)^\gamma\Big).	
\end{align*}
We observe that $B^{\pm}$ are Volterra-related and $(\lambda+D^2)^{-1}$ is a Volterra operator. As the argument in Lemma \ref{moving lemma}, using Theorem \ref{theo:moving}, we can compute the leading order term of $T_{II, j}(\sqrt{t}D)$ as follows,
\begin{multline}\label{equation cterm}
\Tr_s^{\rm geo}\big(T_{II, j}(\sqrt{t}D) \Big)= -\frac{t}{2\pi i} \int_C  \Tr_s^{{\rm geo}}\Bigg[\gamma^{-1} B^-  B^+ \cdots u(t\lambda)(\lambda + D^2)^{-1} \\
\left( f_0c(df_1) \ldots c(df_{j-1}) \left([D^2, f_j]D \right) c(df_{j+1}) \ldots c(df_{2q})  (\lambda + D^2)^{-1} B^+ \ldots (B^+)^3\right)^\gamma\Bigg]+ O(t). 
\end{multline}
In the above equation, we have used the identity that $e^{-tD^2}=(B^+)^2$. Notice that the operators 
\begin{align}
e^{-t D^2} B^-  B^+ \cdots u(\lambda)(\lambda + tD^2)^{-1} \ \ \ \text{and}\ \ \ \ 	  (\lambda + tD^2)^{-1} B^+ \ldots (B^+)^3
\end{align}
are both $\Gamma$-invariant but 
\begin{align}
f_0c(df_1) \ldots c(df_{j-1})  \left([D^2, f_j ]D \right) c(df_{j+1}) \ldots c(df_{2q})	
\end{align}  
is not. Nevertheless, we have the following lemma. 
\begin{lemma}\label{lem:trace}
Suppose that $K_1, K_2$ are two $\Gamma$ invariant smoothing operators on $M$. If $P$ is a differential operator  with compact support (in particular, $P$ may not be $\Gamma$ invariant), then 
\begin{align}
\Tr^{{\rm geo}}(\gamma^{-1} K_1PK_2^\gamma) = \Tr^{\rm geo}(\gamma^{-1} PK_2K_1^\gamma) . 
\end{align} 	
\begin{proof}
By definition, 
\begin{align*}
\Tr^{\rm geo}(\gamma^{-1} K_1PK_2^\gamma)  = &\int_{M \times M} \operatorname{tr}\big(\gamma^{-1} k_1(x, y) p(y) k_2(y, \gamma  x)\big) \;dxdy	\\
=&\int_{M \times M} \operatorname{tr}\big(k_1(x, \gamma y)\gamma^{-1} p(y) k_2(y,  x)\big) \;dxdy = \Tr^{\rm geo}(PK_2K_1^\gamma). 
\end{align*}
In the second equality, we have applied the change of variable $x\mapsto \gamma^{-1}x$, and also the $\gamma$-invariance of the kernel $k_1$, i.e. $\gamma^{-1} k_1(\gamma^{-1} x, y)\gamma =k_1(x, \gamma y)$.  
\end{proof}
\end{lemma}

\noindent
Apply the above lemma to (\ref{equation cterm}), we see that 
\begin{multline}\label{eq:TIIjtrace}
\Tr^{\rm geo}_s\left(T_{II, j}(\sqrt{t}D)\right) = -\frac{t}{2\pi i} \int_C  \Tr_s^{\rm geo}\Bigg[\gamma^{-1} f_0c(df_1) \ldots c(df_{j-1})  \left([D^2, f_j ]D \right) c(df_{j+1}) \ldots c(df_{2q})  \\
\left( (\lambda + D^2)^{-1}B^q u(t\lambda)(\lambda + D^2)^{-1}\right)^\gamma\Bigg] \; d\lambda + O(t),
\end{multline}
where $B^q=(B^+)^{q+2}(B^-)^{q-1}$. 

\noindent
 Since $f_0$ has compact support,  the family 
 $$
\lambda\mapsto f_0c(df_1) \ldots c(df_{j-1})  \left([D^2, f_j ]D \right) c(df_{j+1}) \ldots c(df_{2q})  \\
\left( (\lambda + D^2)^{-1}B^q u(t\lambda)(\lambda + D^2)^{-1}\right)^\gamma
$$
 is an integrable family of trace class operators on $C$. Hence, we can compute the first term in the right side of (\ref{eq:TIIjtrace}) by switching the order between trace and the contour integral, i.e.  
\begin{multline}\label{eq:T-II-contour}
\Tr^{\rm geo}_s\Big(T_{II,j}(\sqrt{t}D)\Big)
= -\Tr^{\rm geo}_s\Bigg[
\gamma^{-1} f_0c(df_1) \ldots c(df_{j-1}) \left([D^2, f_j ]D \right) c(df_{j+1}) \ldots c(df_{2q})  \\
\left(\frac{t}{2\pi i} \int_C  (\lambda + D^2)^{-1}B^q u(t\lambda)(\lambda + D^2)^{-1} \; d\lambda \right)^\gamma\Bigg]+ O(t).
\end{multline}

\noindent
Moreover, as $B^q$ commutes with $D$, we can rewrite the contour integral as follows,
\begin{multline}\label{eq:contour-I}
\frac{1}{2\pi i}\int_C  t(\lambda + D^2)^{-1}B^q u(t\lambda)(\lambda + D^2)^{-1} \; d\lambda \\
=	\frac{1}{2\pi i}\int_C tB^q u(t\lambda)(\lambda + D^2)^{-2} \; d\lambda =t^2B^q u'(tD^2)\\
=t^{q+1} (B^+)^{q+2} (B^-)^{q-1}u'(tD^2)=t^{q+1} e^{-t(1+\frac{q}{2})D^2}u^{q-1}(tD^2)u'(tD^2)=\frac{t^{q+1}}{q} e^{-t(1+\frac{q}{2})D^2} \left(u^q \right)'(tD^2)	\\
 = \frac{t^{q+1}}{q}\int_{1/2}^{3/2}\ldots\int_{1/2}^{3/2}(s_1+\cdots s_{q})\cdot e^{-(\frac{q}{2}+1 + s_1+\ldots+s_{q})tD^2}ds_1\cdots ds_q\\
=\frac{t^{q+1}}{q}\int_1^2\ldots\int_1^2 \left(s_1+\cdots s_{q} -\frac{q}{2}\right)\cdot e^{-(1 + s_1+\ldots+s_{q})tD^2}ds_1\cdots ds_q.
\end{multline}

\noindent
Substituting the expression (\ref{eq:contour-I}) of the contour integral into (\ref{eq:T-II-contour}) about the $\Tr_s^{\rm geo}(T_{II,j}(\sqrt{t}D))$, we have the following formula,
\[
\begin{split}
\Tr_s^{\rm geo}\Big(T_{II,j}(\sqrt{t}D)\Big)
=&-\Tr^{\rm geo}_s\Bigg[\gamma^{-1} f_0c(df_1) \ldots c(df_{j-1}) \left([D^2, f_j ]D \right) c(df_{j+1}) \ldots c(df_{2q})  \\
&\left(\frac{t^{q+1}}{q}\int_1^2\ldots\int_1^2 \left(s_1+\cdots s_{q} -\frac{q}{2}\right)\cdot e^{-(1 + s_1+\ldots+s_{q})tD^2}ds_1\cdots ds_q\right)^\gamma\Bigg]+ O(t).
\end{split}
\]
By the same argument as Lemma \ref{prop:IQ(s)}, we can pull the integral outside the trace and have
\begin{equation}\label{eq:TIIj-t}
\begin{split}
\Tr_s^{\rm geo}\Big(T_{II,j}(\sqrt{t}D)\Big)
=&-\int_1^2\ldots\int_1^2\left(s_1+\cdots s_{q} -\frac{q}{2}\right)ds_1\cdots ds_q\\
&\frac{t^{q+1}}{q} \Tr_s^{\rm geo}\Bigg[\gamma^{-1} f_0c(df_1) \ldots c(df_{j-1}) \big( [D^2, f_j ]D\big) c(df_{j+1}) \ldots c(df_{2q})  \big( e^{-(1 + s_1+\ldots+s_{q})tD^2}\big)^\gamma\Bigg]+ O(t).
\end{split}
\end{equation}
As before, we consider the Volterra-related operator $Q(s) = P_1RP_2 (sD^2+\partial_t)^{-1} $ where 
\[
P_1 =  f_0 c(df_1) \cdots c(df_{j-1}), \qquad P_2=c(df_{j+1}) \cdots c(df_{2q}), \quad  R =  \left[D^2, f_j \right] D. 
\]
Following the computation in the proof of Lemma \ref{lem:getzler},  we get that the Getzler symbol of $R$ equals $-2\langle df_j , \xi \rangle \xi$. Accordingly, its model operator is given by 
\[
R_{(3)} = -2\sum_{a, b}\frac{\partial }{\partial x_a}\frac{\partial }{\partial x_b} e_a(f_j)e_b^*\wedge,
\]
where $\{e_a\}$ is an orthonormal basis of $T_xM$ and $\{e_b^*\}$ its dual basis. 

\begin{lemma}
\label{RG lem}
The model operator of $P_1RP_2(sD^2+\partial_t)^{-1}$ is computed as follows, 
\[
(P_1)_{(j-1)}R_{(3)}(P_2)_{(2q-j)}G_R(x, y, t) = \frac{1}{ t}f_0df_1\wedge\cdots\wedge df_j\wedge\cdots \wedge df_{2q}G_R(x,y,t).
\]	
\end{lemma}
\begin{proof}
Recall that the Mehler kernel of $G_R(x,y,t)$ has the following expression, 
\[
G_R(x,y, t):=\frac{1}{(2\pi t)^{n/2}}{\rm det}^{1/2}\left(\frac{tR/2}{\sinh(tR/2)}\right)\exp\left(-\frac{1}{4t}\Theta(x,y,t)\right),
\]
where $\Theta(x,y,t)$ has the following form
\[
\Theta(x,y,t)=\langle \frac{tR/2}{\tanh(tR/2)}x,x\rangle+\langle \frac{tR/2}{\tanh(tR/2)}y,y\rangle-2\langle \frac{tR/2}{\sinh(tR/2)}e^{tR/2} x,y\rangle.
\]
Applying $(P_1)_{(j-1)}R_{(3)}(P_2)_{(2q-j)}$ to $G_R$, we obtain 
\[
\begin{split}
&(P_1)_{(j-1)}R_{(3)}(P_2)_{(2q-j)}G_R(x,y,t)\\
=&\frac{1}{(4 t)^{n/2+1}\pi^{n/2}}{\rm det}^{1/2}\left(\frac{tR/2}{\sinh(tR/2)}\right)\\
&4\sum_{a,b}f_0df_1\wedge \cdots \wedge e_a(f_j)e_b^*\wedge df_{j+1}\cdots \wedge df_{2q} \Big(\frac{tR/2}{\tanh(tR/2)}\Big)_{ab}\exp\left(-\frac{1}{4t}\Theta(x,y,t)\right).
\end{split}
\]
Recall that the Bianchi identity implies
\[
\sum_b e_b^*\wedge R_{ab}=0,
\]
with $R_{ab}=\langle Re_a, e_b\rangle$. 
It follows that 
\[
\sum_b e_b^*\wedge  \Big(\frac{tR/2}{\tanh(tR/2)}\Big)_{ab}=\delta_{ab}=\left\{\begin{array}{ll}1&a=b\\ 0&\text{otherwise}\end{array}\right. .
\]
We conclude with the following expression of $(P_1)_{(j-1)}R_{(3)}(P_2)_{(2q-j)}G_R(x, y, t)$,
\[
(P_1)_{(j-1)}R_{(3)}(P_2)_{(2q-j)}G_R(x, y, t)=\frac{1}{ t}f_0df_1\wedge\cdots \wedge df_j\wedge \cdots \wedge df_q G_R(x,y,t) .
\]
\end{proof}

\begin{lemma}
\label{lem:I_Q-II}For the Volterra-related operator $Q(s) = P_1RP_2 (sD^2+\partial_t)^{-1} $,
\[
\begin{aligned}
\sigma\left[\gamma^{-1}I_{Q(s)}(x,t, \gamma)\right]^{(a, 0)}=& \frac{2^{-a}\pi^{-\frac{a}{2}} t^{-q-1}s^{-q-1}}{\operatorname{det}^{\frac{1}{2}}\left(1-\gamma|_{N^\gamma}\right)} \cdot \left[f_0df_1\wedge\ldots\wedge df_{2q}\right]^{(2q,0)}
\\
& \wedge \left[\operatorname{det}^{\frac{1}{2}}\left(\frac{ R^{\prime} / 2}{\sinh \left(R^{\prime} / 2\right)}\right) \operatorname{det}^{-\frac{1}{2}}\left(1-\gamma|_{N^\gamma} e^{- R^{\prime \prime}}\right)\wedge e^{-F^V} \right]^{(a-2q, 0)}+O(t^{-q}).
\end{aligned}
\]

\end{lemma}

\begin{proof} By Theorem \ref{thm local trace}, we compute $\sigma[\gamma^{-1}I_{Q(s)}(x,t,\gamma)]^{(a,0)}$ by the model operator $Q(s)_{(2q-2)}$ of $Q(s)$. Using Lemma \ref{RG lem} and Theorem \ref{GR thm}, we compute $
\sigma[\gamma^{-1}I_{Q(s)_{(2q-2)}}]^{(a,0)}$ like Proposition \ref{prop:IQ(s)} and obtain
\[
\sigma\left[\gamma^{-1}I_{Q(s)}(x,t, \gamma)\right]^{(a, 0)}=4\pi\left[f_0df_1\cdots df_{2q}\wedge \frac{(4 \pi st)^{-\frac{a}{2}-1}}{\operatorname{det}^{\frac{1}{2}}\left(1-\gamma|_{N^\gamma}\right)}\operatorname{det}^{\frac{1}{2}}\left(\frac{st R^{\prime} / 2}{\sinh \left(st R^{\prime} / 2\right)}\right) \operatorname{det}^{-\frac{1}{2}}\left(1-\gamma|_{N^\gamma} e^{-st R^{\prime \prime}}\right)\right]^{(a,0)}.
\]

Factoring out the power of $st$, we obtain the desired expression of $\sigma[\gamma^{-1}I_{Q(s)}(x,t,\gamma)]^{(a,0)}$,
\[
\begin{aligned}
\sigma\left[\gamma^{-1}I_{Q(s)}(x,t, \gamma)\right]^{(a, 0)}=&  \frac{2^{-a}\pi^{-\frac{a}{2}} t^{-q-1}s^{-q-1}}{\operatorname{det}^{\frac{1}{2}}\left(1-\gamma|_{N^\gamma}\right)} \cdot \left[f_0df_1\wedge\ldots\wedge df_{2q}\right]^{(2q,0)}
\\
& \wedge \left[\operatorname{det}^{\frac{1}{2}}\left(\frac{ R^{\prime} / 2}{\sinh \left(R^{\prime} / 2\right)}\right) \operatorname{det}^{-\frac{1}{2}}\left(1-\gamma|_{N^\gamma} e^{- R^{\prime \prime}}\right)\wedge e^{-F^V} \right]^{(a-2q, 0)}+O(t^{-q}).
\end{aligned}
\]
\end{proof}

Notice that $Q(s) = P_1RP_2 (sD^2+\partial_t)^{-1} $ is a Volterra-related operator with Getzler order equal to $2q$. The above expression of $\sigma\left[I_{Q(s)}(x,t, \gamma)\right]^{(a, 0)}$ allows us to compute $\left(\operatorname{str}_{\mathcal{S} \otimes E}\right)\left[\gamma^{-1}  I_{T_{II, j}(\sqrt{t}D)}(x, t, \gamma)\right]$ using Theorem \ref{thm local trace}.
\begin{equation}
\label{eq term 2}	
\begin{aligned}
 &\left(\operatorname{str}_{\mathcal{S} \otimes E}\right)\left[\gamma^{-1} I_{T_{II, j}(\sqrt{t}D)}(x, t, \gamma)\right] \\
  \quad=&-(-i)^{\frac{n}{2}} (2 \pi )^{-\frac{a}{2}}\cdot \frac{\delta_q}{q} \cdot \left[f_0df_1\wedge\ldots\wedge df_{2q}\right]^{(2q,0)}
\\& \wedge \left[\operatorname{det}^{\frac{1}{2}}\left(\frac{ R^{\prime} / 2}{\sinh \left( R^{\prime} / 2\right)}\right) \operatorname{det}^{-\frac{1}{2}}\left(1-\gamma|_{N^\gamma} e^{- R^{\prime \prime}}\right)\wedge \Tr(\gamma e^{-F^V}) \right]^{(a-2q, 0)}+O(t),
\end{aligned}
\end{equation}
where 
\[
\begin{aligned}
\delta_q =& \int_1^2\ldots\int_1^2 \left(s_1+\cdots s_{q} -\frac{q}{2}\right)\cdot \left(1 + s_1+\ldots+s_{q}\right)^{-q-1}
 \;ds_1\cdots ds_q\\
=& \beta_q - \frac{q+2}{2} \cdot  \int_1^2\ldots\int_1^2  \left(1 + s_1+\ldots+s_{q}\right)^{-q-1}  \;ds_1\cdots ds_q\\
=&\beta_q-\frac{q+2}{2}\alpha_q, 
\end{aligned}
\]
with 
\[
\alpha_q=\int_1^2\ldots\int_1^2  \left(1 + s_1+\ldots+s_{q}\right)^{-q-1}  \;ds_1\cdots ds_q.
\]

\subsection{Computation of type III and  type IV}
We first compute term (3): 
\[
T_{III,j}(\sqrt{t}D) \colon = 
\gamma^{-1} f_0 B^\pm c(df_1) B^\mp c(df_2) \cdots   C^+_j  \cdots  B^\pm c(df_{2q})\big(e^{-t D^2}\big)^\gamma
 \]
We write 
\begin{align*}
C^+_j  =& \left[\frac{1}{2\pi i} \int_C e^{-\frac{t\lambda}{2}}(\lambda + D^2)^{-1}  \; d\lambda, f_j \right]  \cdot D\\
  =&\frac{1}{2\pi i} \int_C e^{-\frac{t\lambda}{2}}\left[(\lambda + D^2)^{-1}, f_j \right] \; d\lambda \cdot D \\
  =&-\frac{1}{2\pi i} \int_C e^{-\frac{t\lambda}{2}}(\lambda + D^2)^{-1} \left[D^2, f_j \right] D  (\lambda + D^2)^{-1} \; d\lambda.
\end{align*} 
By the same argument as in the study of type II terms, we obtain
\[
\begin{split}
&\Tr_s^{\rm geo}\big(T_{III,j}(\sqrt{t}D)\big)\\
=&-\Tr_s^{\rm geo}\Big(\gamma^{-1} f_0 B^\pm c(df_1) B^\mp c(df_2) \cdots  \\
&\qquad \frac{1}{2\pi i} \int_C e^{-\frac{t\lambda}{2}}(\lambda + D^2)^{-1} \left[D^2, f_j \right] D  (\lambda + D^2)^{-1} \; d\lambda \cdots  B^\pm c(df_{2q})\big(e^{-t D^2}\big)^\gamma \Big)\\
=&-\frac{1}{2\pi i}\int_C e^{-\frac{t\lambda}{2}}d\lambda \\
&\Tr_s^{\rm geo}\Big(\gamma^{-1} f_0 B^\pm c(df_1) B^\mp c(df_2) \cdots  (\lambda + D^2)^{-1} \left[D^2, f_j \right] D  (\lambda + D^2)^{-1} \;  \cdots  B^\pm c(df_{2q})\big(e^{-t D^2}\big)^\gamma \Big)
\end{split}
\]

By Theorem \ref{theo:moving}, we can move $c(df_i)$ and $f_0$ with  $B^\pm$ and $(\lambda+D^2)^{-1}$ in the above trace and obtain
\[
\begin{split}
&\Tr_s^{\rm geo}\big(T_{III,j}(\sqrt{t}D)\big)\\
=&-\frac{1}{2\pi i}\int_C e^{-\frac{t\lambda}{2}}d\lambda \\
&\Tr_s^{\rm geo}\Big(\gamma^{-1} B^\pm B^{\mp}\cdots B^{-}  (\lambda + D^2)^{-1} f_0c(df_1)\cdots c(df_{j-1})\left[D^2, f_j \right] D c(df_{j+1})\cdots c(df_{2q}) \\
&\qquad (\lambda + D^2)^{-1}  B^- \cdots B^{\pm}B^{\mp}\big(e^{-t D^2}\big)^\gamma \Big)+ O(t)
\end{split}
\]

Set 
\[
\begin{split}
K_1&=B^\pm B^{\mp}\cdots B^{-}  (\lambda + D^2)^{-1},\\
P&=f_0c(df_1)\cdots c(df_{j-1})\left[D^2, f_j \right] D c(df_{j+1})\cdots c(df_{2q}),\\
K_2&=(\lambda + D^2)^{-1}B^- \cdots B^{\pm}B^{\mp}\big(e^{-t D^2}\big).
\end{split}
\] 
Applying Lemma \ref{lem:trace}, we have
\[
\begin{split}
&\Tr_s^{\rm geo}\big(T_{III,j}(\sqrt{t}D)\big)\\
=&-\frac{1}{2\pi i}\int_C e^{-\frac{t\lambda}{2}}d\lambda \\
&\Tr_s^{\rm geo}\Big(\gamma^{-1} f_0c(df_1)\cdots c(df_{j-1})\left[D^2, f_j \right] D c(df_{j+1})\cdots c(df_{2q}) (\lambda + D^2)^{-1}B^- \cdots B^{\pm}B^{\mp}\big(e^{-t D^2}\big)\\
&\qquad \big(B^\pm B^{\mp}\cdots B^{-}  (\lambda + D^2)^{-1} \big)^\gamma \Big)+ O(t)\\
=&-\frac{1}{2\pi i}\int_C e^{-\frac{t\lambda}{2}}d\lambda \\
&\Tr_s^{\rm geo}\Big(\gamma^{-1} f_0c(df_1)\cdots c(df_{j-1})\left[D^2, f_j \right] D c(df_{j+1})\cdots c(df_{2q}) \big((\lambda + D^2)^{-2}(B^-)^{q}  (B^{+})^{q+1} \big)^\gamma \Big)+ O(t)\\
=&-\frac{t}{2}\Tr_s^{\rm geo}\Big( \gamma^{-1} f_0c(df_1)\cdots c(df_{j-1})\left[D^2, f_j \right] D c(df_{j+1})\cdots c(df_{2q}) \big((B^-)^q(B^+)^{q+2}\big)^\gamma\Big)+ O(t),
\end{split}
\]
where in the last equality we have applied integration by parts on the contour integral. \\
Notice that we can rewrite $(B^-)^q(B^+)^{q+2}$  as 
\[
\begin{split}
(B^-)^q(B^+)^{q+2}=&-t^{q+1}\int _{\frac{1}{2}}^{\frac{3}{2}}ds_1\cdots \int _{\frac{1}{2}}^{\frac{3}{2}}ds_q e^{-(s_1+\cdots+s_q+1+\frac{q}{2})tD^2}\\
=&-t^{q+1}\int_1^2\cdots \int_1^2 e^{-t(1+s_1+...+s_q)D^2} ds_1\cdots ds_q.
\end{split}
\]
We use this expression in the same way as (\ref{eq:TIIj-t}) and get 
\[
\begin{split}
\Tr^{\rm geo}(T_{III,j}(\sqrt{t}D))=&-\frac{t^{q+1}}{2}\int_1^2\cdots \int_1^2 ds_1\cdots ds_q \\
& \Tr^{\rm geo}\Big( \gamma^{-1} f_0c(df_1)\cdots c(df_{j-1})\left[D^2, f_j \right] D c(df_{j+1})\cdots c(df_{2q}) \big(e^{-t(1+s_1+...+s_q)D^2} \big)^\gamma\Big)+ O(t).
\end{split}
\]
This expression of $\Tr^{\rm geo}(T_{III,j}(\sqrt{t}D))$ allows us to apply the same computation as Lemma \ref{RG lem}, \ref{lem:I_Q-II}, and Eq. (\ref{eq term 2}). We reach the following final result. 
\begin{equation}\label{eq:term 3}	
\begin{aligned}
 &\left(\operatorname{str}_{\mathcal{S} \otimes E}\right)\left[  I_{T_{III, j}(\sqrt{t}D)}(x, t, \gamma)\right] \\
  \quad=&-(-i)^{\frac{n}{2}} (2 \pi )^{-\frac{a}{2}}\cdot \alpha_q \cdot \left[f_0df_1\wedge\ldots\wedge df_{2q}\right]^{(2q,0)}
\\& \wedge \left[\operatorname{det}^{\frac{1}{2}}\left(\frac{ R^{\prime} / 2}{\sinh \left( R^{\prime} / 2\right)}\right) \operatorname{det}^{-\frac{1}{2}}\left(1-\gamma|_{N^\gamma} e^{- R^{\prime \prime}}\right)\wedge \Tr(\gamma e^{-F^V}) \right]^{(a-2q, 0)}+O(t),
\end{aligned}
\end{equation}

The computation of type IV splits into two different subcases.
\begin{enumerate}
\item 
$
T_{IV,j,+}(\sqrt{t}D) \colon=\gamma^{-1}  f_0 B^+_1 B^-_2 \cdots \left[   e^{-t D^2} ,f_j\right]\cdots B^+_{2q-1} B^-_{2q} \left(A^+ \right)^\gamma,
$
\item
$
T_{IV,j,-}(\sqrt{t}D) \colon=\gamma^{-1}  f_0 B^-_1  B^+_2 \cdots \left[   e^{-t D^2} ,f_j\right]\cdots B^-_{2q-1} B^+_{2q} \left(A^- \right)^\gamma.
$
\end{enumerate}

The computation of $\Tr^{\rm geo}\big(T_{IV, j,+}(\sqrt{t}D)\big)$ and $\Tr^{\rm geo}\big(T_{IV, j,-}(\sqrt{t}D)\big)$ uses the similar idea as the computation of type III terms via the contour integral
\[
e^{-tD^2}=\frac{1}{2\pi i}\int_C e^{-t\lambda}(\lambda+D^2)^{-1}d\lambda,\qquad [e^{-tD^2}, f_j]=-\frac{1}{2\pi i}\int_C e^{-t\lambda}(\lambda+D^2)^{-1}[D^2,f](\lambda+D^2)^{-1}d\lambda.
\]

We can compute trace of $T_{IV, j, +}(\sqrt{t}D)$ as follows.
\[
\begin{split}
&\Tr_s^{\rm geo}\big(T_{IV, j, +}(\sqrt{t}D)\big)\\
=&(-1)^jt\Tr_s^{\rm geo}\Big( \gamma^{-1} f_0c(df_1)\cdots c(df_{j-1})\left[D^2, f_j \right] D c(df_{j+1})\cdots c(df_{2q}) \big((B^-)^{q}(B^+)^{q+2}\big)^\gamma\Big)+ O(t)\\
=&(-1)^{j}t^{q+1}\Tr_s^{\rm geo}\Big( \gamma^{-1} f_0c(df_1)\cdots c(df_{j-1})\left[D^2, f_j \right] D c(df_{j+1})\cdots c(df_{2q}) \\
&\int_{\frac{1}{2}}^{\frac{3}{2}}\cdots \int_{\frac{1}{2}}^{\frac{3}{2}}ds_1\cdots ds_{q}\big(e^{-t(s_1+\cdots s_{q})D^2}e^{-\frac{t(q+2)}{2}D^2}\big)^\gamma\Big)+O(t)\\
=&(-1)^{j}\int_{\frac{1}{2}}^{\frac{3}{2}}\cdots \int_{\frac{1}{2}}^{\frac{3}{2}}ds_1\cdots ds_{q} t^{q+1}\Tr_s^{\rm geo}\Big( \gamma^{-1} f_0c(df_1)\cdots c(df_{j-1})\left[D^2, f_j \right] D c(df_{j+1})\cdots c(df_{2q})\\
&\qquad\qquad\qquad\big(e^{-t(s_1+\cdots s_{q}+\frac{q}{2}+1)D^2}\big)^\gamma\Big)+O(t)\\
=&(-1)^{j}\int_{1}^{2}\cdots \int_{1}^{2}ds_1\cdots ds_{q-1} t^{q+1}\Tr_s^{\rm geo}\Big( \gamma^{-1} f_0c(df_1)\cdots c(df_{j-1})\left[D^2, f_j \right] D c(df_{j+1})\cdots c(df_{2q})\\
&\qquad\qquad\qquad\big(e^{-t(s_1+\cdots s_{q}+1)D^2}\big)^\gamma\Big)+O(t)\\
=&(-1)^{j}(-i)^{\frac{n}{2}} (2 \pi )^{-\frac{a}{2}}\cdot \alpha_q \cdot \left[f_0df_1\wedge\ldots\wedge df_{2q}\right]^{(2q,0)}
\\& \wedge \left[\operatorname{det}^{\frac{1}{2}}\left(\frac{ R^{\prime} / 2}{\sinh \left( R^{\prime} / 2\right)}\right) \operatorname{det}^{-\frac{1}{2}}\left(1-\gamma|_{N^\gamma} e^{- R^{\prime \prime}}\right)\wedge \Tr(\gamma e^{-F^V}) \right]^{(a-2q, 0)}+O(t).
\end{split}
\]

Similarly, we compute $\Tr^{\rm geo}(T_{IV, j, -}(\sqrt{t}D))$ as follows. 
\[
\begin{split}
&\Tr_s^{\rm geo}\big(T_{IV, j, -}(\sqrt{t}D)\big)\\
=&(-1)^{j}t\Tr_s^{\rm geo}\Big( \gamma^{-1} f_0c(df_1)\cdots c(df_{j-1})\left[D^2, f_j \right] D c(df_{j+1})\cdots c(df_{2q}) \big((B^-)^{q}(B^+)^{q+2}\big)^\gamma\Big)+ O(t)\\
=&(-1)^{j}t^{q+1}\Tr_s^{\rm geo}\Big( \gamma^{-1} f_0c(df_1)\cdots c(df_{j-1})\left[D^2, f_j \right] D c(df_{j+1})\cdots c(df_{2q}) \\
&\int_{\frac{1}{2}}^{\frac{3}{2}}\cdots \int_{\frac{1}{2}}^{\frac{3}{2}}ds_1\cdots ds_{q}\big(e^{-t(s_1+\cdots s_{q})D^2}e^{-\frac{t(q+2)}{2}D^2}\big)^\gamma\Big)+O(t)\\
=&(-1)^{j}\int_{\frac{1}{2}}^{\frac{3}{2}}\cdots \int_{\frac{1}{2}}^{\frac{3}{2}}ds_1\cdots ds_{q} t^{q+1}\Tr_s^{\rm geo}\Big( \gamma^{-1} f_0c(df_1)\cdots c(df_{j-1})\left[D^2, f_j \right] D c(df_{j+1})\cdots c(df_{2q})\\
&\qquad\qquad\qquad\big(e^{-t(s_1+\cdots s_{q}+\frac{q}{2}+1)D^2}\big)^\gamma\Big)+O(t)\\
=&(-1)^{j}\int_{1}^{2}\cdots \int_{1}^{2}ds_1\cdots ds_{q-1} t^{q+1}\Tr_s^{\rm geo}\Big( \gamma^{-1} f_0c(df_1)\cdots c(df_{j-1})\left[D^2, f_j \right] D c(df_{j+1})\cdots c(df_{2q})\\
&\qquad\qquad\qquad\big(e^{-t(s_1+\cdots s_{q}+1)D^2}\big)^\gamma\Big)+O(t)\\
=&(-1)^{j}(-i)^{\frac{n}{2}} (2 \pi )^{-\frac{a}{2}}\cdot \alpha_q \cdot \left[f_0df_1\wedge\ldots\wedge df_{2q}\right]^{(2q,0)}
\\& \wedge \left[\operatorname{det}^{\frac{1}{2}}\left(\frac{ R^{\prime} / 2}{\sinh \left( R^{\prime} / 2\right)}\right) \operatorname{det}^{-\frac{1}{2}}\left(1-\gamma|_{N^\gamma} e^{- R^{\prime \prime}}\right)\wedge \Tr(\gamma e^{-F^V}) \right]^{(a-2q, 0)}+O(t).
\end{split}
\]

\subsection{Summary and main result}Summarizing all the computations above together, we conclude with the following short time asymptotics for each type in Definition \ref{defn:4types}. 
\begin{enumerate}
\item Type I terms: 
\[
e^{-t D^2}f_0 B^\mp_1 B^\pm_2 \cdots B_{2q-1}^\mp B_{2q}^\pm.
\]
The choice of the initial term being $B^+$ or $B^-$ determines the remaining of the terms. More explicitly, we have the following two terms, 
\[
e^{-t D^2}f_0 B^+_1 B^-_2 \cdots B_{2q-1}^+ B_{2q}^-,\qquad e^{-t D^2}f_0 B^-_1 B^+_2 \cdots B_{2q-1}^- B_{2q}^+.
\]
There are in total 2 terms. Each term has the limit as $t\to 0$ equal to 
\begin{align*}
&(-i)^{\frac{n}{2}} (2 \pi )^{-\frac{a}{2}}\cdot \beta_q \cdot \left[f_0df_1\wedge\ldots\wedge df_{2q}\right]^{(2q,0)}\\
  \wedge& \left[\operatorname{det}^{\frac{1}{2}}\left(\frac{ R^{\prime} / 2}{\sinh \left( R^{\prime} / 2\right)}\right) \operatorname{det}^{-\frac{1}{2}}\left(1-\gamma|_{N^\gamma} e^{- R^{\prime \prime}}\right)\wedge \Tr(\gamma e^{-F^V}) \right]^{(a-2q, 0)}.
\end{align*}
\item Type II terms:
\[
e^{-t D^2}f_0 B^\pm_1 B^\mp_2 \cdots C_j^- \cdots B^\pm_{2q-1} B^\mp_{2q};
\]
Every $j$ contributes one such term. When $j$ runs through 1 to $2q$, there are in total $2q$ terms. And the limit of each term as $t\to 0$ equals  
\begin{align*}
 & (-i)^{\frac{n}{2}} (2 \pi )^{-\frac{a}{2}}\cdot  \left(- \frac{\beta_q}{q} + \frac{\alpha_q}{2} + \frac{\alpha_q}{q} \right) \cdot \left[f_0df_1\wedge\ldots\wedge df_{2q}\right]^{(2q,0)}
\\& \wedge \left[\operatorname{det}^{\frac{1}{2}}\left(\frac{ R^{\prime} / 2}{\sinh \left( R^{\prime} / 2\right)}\right) \operatorname{det}^{-\frac{1}{2}}\left(1-\gamma|_{N^\gamma} e^{- R^{\prime \prime}}\right)\wedge \Tr(\gamma e^{-F^V} )\right]^{(a-2q, 0)}	.
\end{align*}
\item Type III terms:
\[
e^{-t D^2}f_0 B^\pm_1 B^\mp_2 \cdots C_j^+ \cdots B^\pm_{2q-1} B^\mp_{2q};
\]
Every $j$ contributes one such term. When $j$ runs through $1$ to $2q$, there are in total $2q$ terms.  And the limit of each term as $t\to 0$ equals  

\begin{align*}
& (-i)^{\frac{n}{2}} (2 \pi )^{-\frac{a}{2}}\cdot (-\frac{\alpha_q}{2}) \cdot \left[f_0df_1\wedge\ldots\wedge df_{2q}\right]^{(2q,0)}
\\& \wedge \left[\operatorname{det}^{\frac{1}{2}}\left(\frac{ R^{\prime} / 2}{\sinh \left( R^{\prime} / 2\right)}\right) \operatorname{det}^{-\frac{1}{2}}\left(1-\gamma|_{N^\gamma} e^{- R^{\prime \prime}}\right)\wedge \Tr(\gamma e^{-F^V}) \right]^{(a-2q, 0)}	
\end{align*}

\item Type IV terms:
\[
\begin{split}
T_{IV,j,+}(\sqrt{t}D) \colon&=\gamma^{-1}  f_0 B^+_1 B^-_2 \cdots \left[   e^{-t D^2} ,f_j\right]\cdots B^+_{2q-1} B^-_{2q} \left(A^+ \right)^\gamma,\\
T_{IV,j,-}(\sqrt{t}D) \colon&=\gamma^{-1}  f_0 B^-_1  B^+_2 \cdots \left[   e^{-t D^2} ,f_j\right]\cdots B^-_{2q-1} B^+_{2q} \left(A^- \right)^\gamma.
\end{split}
\]
Every $j$ contributes one such term. When $j$ runs through 1 to $2q$ together with the choice of $+$ and $-$, there are in total $4q$ terms.
And the limit of each term as $t\to 0$ equals  

\begin{align*}
&(-1)^{j}(-i)^{\frac{n}{2}} (2 \pi )^{-\frac{a}{2}}\cdot \alpha_q \cdot \left[f_0df_1\wedge\ldots\wedge df_{2q}\right]^{(2q,0)}
\\
& \wedge \left[\operatorname{det}^{\frac{1}{2}}\left(\frac{ R^{\prime} / 2}{\sinh \left( R^{\prime} / 2\right)}\right) \operatorname{det}^{-\frac{1}{2}}\left(1-\gamma|_{N^\gamma} e^{- R^{\prime \prime}}\right)\wedge \Tr(\gamma e^{-F^V}) \right]^{(a-2q, 0)}+O(t).
\end{align*}
Adding up the contribution from Type I, II, III, IV, we conclude by Prop. \ref{prop:trPi} that
$\lim_{t \to 0^+}\Tr^{\mathrm{geo}}\!\big(\Pi(\sqrt{t}D)\big)$ is equal to 
\begin{align*}
&\left(2\beta_q+2q\big(- \frac{\beta_q}{q} + \frac{\alpha_q}{2} + \frac{\alpha_q}{q}\big)+2q(-\frac{\alpha_q}{2})+4q(\alpha_q)\right)\\
&(-1)^q(-i)^{\frac{n}{2}} (2 \pi )^{-\frac{a}{2}}\cdot \alpha_q \cdot \left[f_0df_1\wedge\ldots\wedge df_{2q}\right]^{(2q,0)}
\\
& \wedge \Tr^E\left[\operatorname{det}^{\frac{1}{2}}\left(\frac{ R^{\prime} / 2}{\sinh \left( R^{\prime} / 2\right)}\right) \operatorname{det}^{-\frac{1}{2}}\left(1-\gamma|_{N^\gamma} e^{- R^{\prime \prime}}\right)\wedge e^{-F^V} \right]^{(a-2q, 0)}+O(t)\\
=&(-1)^q(-i)^{\frac{n}{2}} (2 \pi )^{-\frac{a}{2}}\cdot (4q+2)\alpha_q \cdot \left[f_0df_1\wedge\ldots\wedge df_{2q}\right]^{(2q,0)}
\\
& \wedge \left[\operatorname{det}^{\frac{1}{2}}\left(\frac{ R^{\prime} / 2}{\sinh \left( R^{\prime} / 2\right)}\right) \operatorname{det}^{-\frac{1}{2}}\left(1-\gamma|_{N^\gamma} e^{- R^{\prime \prime}}\right)\wedge \Tr(\gamma e^{-F^V} )\right]^{(a-2q, 0)}+O(t).
\end{align*}
\end{enumerate}
It was  computed in \cite[(3.5)]{CM} that  
\[
\alpha_q = \frac{q!}{(2q+1)!}.
\]
Applying the definitions of the characteristic classes as in \cite[Prop. 3.7]{CM}, i.e. 
\[
\hat{A}(M^\gamma)=\operatorname{det}^{\frac{1}{2}}\left(\frac{ R^{\prime} / 4\pi i }{\sinh \left( R^{\prime} / 4\pi i\right)}\right) ,
\] 
we obtain the following theorem computing the index pairing on $\mathcal{A}^{c}_\Gamma(M)$. 
\begin{theorem}\label{main theorem}
Given $c \in \mathscr{C}^{2q}_\gamma(\Gamma)$, the following identity holds
\[
\begin{split}
&\langle\Ind_{c} (D_V), \Phi(\tau (c))\rangle\\
=& 2(-1)^q \frac{q!}{(2 \pi i)^q (2q)!} \int_{M^\gamma} (-i)^{\frac{n-a}{2}} \chi_\gamma(x)\Psi^\gamma(c)
  \wedge \widehat{A}(M^\gamma) \operatorname{det}^{-\frac{1}{2}}\left(1-\gamma|_{N^\gamma} e^{- \frac{R^{\perp}}{2\pi i}}\right)\wedge \Tr(\gamma e^{-\frac{F^V}{2\pi i}} ), 
\end{split}
\]
where $\chi_\gamma$ is a cutoff function of the $Z_\gamma$ action on $M^\gamma$,  $R^T$ is the curvature on $TM^\gamma$, $R^\perp$ is the curvature of $N^\gamma$, and $F^V$ is the curvature of $V$.\end{theorem}

\begin{proof}
By the commutative diagram \ref{diagram}, we have
\[
\langle\Ind_{c} (D_V), \Phi(\tau (c))\rangle=\langle\Ind_{c} (D_V), \rho(\Psi(c))\rangle.
\]
It follows from (\ref{equtrgeo}) that the right hand side equals to 
\[
\frac{1}{(2q +1)!}  \; \sum_{I = (\gamma_0, \ldots, \gamma_{2q}) \in \mathcal{I}_{2q+1}}c(\gamma_0, \cdots, \gamma_{2q})\cdot\chi_\gamma(\gamma_0)\;
 \Tr^{{\rm geo}}_a\left(\gamma^{-1} f_{0} R( \sqrt{t}D) f_{1} R( \sqrt{t}D)\cdots f_{2q}R( \sqrt{t}D)^\gamma\right)+O(t^\infty). 
 \]
where we refer to \eqref{anti-expression} for the definition of $\Tr^{{\rm geo}}_a$. 
With Lemma \ref{lem:symmetry}, Proposition \ref{prop:trPi}, and the above computation of Type I-IV limits, we obtain
\[
\begin{split}
&\sum_{I \in \mathcal{I}_{2q+1}}\chi^\Gamma_\gamma(\gamma_0)  \cdot c(\gamma_0, \cdots, \gamma_{2q})  \lim_{t\to 0}  \Tr^{{\rm geo}}_a\left(\gamma^{-1} f_0 R( \sqrt{t}D) f_1 R( \sqrt{t}D)\cdots f_{2q}R( \sqrt{t}D)^\gamma\right) \\
=&\sum_{I \in \mathcal{I}_{2q+1}}\chi^\Gamma_\gamma(\gamma_0)  \cdot c(\gamma_0, \cdots, \gamma_{2q}) 2(-1)^q \frac{q!}{(2 \pi i)^q (2q)!} \int_{M^\gamma} (-i)^{\frac{n-a}{2}} f_0df_1\cdots df_{2q}\\
&  \qquad \wedge \widehat{A}(M^\gamma) \operatorname{det}^{-\frac{1}{2}}\left(1-\gamma|_{N^\gamma} e^{- \frac{R^{\perp}}{2\pi i}}\right)\wedge \Tr(\gamma e^{-\frac{F^V}{2\pi i}} )\\
=&2(-1)^q \frac{q!}{(2 \pi i)^q (2q)!} \int_{M^\gamma} (-i)^{\frac{n-a}{2}}\sum_{I \in \mathcal{I}_{2q+1}}\chi^\Gamma_\gamma(\gamma_0)  \cdot c(\gamma_0, \cdots, \gamma_{2q})f_0df_1\cdots df_{2q}\\
&\qquad  \wedge \widehat{A}(M^\gamma) \operatorname{det}^{-\frac{1}{2}}\left(1-\gamma|_{N^\gamma} e^{- \frac{R^{\perp}}{2\pi i}}\right)\wedge \Tr(\gamma e^{-\frac{F^V}{2\pi i}} ).
\end{split}
\]

Notice that in Proposition \ref{prop:finitesum},  the finite set $\mathcal{I}_{2q+1}$  was chosen
so that any term $\chi(\gamma_0^{-1}y_0)\otimes \cdots \otimes \chi(\gamma_{2q}^{-1}y_{2q})$ not from the finite set $\mathcal{I}_{2q+1}$ are functions supported away from the $\gamma$-diagonal in $M^{\times (2q+1)}$. Hence, for those $(\gamma_0, \cdots, \gamma_{2q})$ not in $\mathcal{I}_{2q+1}$, 
\[
\chi(\gamma_0^{-1}x)d\chi(\gamma_1^{-1}x)\cdots d\chi(\gamma_{2q}^{-1}x)=0,\ \forall x\in M^\gamma.
\]
It follows from this observation that we have the following equation,
\[
\begin{split}
&\sum_{I \in \mathcal{I}_{2q+1}}\chi^\Gamma_\gamma(\gamma_0)  \cdot c(\gamma_0, \cdots, \gamma_{2q})\chi(\gamma_0^{-1}x)d\chi(\gamma_1^{-1}x)\cdots d\chi(\gamma_{2q}^{-1}x)\\
=&\sum_{}\chi^\Gamma_\gamma(\gamma_0)  \cdot c(\gamma_0, \cdots, \gamma_{2q})\chi(\gamma_0^{-1}x)d\chi(\gamma_1^{-1}x)\cdots d\chi(\gamma_{2q}^{-1}x), 
\end{split}
\]
where the summation ranges over all $\gamma_0, \cdots, \gamma_{2q}$. 

Hence, 
\[
\begin{split}
&\lim_{t\to 0}\rho \circ \Psi(c)(R( \sqrt{t}D),\ldots,R( \sqrt{t}D))\\
=&2(-1)^q \frac{q!}{(2 \pi i)^q (2q)!} \int_{M^\gamma} (-i)^{\frac{n-a}{2}}\left(\sum\chi^\Gamma_\gamma(\gamma_0)  \cdot c(\gamma_0, \cdots, \gamma_{2q})\gamma_0^*\chi \gamma_1^*d\chi \cdots \gamma_{2q}^*d\chi \right)\\
&\qquad  \wedge \widehat{A}(M^\gamma) \operatorname{det}^{-\frac{1}{2}}\left(1-\gamma|_{N^\gamma} e^{- \frac{R^{\perp}}{2\pi i}}\right)\wedge \Tr(\gamma e^{-\frac{F^V}{2\pi i}} )
\end{split}
\]
Recall that $\chi_\gamma$ is the cut-off function for the $Z_\gamma$-action on $M^\gamma$. For any $x \in M^\gamma$, we can insert 
\[
\sum_{\eta \in Z_\gamma} \chi_\gamma(\eta x)  = 1
\] 
as follow 
\[
\begin{split}
& \int_{M^\gamma}(-i)^{\frac{n-a}{2}} \sum_{\gamma_0, \cdots, \gamma_{2q}}\chi^\Gamma_\gamma(\gamma_0)  \cdot c(\gamma_0, \cdots, \gamma_{2q})\cdot \gamma_0^*\chi \gamma_1^*d\chi \cdots \gamma_{2q}^*d\chi\\
&\hskip 2cm \wedge \widehat{A}(M^\gamma) \operatorname{det}^{-\frac{1}{2}}\left(1-\gamma|_{N^\gamma} e^{- \frac{R^{\perp}}{2\pi i}}\right)\wedge \Tr(\gamma e^{-\frac{F^V}{2\pi i}} )\\
=&\int_{M^\gamma}(-i)^{\frac{n-a}{2}}\sum_{\gamma_0, \cdots, \gamma_{2q}} \big(\sum_{\eta \in Z_\gamma} \chi_\gamma(\eta x) \big) \chi^\Gamma_\gamma(\gamma_0)  \cdot c(\gamma_0, \cdots, \gamma_{2q})\cdot \gamma_0^*\chi \gamma_1^*d\chi \cdots \gamma_{2q}^*d\chi\\
&\hskip 2cm \wedge \widehat{A}(M^\gamma) \operatorname{det}^{-\frac{1}{2}}\left(1-\gamma|_{N^\gamma} e^{- \frac{R^{\perp}}{2\pi i}}\right)\wedge \Tr(\gamma e^{-\frac{F^V}{2\pi i}} ).
\end{split}
\]
By the change of variable $ x \mapsto \eta^{-1}x$ and $Z_\gamma$-invariance of $c$, we can rewrite the above integral as
\[
\begin{split}
& \int_{M^\gamma} (-i)^{\frac{n-a}{2}}\sum_{\gamma_0, \cdots, \gamma_{2q}}\chi_\gamma(x)\sum_{\eta \in Z_\gamma} \chi^\Gamma_\gamma(\gamma_0)  \cdot c(\gamma_0, \cdots, \gamma_{2q})\cdot\chi(\gamma_0^{-1}\eta^{-1}x)(\eta\gamma_1)^*d\chi \cdots (\eta\gamma_{2q})^*d\chi\\
&\hskip 2cm \wedge  \widehat{A}(M^\gamma) \operatorname{det}^{-\frac{1}{2}}\left(1-\gamma|_{N^\gamma} e^{- \frac{R^{\perp}}{2\pi i}}\right)\wedge \Tr(\gamma e^{-\frac{F^V}{2\pi i}} )\\
=& \int_{M^\gamma} (-i)^{\frac{n-a}{2}}\sum_{\gamma_0, \cdots, \gamma_{2q}} \chi_\gamma(x)\sum_{\eta \in Z_\gamma} \chi^\Gamma_\gamma(\gamma_0)  \cdot c(\eta \gamma_0, \cdots, \eta\gamma_{2q})\cdot\chi(\gamma_0^{-1}\eta^{-1}x)(\eta\gamma_1)^*d\chi \cdots (\eta\gamma_{2q})^*d\chi\\
&\hskip 2cm \wedge  \widehat{A}(M^\gamma) \operatorname{det}^{-\frac{1}{2}}\left(1-\gamma|_{N^\gamma} e^{- \frac{R^{\perp}}{2\pi i}}\right)\wedge \Tr(\gamma e^{-\frac{F^V}{2\pi i}} ).
\end{split}
\]
Set $\tilde{\gamma}_0=\eta \gamma_0$, $\tilde{\gamma}_1=\eta \gamma_1$, ..., $\tilde{\gamma}_k=\eta\gamma_k$. The above integral can be written as 
\[
\begin{split}
 &\int_{M^\gamma} (-i)^{\frac{n-a}{2}}\sum_{\tilde{\gamma}_0, \cdots, \tilde{\gamma}_{2q}}\chi_\gamma(x)\sum_{\eta \in Z_\gamma} \chi^\Gamma_\gamma(\eta^{-1}\tilde{\gamma}_0)  \cdot c(\tilde{\gamma}_0, \cdots, \tilde{\gamma}_{2q})\cdot\chi(\tilde{\gamma}_0^{-1} x)\tilde{\gamma}_1^*d\chi \cdots \tilde{\gamma}_{2q}^*d\chi\\ 
&\hskip 2cm \wedge \widehat{A}(M^\gamma) \operatorname{det}^{-\frac{1}{2}}\left(1-\gamma|_{N^\gamma} e^{- \frac{R^{\perp}}{2\pi i}}\right)\wedge \Tr(\gamma e^{-\frac{F^V}{2\pi i}} ).
\end{split}
\]
As $\chi^\Gamma_\gamma$ is the cutoff function of the $Z_\gamma$-action on $\Gamma$, we have 
\[
\sum_{\eta \in Z_\gamma} \chi^\Gamma_\gamma(\eta^{-1}\tilde{\gamma}_0)  =1. 
\]
Hence, we can conclude 
\[
\begin{split}
 &\int_{M^\gamma} (-i)^{\frac{n-a}{2}}\sum_{\tilde{\gamma}_0, \cdots, \tilde{\gamma}_{2q}}\chi_\gamma(x)\sum_{\eta \in Z_\gamma} \chi^\Gamma_\gamma(\eta^{-1}\tilde{\gamma}_0)  \cdot c(\tilde{\gamma}_0, \cdots, \tilde{\gamma}_{2q})\cdot\chi(\tilde{\gamma}_0^{-1} x)\tilde{\gamma}_1^*d\chi \cdots \tilde{\gamma}_{2q}^*d\chi\\ 
&\hskip 2cm \wedge \widehat{A}(M^\gamma) \operatorname{det}^{-\frac{1}{2}}\left(1-\gamma|_{N^\gamma} e^{- \frac{R^{\perp}}{2\pi i}}\right)\wedge \Tr(\gamma e^{-\frac{F^V}{2\pi i}} )\\
=& \int_{M^\gamma} (-i)^{\frac{n-a}{2}}\chi_\gamma(x) \sum_{\tilde{\gamma}_0, \cdots, \tilde{\gamma}_{2q}} c(\tilde{\gamma}_0, \cdots, \tilde{\gamma}_{2q})\cdot\chi(\tilde{\gamma}_0^{-1} x)\tilde{\gamma}_1^*d\chi \cdots \tilde{\gamma}_{2q}^*d\chi\\
&\hskip 2cm \wedge \widehat{A}(M^\gamma) \operatorname{det}^{-\frac{1}{2}}\left(1-\gamma|_{N^\gamma} e^{- \frac{R^{\perp}}{2\pi i}}\right)\wedge \Tr(\gamma e^{-\frac{F^V}{2\pi i}} ).
\end{split}
\]
By (\ref{eq:lambdagammapsi}), we have 
\[
\Psi^{\gamma}(c)=\sum_{\tilde{\gamma}_0, \cdots, \tilde{\gamma}_{2q}} c(\tilde{\gamma}_0, \cdots, \tilde{\gamma}_{2q})\cdot\chi(\tilde{\gamma}_0^{-1} x)\tilde{\gamma}_1^*d\chi \cdots \tilde{\gamma}_{2q}^*d\chi. 
\]
And we conclude with the final formula
\[
\begin{split}
&\lim_{t\to 0}\rho \circ \Psi(c)(R( \sqrt{t}D),\ldots,R( \sqrt{t}D))\\
=&2(-1)^q \frac{q!}{(2 \pi i)^q (2q)!} \int_{M^\gamma} (-i)^{\frac{n-a}{2}}\chi_\gamma(x)\Psi^\gamma(c)  \wedge \widehat{A}(M^\gamma) \operatorname{det}^{-\frac{1}{2}}\left(1-\gamma|_{N^\gamma} e^{- \frac{R^{\perp}}{2\pi i}}\right)\wedge \Tr(\gamma e^{-\frac{F^V}{2\pi i}} ).
\end{split}
\]
\end{proof}
\begin{remark}We would like to point out that the extra factor 2 in the formula of $\langle\Phi(\tau (c)), \Ind_{c}(D)\rangle$ comes from the fact that we have followed the construction of \cite{moscovici-wu} to consider the double $V_{{\rm CM}} (tD)\oplus (V_{{\rm CM}} (tD))^*$. 
\end{remark}

\appendix

\section{The Volterra calculus}\label{appendix-section:basic-volterra}
The role of this Appendix  is to recall briefly the basic definitions and results around the theory
of Volterra pseudodifferential operators (briefly, the Volterra calculus). We also give 
the easy extension of the theory to the equivariant context. The articles of Ponge \cite{Ponge} and Ponge-Wang \cite{Ponge-Wang-2}
are a nice introduction to this theory, initiated by Piriou [Pi], Greiner [Gr] and Beals-Greiner-Stanton \cite{BGS}. The treatment of Getzler rescaling within the Volterra
calculus and its use in index theory is due to Ponge \cite{Ponge} for the classic Atiyah-Singer index theorem and to Ponge and Wang \cite{Ponge-Wang-2}  for the Atiyah-Segal-Singer equivariant index theorem. 
We refer to   these two references  for more
on the  material that follows. 

\subsection{Basic definitions and results on the Volterra calculus.}\label{appendix-subsection:basic-volterra}

\noindent
Some of the defining features of the Volterra calculus come in fact from the 
properties  of $(\partial_t + D^2)^{-1}$. First notice that 
$$(\partial_t + D^2)^{-1} u (s)= \int_0^{+\infty} e^{-t D^2} u(s-t) dt , \quad u\in C_+^\infty (M\times\RR, E).$$
If $u$ is supported in $M\times [c,+\infty)$ then 
$$(\partial_t + D^2)^{-1} u (s)=\int_{\{0\leq t\leq s-c\}} e^{-t D^2} u(s-t) dt$$
The Schwartz kernel of $(\partial_t + D^2)^{-1}$, viz. 
$k_{(\partial_t + D^2)^{-1}}(x,s,y,s')$, can be explicitly described in terms of the heat kernel $k_t (x,y)$; indeed, 
\begin{equation}\label{heat-volterra}
k_{(\partial_t + D^2)^{-1}}(x,s,y,s')= \begin{cases} k_{s-s'} (x,y) \quad \text{if} \quad s-s'>0\\
0 \quad \text{if} \quad s-s'<0\end{cases}
\end{equation}
The operator $(\partial_t + D^2)^{-1}$ thus  satisfies the {\it Volterra property}, namely:\\(i) time translation 
invariance;\\
(ii) causality principle: if $u=0$ on $M\times (-\infty,t_0]$ then the same will be true for the
transformed section.

\medskip
\noindent
A continuous operator $Q: C^\infty_c (M\times \RR,E)\to C^\infty (M\times \RR,E)$
satisfying the Volterra property has a Schwartz kernel $k_Q (x,s,y,s')=K_Q (x,y,s-s')$
for some $K_Q (x,y,t)\in C^\infty (M,E)\otimes \mathcal{D}' (M\times \mathbb{R},E)$
with $K_Q (x,y,t)=0$ for $t<0$. The distribution $K_Q (x,y,t)$ is the {\it Volterra kernel}
of $Q$.

\medskip
\noindent
The Volterra property will be a defining feature of Volterra pseudodifferential operators. In order to introduce them we first describe briefly the local theory.

\begin{definition}
  A distribution $G \in \mathcal{S}' (\RR^{n+1})$ is said to be (parabolic) homogeneous of	order $m$, $m\in\mathbb{Z}$, when
 \begin{equation}
 \label{homo}
 G_\lambda = \lambda^m G  	
 \end{equation}
 with $G_\lambda$ defined by 
 $$\langle G_\lambda, u (x,t)\rangle = |\lambda|^{-(n+2)} \langle G, u (\lambda^{-1} x,\lambda^{-2} t)\rangle $$
\end{definition}

\noindent
 If $q\in C^\infty (\mathbb{R}^n\times \mathbb{R}\setminus 0)$ is parabolic homogeneous of order $m$  and, in addition, 
 
 \medskip
  (*) it extends to a continuous function on 
 $\mathbb{R}^n\times \overline{\mathbb{C}}_- \setminus 0$, $\mathbb{C}_- := \{{\rm Im} z <0\}$, in such way to be holomorphic 
 
 with respect to the variable $z\in \mathbb{C}_-$,

 \medskip
 \noindent
 then $q$ admits a {\rm unique} extension  $G\in \mathcal{S}'(\RR^{n+1})$ which is 
 is parabolic homogeneous of order $m$ and such that the inverse Fourier transform $G^\vee$
 vanishes for $t<0$. In the sequel, with small abuse of notation, we shall still denote by 
 $q^\vee$ this inverse Fourier transform.

We now fix an  open neighborhood $U$ in $\mathbb{R}^n$. 
\begin{definition}
$S_{\mathrm{v}}^m\left(U \times \mathbb{R}^{n+1}\right), m \in \mathbb{Z}$, consists of smooth functions $q(x, \xi, \tau)$ on $U \times$ $\mathbb{R}^n \times \mathbb{R}$ with an asymptotic expansion $q \sim \sum_{j \geq 0} q_{m-j}$, where:
\begin{itemize}
	\item $q_k \in C^{\infty}\left(U \times\left[\left(\mathbb{R}^n \times \mathbb{R}\right) \backslash 0\right]\right)$ is a homogeneous Volterra symbol of order k; this means that 
	$q_k$ is parabolic homogeneous of degree $k$ in the last $n+1$ variables and, in addition, it satisfies 
	the extension property (*);
		\item The sign $\sim$ means that, for any integer $N$ and any compact $K \subset U$, there is a constant $C_{N K \alpha \beta k}>0$ such that for $x \in K$ and for $|\xi|+|t|^{\frac{1}{2}}>1$ we have
\begin{align}\label{eq:symbol-growth}
\left|\partial_x^\alpha \partial_{\xi}^\beta \partial_t^k\left(q-\sum_{j<N} q_{m-j}\right)(x, \xi, t)\right| \leq C_{N K \alpha \beta k}\left(|\xi|+|t|^{1 / 2}\right)^{m-N-|\beta|-2 k}
\end{align}
 \end{itemize}
 \end{definition}

\begin{definition}
 $\Psi_{\mathrm{v}}^m(U \times \mathbb{R}), m \in \mathbb{Z}$, consists of continuous operators $Q$ from $C_c^{\infty}\left(U_x \times \mathbb{R}_t\right)$ to $C^{\infty}\left(U_x \times \mathbb{R}_t\right)$ such that 
  \begin{itemize}
 	\item $Q$ has the Volterra property (thus, in particular,  its  distributional Volterra kernel
	 $K_Q(x, y, t)$ vanishes for $t<0$);   	\item We have $Q=q\left(x, D_x, D_t\right)+R$ for some symbol $q$ in $S_{\mathrm{v}}^m(U \times \mathbb{R})$ and some (Volterra) smoothing operator $R$.  Here  $q\left(x, D_x, D_t\right)$ is the operator defined by
	 the Schwartz kernel $q^{\vee} (x, x-y, t)$. \end{itemize}
\end{definition}

\smallskip
\noindent
Any operator $Q\in \Psi_{\mathrm{v}}^m(U \times \mathbb{R})$ has a unique {\it Volterra kernel}
$K_Q(x,y,t)\in C^\infty (U,\mathcal{D}' (U\times\mathbb{R}))$ such that
$$Qu (x,s)= \langle K_Q (x,y,s-s'), u(y, s')\rangle$$
In fact, up to the smoothing operator $R$,
$$K_Q(x,y,t) = q (x, x-y, t)^\vee $$
where we are taking the inverse Fourier transform of $q$ with respect to the last $n + 1$ variables.

\smallskip
\noindent
Let $q_m(x, \xi, t) \in C^\infty(U \times  \RR^{n+1} \setminus 0)$ be a homogeneous Volterra symbol of order $m$. Then the operator
\[
Q = q_m(x, D_x, D_t)
\]
 defined by the distributional kernel $q^{\vee}_m (x, x-y, t)$ is indeed a Volterra pseudodifferential operator with symbol $\sim$ $q_m$. See \cite{Ponge-Wang-2} Lemma 2.16 for the classic argument involving a smoothing of $q_m$ at the origin.

\smallskip
\noindent
A Volterra pseudodifferential operator $Q$ is properly supported if its Schwartz kernel $k_Q$
is  properly supported in the usual sense; for the associated Volterra kernel $K_Q$ this means 
that $K_Q$ is compactly supported in the $t$-variable. We also remark that a smoothing Volterra operator, being smooth
and vanishing for $t<0$,
is $O(t^\infty)$ as $t\downarrow 0^+$ in $C^\infty (U\times U)$ uniformly on compact sets (this is a simple Taylor expansion argument, see for example the proof of Lemma 3.2  in \cite{Ponge-Wang-2}).

\smallskip
\noindent
For Volterra pseudodifferential operators we have all the expected basic results:

(i) pseudolocality

(ii) proper support modulo smoothing operators

(iii) composition formula for properly supported operators

(iv) asymptotic completness

(v) existence of a parametrix: {\em $Q\in \Psi^m_{\mathrm{v}} (U\times \mathbb{R})$ admits a parametrix if and only if its principal symbol is 

nowhere vanishing in $U\times \mathbb{R}^n \times \overline{\mathbb{C}}_-\setminus 0$.}

(vi) diffeomorphism invariance.\\
See Prop. 2.17 in \cite{Ponge-Wang-2} and references therein.
The last property allows for a globalization to manifolds.

A  Volterra pseudodifferential operator 
 $Q\in \Psi_{{\rm v}}^{m} (U)$ with symbol $q\sim \sum_{j\geq 0} q_{m-j}$ is such that 
 its Volterra kernel $K_Q (x,y,t)$ admits the following asymptotic expansion:\\ {\em $\forall N\in\mathbb{N}$ $\exists J$ such that 
 $$K_Q(x,y,t)=\sum_{j\leq J} q^\vee_{m-j} (x, x-y, t) \;\;\;{\rm mod}\;\;\;C^N (U\times U \times \mathbb{R}) \;\;\forall t>0$$}\\
 See Proposition 2.19 in \cite{Ponge-Wang-2}.
 This allows for the following asymptotic expansion in $C^\infty (U)$: for $Q$ as above we have 
 $$K_Q (x,x,t)\sim t^{-\frac{n}{2}+[\frac{m}{2}]+1} \sum_{\ell\geq 0} t^\ell q^\vee_{2[\frac{m}{2}] -2\ell} (x,0, 1)\,.$$
 See Lemma 2 in \cite{Ponge}. More refined properties on asymptotic expansions near $t=0$
 will be given below.
 
The operator $\partial_t + P$, with $P$ an elliptic  differential operator
of order 2  with positive principal symbol, admits a parametrix $Q\in \Psi_{{\rm v}}^{-2} (U)$.
See Example 2.13 in \cite{Ponge-Wang-2}.
Similarly, if $M$ is a smooth compact manifold and $D$ is a Dirac operator acting on the sections of a bundle of Clifford modules $E$, then the Volterra operator $\partial_t + D^2$ 
admits a parametrix $Q\in \Psi_{{\rm v}}^{-2} (M,E)$. Using this fundamental fact and
the composition formula one can prove that  
\begin{equation}\label{inverse}
(\partial_t + D^2)^{-1}\in \Psi_{{\rm v}}^{-2} (M\times\mathbb{R},E).
\end{equation}
Moreover, for $t>0$ we have that the Volterra kernel $K_{(\partial_t + D^2)^{-1}} (x,y,\cdot)$
computed at $t$ is equal to the heat kernel $k_t (x,y)$. 
Let us elaborate  further on \eqref{inverse}, following \cite{BGS}. We know that there exists $Q\in \Psi_{\mathrm{v}}^{-2}(M \times \mathbb{R},E)$ such that
\begin{equation}\label{parametrix-volterra-0}
(\partial_t + D^2)Q=I-R_1\,\quad Q(\partial_t + D^2)=I-R_2\,\quad \quad R_j\in  \Psi_{\mathrm{v}}^{-\infty}(M \times \mathbb{R},E).
\end{equation}
For simplicity, and following \cite{BGS}, we denote by $\widetilde{Q}$ the operator 
$(\partial_t + D^2)^{-1}$. Then using  \eqref{parametrix-volterra-0}  we have $$\widetilde{Q}= Q + \widetilde{Q}R_1\,,\quad\quad \widetilde{Q}= Q +  R_2 \widetilde{Q}.$$
These two equations imply together that 
  $\widetilde{Q} - Q$, that is $(\partial_t + D^2)^{-1}-Q$, is a smoothing Volterra operator.
  See \cite{BGS}, proof of  Theorem (5.16). This establishes \eqref{inverse}.

\subsection{Equivariant Volterra calculus}\label{subsection:volterra-equivariant}
We now consider a smooth manifold $M$ endowed with a cocompact proper action of $\Gamma$, a $\Gamma$-invariant Riemannian metric $g$ and an equivariant bundle $E$ of  
Clifford modules. We extend the action to $M\times\mathbb{R}$ by letting $\Gamma$ act
trivially on $\mathbb{R}$. Following the classic case of pseudodifferential operators
on $\Gamma$-proper manifolds, we can define  $\Psi_{\Gamma,\mathrm{v},c}^m(M \times \mathbb{R})$, $m \in \mathbb{Z}$, the equivariant Volterra operators $Q$  of order $m$, 
that have kernels $k_Q (x,s,y,s')$ that are of compact support in $M \times \mathbb{R} \times M\times \mathbb{R}/\Gamma$,
where $\gamma\cdot (x,s,y,s')=(\gamma x, s, \gamma y, s')$.
  Notice that  the corresponding Volterra kernels $K_Q (x,y,t)$ will be  compactly supported 
in time. The equivariant Volterra calculus of $\Gamma$-compact support can be developed 
in the usual fashion and it will satisfy the expected properties. As this is standard, we shall not enter into the details. In particular 
$\Psi_{\Gamma,\mathrm{v},c}^*(M \times \mathbb{R})$ is a graded algebra. If $D$ 
is a $\Gamma$-equivariant Dirac operator then there exists a  parametrix $Q$ for $\partial_t + D^2$. More precisely:\\
{\em there exists $Q\in \Psi_{\Gamma,\mathrm{v},c}^{-2}(M \times \mathbb{R},E)$ such that
\begin{equation}\label{parametrix-volterra}
(\partial_t + D^2)Q=I-R_1\,\quad Q(\partial_t + D^2)=I-R_2\,\quad \quad R_j\in  \Psi_{\Gamma,\mathrm{v},c}^{-\infty}(M \times \mathbb{R},E).
\end{equation}
}
Notice that the parametrix $Q$ and the remainders
$R_1$, $R_2$ are of $\Gamma$-compact support. As in the compact case 
we can use this parametrix in order to study the Volterra kernel $K_{(\partial_t + D^2)^{-1}}(x,y,t)$ for $t\downarrow 0^+$. 

\begin{proposition}
Consider the Volterra kernel $K_{(\partial_t + D^2)^{-1}}(x,y,t)$ and let
 $Q$ be a parametrix of $\Gamma$-compact support for 
 $\partial_t + D^2$. Then $(\pa_t + D^2)^{-1}-Q$ is $O(t^N)$, for $t\downarrow 0^+$, in $C^\infty$ seminorms uniformy on compact
sets $K\subset M\times M$.
\end{proposition}

\begin{proof}
Formula \eqref{heat-volterra} describes precisely the kernel
$K_{(\partial_t + D^2)^{-1}}(x,y,t)$
for $t>0$ in term of the heat kernel.
As in the previous subsection we denote  by $\widetilde{Q}$ the operator 
$(\partial_t + D^2)^{-1}$ and we also use this symbol for its Volterra kernel.
Using  \eqref{parametrix-volterra}  we have again
$\widetilde{Q}= Q + \widetilde{Q}R_1$ and  $\widetilde{Q}= Q +  R_2 \widetilde{Q}$
so that
$$\widetilde{Q}= Q +R_2 Q + R_2 \widetilde{Q}R_1\,.$$
We  observe 
that the smooth Volterra kernel $R_2 \widetilde{Q}R_1$ is well defined, given \eqref{heat-volterra}, the rapid exponential decay of the heat kernel and the fact that 
$R_j$ is of $\Gamma$-compact support.
Observe next that $R_2 Q$ is a smoothing Volterra operator of $\Gamma$-compact support; 
thus the corresponding Volterra kernel will be $O(t^N)$ in $C^\infty$ seminorms, as $t\downarrow 0^+$, uniformly on compact sets $K\subset M\times M$.
Next notice that $R_2 \widetilde{Q}R_1$ defines a $\Gamma$-equivariant Volterra smoothing kernel 
in $M\times M\times \mathbb{R}$ (but of course not of $\Gamma$-compact support).
Indeed the only potential singularity is at $t=0$ but the composition in the time variable is by convolution and we know
that the convolution of a distribution (in this case, a distribution with singular support only in one point ($t=0$))
with a smooth function is a smooth function. We observe that $R_2 \widetilde{Q}R_1$
is vanishing for $t \leq 0$. This follows from the fact that the support of the convolution of a distribution
and a function is contained in the Minkowski sum of the two supports.
We can therefore conclude that 
$R_2 \widetilde{Q}R_1$ is smooth in $M \times M \times \RR$ and vanishing for $t\leq 0$
and thus, by the usual Taylor expansion argument, it is $O(t^N)$, for $t\downarrow 0^+$, in $C^\infty$ seminorms uniformy on compact
sets $K\subset M\times M$, as required.
\end{proof}
The above Proposition  will allow us to use the  parametrix $Q$ instead of the global operator $(\pa_t + D^2)^{-1}$
for many questions having to do with asymptotic expansions near $t=0$, $t>0$.\\

\subsection{Fixed point set and Getzler-Volterra rescaling}\label{subsection:getzler-volterra}
Let $Q\in  \Psi_{\Gamma,\mathrm{v},c}^{m}(M \times \mathbb{R},E)$. We fix $\gamma \in \Gamma$ and denote by $M^\gamma$ the fixed point set of $\gamma$ action. We decompose
\[
M^\gamma = \bigsqcup_{0\leq a \leq n} M^\gamma_a,
\]
where $a$ is the dimensional of the connected components. For any $x \in M^\gamma_a$, let $N^\gamma$ be the normal bundle of $M^\gamma \hookrightarrow M$ and $N^\gamma(\epsilon)$ the  the ball bundle of $N^\gamma$ of radius $\epsilon$ around the zero-section. Then we have that 
\[
K_Q\left(\exp_x v, \exp_x \left( \gamma v \right), t\right) = K_Q\left(\exp_x v, \gamma \exp_x v, t\right) 
\]
For $x \in M^\gamma$ and $t>0$ set
\begin{align}
I_Q(x, t, \gamma ):= \int_{N_x^\gamma(\epsilon)} \gamma^{-1}\cdot K_Q\left(\exp _x v, \exp _x\left(\gamma v\right), t\right)|d v|
\end{align}
This defines a smooth section of $\text{End}( E)$ over $M^\gamma \times(0, \infty)$.\\
More generally we shall be concerned with operators of the following type: $PQ$ with $P$ a differential operator
(not necessarily $\Gamma$-equivariant) and $Q\in  \Psi_{\Gamma,\mathrm{v}}^{m}(M \times \mathbb{R},E)$. We can still consider 
\begin{align}
I_{PQ}(x, t, \gamma ):= \int_{N_x^\gamma(\epsilon)} \gamma^{-1}\cdot  K_{PQ} \left(\exp _x v, \exp _x\left(\gamma v\right), t\right)|d v|
\end{align}

\noindent

Recall Definition \ref{def: Volterra-related} (Volterra-related operators) 
and Definition \ref{def: exponential-control-no-appendix} (operators of exponential control).

\noindent

For the following Lemma we remark that if $Q$ is Volterra-related, so is $PQ$ if $P$
is a compactly supported differential operator.

\begin{lemma}\label{lemma:first-reduction}
Assume that $Q$ is Volterra-related and of exponential control. Assume that $P$ is compactly supported, for example $P$ is a 0-th order differential operator 
of compact support.
Then, as $t \rightarrow 0^{+}$, we have
\begin{align}
\int_M \operatorname{str}\left[\gamma^{-1}\cdot  K_{PQ}(x, \gamma x, t)\right]|d x|=\int_{M^\gamma} \operatorname{str}\left[I_{PQ}(x, t, \gamma )\right]|d x|+\mathrm{O}\left(t^{\infty}\right).
\end{align}
\begin{proof}
The proof given in  \cite[Lemma 3.1]{Ponge-Wang-2} can be adapted easily to the present case, given that $P$ is of compact support.
\end{proof}
\end{lemma}

For the local higher index theorem we shall need to determine the asymptotic expansion of $I_{PQ} (x,t, \gamma)$. To this end, always because of the compact support assumption on $P$, we go back to the local theory.
Given a fixed-point $x$ in a submanifold component $M^\gamma_a$, consider some local coordinates $x = (x_1, \cdots x_a)$ and define fiber coordinates $v = (v_1, \cdots v_b)\in N^\gamma_x$. Then we get local coordinates $x_1, \cdots x_a, v_1, \cdots v_b$. We shall refer to this type of coordinates as \emph{tubular coordinates}
and we denote by $V_T$ the open set on which these coordinates are defined. Notice that $V_T$ can be chosen to be  $\cup_{p\in V} N^\gamma_p (\epsilon)$, with $V\subset M^\gamma_a$ and $V$ a chart on 
$M^\gamma_a$ centered in $x$ and with image $U$ in $\mathbb{R}^a$ and $x$ corresponding to the origin. Let $Q$ be Volterra-related.
Then the symbol of the Volterra operator associated to $Q$ in tubular coordinates will have an asymptotic expansion 
\[
q(x, v, \xi, \nu; t) \sim \sum_{j \geq 0} q_{m-j}(x, v, \xi, \nu; t). 
\]

\begin{theorem}(\cite[Proposition 3.4]{Ponge-Wang-2} Asymptotic expansion )\label{thm:PW-local}
As $t \rightarrow 0^{+}$, uniformly on compact subsets of $V=V_T \cap M^\gamma_a$, we have 
\[
I_Q(x, t, \gamma )\sim  \sum_{j \geq 0} t^{\left(-\frac{a}{2}-\lfloor \frac{m}{2}\rfloor+1\right) + j}\cdot  I_{Q, j}(x, \gamma)
\]
where 
\begin{equation}
\label{eq Ij}
I_{Q, j}(x, \gamma) \colon = \sum_{|\alpha| \leq m - 2\lfloor \frac{m}{2}\rfloor + 2j }\int_{\mathbb{R}^{n-a}} \frac{\nu^\alpha}{\alpha!} \left(\partial_\nu^\alpha q_{2\lfloor \frac{m}{2}\rfloor - 2j + \alpha}\right)^\vee\left(x, 0; 0, (1 - \gamma) \nu ; 1\right)\; d\nu	\in \mathrm{End}(\mathcal{S}\otimes E). 
\end{equation}
\end{theorem}

We now pass to Gezler rescaling in the context of Volterra calculus. Let $\sigma \colon \text{Cliff}(TM) \to \Lambda^\bullet T_\CC M$ be the symbol map. We assign
\begin{equation}
\label{getzler order}
\text{deg} \partial_j= \frac{1}{2}\text{deg} \partial_t  = \text{deg} c(dx_j) = - \text{deg} x_j = 1. 
\end{equation}
We fix $x\in M$ and a coordinate chart centered at $x$ and with image $\mathbb{B}$, with $\mathbb{B}$ a ball in $\mathbb{R}^n$ centred at the origin. The bundle of Clifford modules $E$ restricted to $U$ can be described as the spinor bundle $\mathcal{S}$ of  $U$ tensor an auxiliary bundle of rank $d$ on $U$. Localizing a global Volterra pseudodifferential operator to $U$ we are led to consider a Volterra pseudodifferential operator 
$Q \in \Psi^*_{\mathrm{v}}(\mathbb{B} \times \mathbb{R},S_n\otimes \mathbb{C}^d)$
with $S_n$ the spinor bundle of $\mathbb{R}^n$. Consider its  symbol 
$q \sim \sum_{k \leq m} q_{k}$.  
Then taking components in each subspace $\Lambda^j T_{\mathbb{C}}^* \mathbb{R}^n$ and then using Taylor expansions at $x=0$ gives the formal expansions
$$
\sigma[q(x, \xi, t)] \sim \sum_{j, k} \sigma\left[q_k(x, \xi, t)\right]^{(j)} \sim \sum_{j, k, \alpha} \frac{x^\alpha}{\alpha!} \sigma\left[\partial_x^\alpha q_k(0, \xi, t)\right]^{(j)}
$$
According to (\ref{getzler order}) the symbol $\frac{x^\alpha}{\alpha!} \partial_x^\alpha \sigma\left[q_k(0, \xi, \tau)\right]^{(j)}$ is Getzler homogeneous of degree $k+j-|\alpha|$. So we can expand $\sigma[q(x, \xi, \tau)]$ as
$$
\sigma[q(x, \xi, \tau)] \sim \sum_{j \geq 0} q_{(m-j)}(x, \xi, \tau), \quad q_{(m)} \neq 0,
$$
where $q_{(m-j)}$ is a Getzler homogeneous symbol of degree $m-j$.
\begin{definition}
\label{getzler def}
We set-up the following definitions:
\begin{itemize}
	\item The integer $m$ is the Getzler order of $Q$ at the origin;
	\item  The symbol $q_{(m)}$ is the principal Getzler homogeneous symbol of $Q$ at the origin,
	\item  The operator $Q_{(m)}=q_{(m)}\left(x, D_x, D_t\right)$ is the model operator of $Q$  at the origin.
\end{itemize}	
\end{definition}

We will use the following three theorems to compute the local higher index pairing. 

\begin{theorem}(\cite[Lemma 4.15]{Ponge-Wang-2} Asymptotic expansion using model operators)
\label{as-model-op}
Let  $Q \in \Psi^*_{\mathrm{v}}(\mathbb{B} \times \mathbb{R},S_n\otimes \mathbb{C}^d)$ with Getzler order $m$ and model operator $Q_{(m)}$ (all 
at the origin).
\begin{itemize}
	\item  If $m-j$ is an odd integer, then
\[
\sigma\left[I_Q(0, t, \gamma)\right]^{(j, 0)}=\mathrm{O}\left(t^{\frac{j-m-a-1}{2}}\right) \quad \text { as } t \rightarrow 0^{+} .
\]
\item 
If $m-j$ is an even integer, then
\[
\sigma\left[I_Q(0, t, \gamma)\right]^{(j, 0)}=t^{\frac{j-m-a}{2}-1} \left[I_{Q_{(m)}}(0,1, \gamma)\right]^{(j, 0)}+\mathrm{O}\left(t^{\frac{j-m-a}{2}}\right) \quad \text { as } t \rightarrow 0^{+} .
\]
\end{itemize}
\end{theorem}

\begin{lemma}(\cite[Lemma 4.13]{Ponge-Wang-2} Composition of operators )
\label{lem:compos}
Let  $Q \in \Psi^*_{\mathrm{v}}(\mathbb{B} \times \mathbb{R},S_n\otimes \mathbb{C}^d)$ with Getzler order $m$ and model operator $Q_{(m)}$. In addition, let $P \colon  C^\infty(\mathbb{B},S_n\otimes \mathbb{C}^d) \to  C^\infty(\mathbb{B},S_n\otimes \mathbb{C}^d)$ be a differential operator with Getzler order $m'$ (independent of time $t$). Then $PQ \in \Psi^*_{\mathrm{v}}(\mathbb{B} \times \mathbb{R},S_n\otimes \mathbb{C}^d)$ and 
\[
\sigma\left[PQ\right] = P_{(m')}Q_{(m)} + \text{lower Getzler order terms}. 
\]
\end{lemma}

\noindent
Let us go back to our cocompact $\Gamma$-proper manifold $M$.
For any $x \in M^\gamma_a$, the tangent bundle decomposes
\[
T_xM \cong T_xM^\gamma_a \oplus N^\gamma_x
\]
Accordingly, the differential form 
\[
\wedge^\bullet T^*_xM \cong \wedge^\bullet \left(T^*_xM^\gamma_a \right) \otimes \wedge^\bullet\left(N^\gamma_x\right) \colon =  \wedge^{k, l}, \quad 1\leq k \leq a, 1 \leq l \leq n-a. 
\]
It is convenient to introduce the following curvature matrices:
\begin{equation}
\label{equ-R}
R^{\prime}:=\left(R_{i j}\right)_{1 \leq i, j \leq a} \quad \text { and } \quad R^{\prime \prime}:=\left(R_{a+i, a+j}\right)_{1 \leq i, j \leq n-a} .	
\end{equation}
Note that the component in $\Lambda^{\bullet, 0}$ of $R^{\prime}$ (resp., $R^{\prime \prime}$ ) is $R^{T M^\gamma}$ (resp., $R^{N^\phi}$ ).

\medskip
\noindent
Let $Q \in \Psi_{\Gamma, \mathrm{v},c}^*\left(M \times \mathbb{R},  E\right)$
with $E$ a bundle of Clifford modules.
We assume, as it is often done in index theory when local computations are involved, 
that $M$ admits a $\Gamma$-invariant spin structure and that $E=\mathcal{S}\otimes V$,
with $V$ a $\Gamma$-equivariant auxiliary complex bundle of rank $d$ and 
with $\mathcal{S}$ denoting the spinor bundle of $M$.
It is possible to give the notion of Getzler order at a point $x\in M$ 
and more generally along a subset of $M$ (see \cite{Ponge-Wang-2}, Definition 4.19 and Remark 4.21).

\begin{theorem}(\cite[Theorem 4.22]{Ponge-Wang-2} Computing the local trace)
\label{thm local trace}
	Let $Q \in \Psi_{\Gamma, \mathrm{v},c}^*\left(M \times \mathbb{R}, E\right)$ have Getzler order $m$ along $ M^\gamma_a$. Let $x_0 \in M^\gamma_a$.  
\begin{enumerate}
	\item If $m$ is an odd integer, then
$$
\operatorname{str}_{E}\left[\gamma^{\mathcal{S} \otimes F} I_Q(x_0, t, \gamma )\right]=\mathrm{O}\left(t^{-\frac{m+1}{2}}\right) \quad \text { as } t \rightarrow 0^{+} .
$$
and this is uniform on compact sets of $M^\gamma$.
 \item  If $m$ is an even integer, then, as $t \rightarrow 0^{+}$, we have

$$
\begin{aligned}
& \left(\operatorname{str}_{E}\right)\left[\gamma^{E}  I_Q(x_0, t, \gamma)\right] \\
& \quad=(-i)^{\frac{n}{2}} t^{-\left(\frac{m}{2}+1\right)} 2^{\frac{a}{2}} \operatorname{det}^{\frac{1}{2}}\left(1-\gamma|_{N^\gamma}\right)\left[\operatorname{tr}_{\mathbb{C}^d}\left[\gamma^d I_{Q_{(m)}}(0,1, \gamma)\right]\right]^{(a, 0)}+\mathrm{O}\left(t^{-\frac{m}{2}}\right)
\end{aligned}
$$
in any  synchronous normal coordinates centered at $x_0$ over which $E=\mathcal{S}\otimes V$ is trivialized
and where $Q_{(m)}$ is the Getzler operator at the origin in that particular coordinate system.
Here $\gamma^d$ is the action of $\gamma$ at the origin on the auxiliary vector space $\mathbb{C}^d$ corresponding to $V$.

\item If $P$ is a differential operator of compact support with constant Getzler order $m$ along 
 an open subset  $W$ of $ M^\gamma_a$ then a similar result holds on $x_0\in W\subset M^\gamma_a$.\\

\end{enumerate}
\end{theorem}

\begin{theorem}[Computation for the model operator]
\label{GR thm}
Let $P \colon C^\infty(M, E) \to C^\infty(M, E)$ be a differential operator whose Getzler order is equal to $m$ along $M^\gamma_a$. Let $x_0 \in M_a^\gamma$. Consider local coordinates and define the Harmonic oscillator
\[
H_R:=-\sum_{1 \leq i \leq n}\left(\partial_i+\sqrt{-1} R_{i j} x^j\right)^2
\]
with $R_{ij}$ defined in (\ref{equ-R}). We have the following
\begin{enumerate}
\item the model operator for $Q(s) = P(sD^2+\partial_t)^{-1}$ at $x_0$ equals
\[
Q(s)_{(m-2)} = P_{(m)}\left(sH_R + \partial_t\right)^{-1}\wedge \exp \left(-t F^V(0)\right),
\]
where $P_{(m)}$ is the model operator of $P$ at $x_0$, and $F^V$ is the curvature of the auxiliary vector bundle $V$.  
\item 
Moreover, 
\[
\gamma I_{\left(sH_R+ \partial_t\right)^{-1}}(x, t, \gamma) =\frac{(4 \pi st)^{-\frac{a}{2}}}{\operatorname{det}^{\frac{1}{2}}\left(1-\gamma|_{N^\gamma}\right)}\operatorname{det}^{\frac{1}{2}}\left(\frac{st R^{\prime} / 2}{\sinh \left(st R^{\prime} / 2\right)}\right) \operatorname{det}^{-\frac{1}{2}}\left(1-\gamma|_{N^\gamma} e^{-st R^{\prime \prime}}\right).
\]
\item The smoothing kernel $K_{Q(s)_{(m-2)}}$ equals:
	\[
	K_{Q(s)_{(m-2)}}(x, y, t)=\left(P_{(m)} G_R\right)(x, y,st) \wedge e^{-stF^V}. 
		\]
	where $G_R$ is the heat kernel for the harmonic oscillator given by Mehler's formula:
\[
G_R(x,y, t):=\frac{1}{(4\pi t)^{n/2}}{\rm det}^{1/2}\left(\frac{tR/2}{\sinh(tR/2)}\right)\exp\left(-\frac{1}{4t}\Theta(x,y,t)\right),
\]
where $\Theta(x,y,t)$ has the following form
\[
\Theta(x,y,t)=\langle \frac{tR/2}{\tanh(tR/2)}x,x\rangle+\langle \frac{tR/2}{\tanh(tR/2)}y,y\rangle-2\langle \frac{tR/2}{\sinh(tR/2)}e^{tR/2} x,y\rangle.
\]
Here the operator $P_{(m)}$ acts only on the factor $\exp\left(-\frac{1}{4t}\Theta(x,y,t)\right)$.
\end{enumerate}
	\begin{proof} 
	The first two claims are proved in \cite[Lemma 4.17]{Ponge-Wang-2}, while (3) follows from \cite[(4.38)]{Ponge-Wang-2}.
\end{proof}
\end{theorem}

\end{document}